\documentclass[a4paper,11pt]{amsart}

\usepackage{latexsym,amsmath,amscd,amssymb,amsthm,epsfig,times}
\usepackage{graphicx}

\usepackage{color}
\definecolor{maus}{rgb}{0.2,0.2,0.2}

\newtheorem{thm}{Theorem}[section]
\newtheorem{prop}[thm]{Proposition}
\newtheorem{lem}[thm]{Lemma}
\newtheorem{cor}[thm]{Corollary}

\newtheorem*{thm*}{Theorem}

\newtheorem*{sympfillthm*}{Theorerm \ref{t:no polygons if filled}}
\newtheorem*{ballapproxthm*}{Theorem \ref{t:tight on balls}}
\newtheorem*{nostarsimplytbthm*}{Theorem \ref{t:no stars imply tb}}

\theoremstyle{definition}
\newtheorem{defn}[thm]{Definition}

\newtheorem{rem}[thm]{Remark}
\newtheorem*{ex*}{Example}

\newcommand{\thmref}[1]{Theorem~\ref{#1}}
\newcommand{\lemref}[1]{Lemma~\ref{#1}}
\newcommand{\propref}[1]{Proposition~\ref{#1}}
\newcommand{\defref}[1]{Definition~\ref{#1}}

\newcommand{\remref}[1]{Remark~\ref{#1}}

\newcommand{\figref}[1]{Figure~\ref{#1}}
\newcommand{\secref}[1]{Section~\ref{#1}}

\newcommand{\lra}{\longrightarrow}

\newcommand{\Alpha}{\mathcal{A}}

\newcommand{\II}{\mathcal{ I}}

\newcommand{\ZZ}{\mathcal{ Z}}
\newcommand{\LL}{\mathcal{ L}}

\newcommand{\R}{\mathbb{ R}}

\newcommand{\Z}{\mathbb{ Z}}
\newcommand{\selfl}{\mathrm{sl}}
\newcommand{\N}{\mathbb{N}}

\newcommand{\tsig}{\widetilde{\Sigma}}

\newcommand{\mlabel}[1]{
                        \label{#1}}

\newcommand{\eing}[1]{\big|_{#1}}
\newcommand{\x}{\times}

\newcommand{\ww}{\wedge}

\newcommand{\eps}{\varepsilon}

\DeclareMathAccent{\ring}{\mathalpha}{operators}{"17}

\author{T. Vogel}

\title{Rigidity versus flexibility of tight confoliations}

\address{Mathematisches~Institut,~Ludwig-Maximilians-Universit\"at M\"unchen, Theresienstr.~39, 80333 M\"unchen, Germany}

\email{tvogel@math.lmu.de}

\date{\today; MSC 2000: 57R17, 57R30}

\begin{document}


\begin{abstract}
In \cite{confol} Y.~Eliashberg and W.~Thurston gave a definition of tight confoliations. We give an example of a tight confoliation $\xi$ on $T^3$ violating the Thurston-Bennequin inequalities. This answers a question from \cite{confol} negatively. Although the tightness of a confoliation does not imply the Thurston-Bennequin inequalities, it is still possible to prove restrictions on homotopy classes of plane fields which contain tight confoliations. 

The failure of the Thurston-Bennequin inequalities for tight confoliations is due to the presence of overtwisted stars. Overtwisted stars are particular configurations of Legendrian curves which bound a disc with finitely many punctures on the boundary. We prove that the Thurston-Bennequin inequalities hold for tight confoliations without overtwisted stars and that symplectically fillable confoliations do not admit overtwisted stars.
\end{abstract}

\maketitle

\tableofcontents


\section{Introduction}\mlabel{s:intro}


In \cite{confol} Eliashberg and Thurston explore the relationship between foliations and contact structures on oriented $3$-manifolds. Foliations respectively contact structures are locally defined by $1$-forms $\alpha$ such that $\alpha\ww d\alpha\equiv 0$ respectively $\alpha\ww d\alpha>0$ (more precisely this defines {\em positive} contact structures).

One of the main results of \cite{confol} is the following remarkable theorem.
\begin{thm}[Theorem 2.4.1 in \cite{confol}]\mlabel{t:El-Th approx}
Suppose that a $C^2$-foliation $\xi$ on a closed oriented $3$-manifold is different from the product foliation of $S^1\times S^2$ by spheres. Then $\xi$ can be $C^0$-approximated by a positive contact structure. 
\end{thm}
In the main part of the proof of this theorem a given foliation on $M$ is modified so that the resulting plane field is somewhere integrable while it is a positive contact structure on other parts of $M$. This motivates the following definition.  
\begin{defn}
A {\em positive confoliation} on $M$ is a $C^2$-smooth plane field on a $3$-manifold $M$ which is locally defined by a $1$-form $\alpha$ such that $\alpha\ww d\alpha\ge 0$. We denote the region where $\xi$ is a contact structure by $H(\xi)$. 
\end{defn}
\thmref{t:El-Th approx} remains true when foliations are replaced by confoliations. Like foliations and contact structures the definition of confoliations can be generalized to higher dimensions (cf. \cite{AlW, confol}) but in this article we are only concerned with dimension $3$. All plane fields appearing in this article will be oriented, in particular these plane fields have an Euler class.  

In the last chapter of \cite{confol} Eliashberg and Thurston discuss several properties of foliations (tautness, absence of Reeb components) and contact structures (symplectic fillability, tightness) and what can be said about a contact structure approximating a taut or Reebless foliation. For example they establish the following theorem.
\begin{thm*}[Eliashberg, Thurston, \cite{confol}]
If a contact structure $\xi$ on a closed $3$-mani\-fold is sufficiently close to a taut foliation in the $C^0$-topology, then $\xi$ is symplectically fillable and therefore tight.
\end{thm*}
Another result in this direction is due to V.~Colin.
\begin{thm*}[Colin, \cite{colin pert}]
A $C^2$-foliation without Reeb components on a closed oriented $3$-manifold can be $C^0$-approximated by tight contact structures.
\end{thm*}
In \cite{cont from fol} J.~Etnyre shows that every contact structure (tight or not) may be obtained by a perturbation of a foliation with Reeb components. This result is implicitly contained in \cite{mori}. Moreover, J.~Etnyre improved \thmref{t:El-Th approx} by showing that $C^k$-smooth foliations can be $C^k$-approximated by contact structures provided that $k\ge 2$ (a written account will hopefully be available in the near future, cf. \cite{etnyreCk}).  

In order to understand better the relationship between geometric properties of foliations and properties of the contact structures approximating them, it is interesting to ask about properties of confoliations which appear in the approximation process. For example the notion of symplectic fillability can be extended to confoliations in an obvious fashion. 

The question how to generalize the notion of tightness is more complicated. One aim of this article is to clarify this point. The following definition is suggested in \cite{confol}. 
\begin{defn} \mlabel{d:tight confol}
A confoliation $\xi$ on $M$ is {\em tight} if for every embedded disc $D\subset M$ such that
\begin{itemize}
\item[(i)] $\partial D$ is tangent to $\xi$,
\item[(ii)] $TD$ and $\xi$ are transverse along $\partial D$
\end{itemize}
there is an embedded disc $D'$ satisfying the following requirements
\begin{itemize}
\item[(1)] $\partial D=\partial D'$,
\item[(2)] $D'$ is everywhere tangent to $\xi$,
\item[(3)] $e(\xi)[D\cup D']=0$.
\end{itemize}
\end{defn}
This definition is motivated by the following facts. If $\xi$ is a contact structure, then there are no surfaces tangent to $\xi$ and \defref{d:tight confol} reduces to a definition of tightness for contact structures. In the case when $\xi$ is a foliation on a closed manifold \defref{d:tight confol} is equivalent to the absence Reeb components by a theorem of Novikov \cite{No}. Thus \defref{d:tight confol} interpolates between tight contact structures and Reebless foliations. The following theorem is also shown in \cite{confol} (we recall the definition of symplectic fillability in \secref{s:tightness of confol}).

\begin{thm}[Theorem 3.5.1. in \cite{confol}] \mlabel{t:fillable confol are tight}
Symplectically fillable confoliations are tight.
\end{thm}

As pointed out in \cite{confol} there are inequalities imposing restrictions on the Euler class $e(\xi)$ of $\xi$ when $\xi$ is a tight contact structure or a Reebless foliation. Before we can state these inequalities we need one more definition. 

\begin{defn} 
Let $\gamma$ be a nullhomologous knot in a confoliated manifold $(M,\xi)$ which is positively transverse to $\xi$. For each choice $F$ of an oriented Seifert surface of $\gamma$ we define the  {\em self linking number} $\selfl(\gamma,F)$ of $\gamma$ as follows. Choose a nowhere vanishing section $X$ of $\xi|_F$ and let $\gamma'$ be the knot obtained by pushing $\gamma$ off itself by $X$. Then 
$$
\selfl(\gamma,F)=\gamma'\cdot F~.
$$
\end{defn}
\noindent Obviously $\selfl(\gamma,F)$ depends only on $[F]\in H_2(M,\gamma;\mathbb{Z})$. 

In \cite{Be} D.~Bennequin proved an inequality between $\selfl(\gamma)$ of a transverse knot in the standard contact structure  $\ker(dz+x\,dy)$ on $\R^3$ and the Euler number of a Seifert surface of $\gamma$. This inequality was extended to all tight contact structures by Eliashberg in \cite{El}. From Thurston's work in \cite{Th} it follows that the same inequalities hold for surfaces in foliated manifolds without Reeb components. We summarize these results as follows.

\begin{thm}[Eliashberg \cite{El}, Thurston \cite{Th}] \mlabel{t:TB Ungl tight}
Let $\xi$ be a tight contact structure or a foliation without Reeb components on a closed manifold $M$ (different from a foliation by spheres) and $F\subset M$ an embedded oriented surface.
\begin{itemize}
\item[a)] If $F\simeq S^2$, then $e(\xi)[F]=0$. 
\item[b)] If $\partial F=\emptyset$ and $F\not\simeq S^2$, then $|e(\xi)[F]|\le-\chi(F)$.
\item[c)] If $\partial F\neq\emptyset$ is positively transverse to $\xi$, then $\selfl(\gamma,[F])\le -\chi(F)$.
\end{itemize}
\end{thm}
The inequalities stated in this theorem are usually referred to as Thurston-Benne\-quin inequalities. They imply that only finitely many classes in $H^2(M;\Z)$ are Euler classes of tight contact structures or foliations without Reeb components. Foliations by spheres violate a) and we exclude such foliations from our discussion. 

It was conjectured (Conjecture 3.4.5 in \cite{confol}) that tight confoliations satisfy the Thurston-Bennequin inequalities. We give a counterexample $(T^3,\xi_T)$ with the property that $e(\xi)[T_0]=-4$ for an embedded torus in $T^3$. Therefore every contact structure which is close to $\xi_t$ must be overtwisted. This yields a negative answer to Question 1 on p.~63 of \cite{confol}. The construction of $(T^3,\xi_T)$ is based on the classification of tight contact structures on $T^2\times[0,1]$ due to E.~Giroux and K.~Honda.

In this article we show that a) is true for tight confoliations and c) holds when $F$ is a disc. On the other hand we give an example of a tight confoliation $\xi_{T}$ on $T^3$ which violates b) and c) for surfaces which are not simply connected. 

Our example indicates that tight confoliations are much more flexible objects than tight contact structures or foliations without Reeb components. For example infinitely many elements of $H^2(T^3;\Z)$ are Euler classes of tight confoliations. Nevertheless, tight confoliations have some rigidity properties. In addition to the Thurston-Bennequin inequalities for simply connected surfaces we show the following theorem.

\begin{ballapproxthm*} 
Let $M$ be a manifold carrying a tight confoliation $\xi$ and $B\subset M$ a closed embedded ball in $M$. There is a neighbourhood of $\xi$ in the space of plane fields with the $C^0$-topology such that $\xi'\eing{B}$ is tight for every contact structure $\xi'$ in this neighbourhood of $\xi$.
\end{ballapproxthm*}

This theorem leads to restrictions on the homotopy class of plane fields which contain tight confoliations. For example only one homotopy class of plane fields on $S^3$ contains a tight confoliation by Eliashberg's classification of tight contact structures on balls together with \thmref{t:tight on balls}. For the proof of \thmref{t:tight on balls} we study the characteristic foliation $S(\xi)=TS\cap\xi$ on embedded spheres $S\subset M$ (we generalize the notion of taming functions introduced in \cite{El} to confoliations and use results from \cite{giroux2}). 

Motivated by the example $(T^3,\xi_T)$ we define the notion of an overtwisted star. Roughly speaking, an overtwisted star on an embedded surface $F$ is a domain in $F$ whose interior is homeomorphic to a disc, the boundary of this domain consists of Legendrian curves and all singularities on the boundary have the same sign. The main difference between overtwisted stars and overtwisted discs is that the set theoretic boundary of an overtwisted star may contain closed leaves or quasi-minimal sets of the characteristic foliation. 

An example of an overtwisted star is shown in \figref{b:starfish} on p.~\pageref{b:starfish}. It will be clear from the definition of overtwisted stars that contact structures which admit overtwisted stars are not tight, ie. they are overtwisted in the usual sense. Following Eliashberg's strategy from \cite{El} we prove the following theorem.

\begin{nostarsimplytbthm*}
Let $(M,\xi)$ be an oriented tight confoliation such that no compact embedded oriented surface contains an overtwisted star and $(M,\xi)$ is not a foliation by spheres. 

Every embedded surface $F$ whose boundary is either empty or positively transverse to $\xi$ satisfies the following relations. 
\begin{itemize}
\item[a)] If $F\simeq S^2$, then $e(\xi)[F]=0$. 
\item[b)] If $\partial F=\emptyset$ and $F\not\simeq S^2$, then $|e(\xi)[F]|\le-\chi(F)$.
\item[c)] If $\partial F\neq\emptyset$ is positively transverse to $\xi$, then $\selfl(\gamma,[F])\le -\chi(F)$.
\end{itemize}
\end{nostarsimplytbthm*}

Moreover, \thmref{t:fillable confol are tight} can be refined as follows.
\begin{sympfillthm*}
Symplectically fillable confoliations do not admit overtwisted stars.
\end{sympfillthm*}
These results indicate that tightness in the sense of \defref{d:tight confol} together with the absence of overtwisted stars is the right generalization of tightness to confoliations.

This article is organized as follows: In \secref{s:facts} we recall several facts about confoliations and characteristic foliations. \secref{s:manipulation} contains a discussion of several methods for the manipulation of characteristic foliation on embedded surfaces. For example we generalize the elimination lemma to confoliations and we discuss several surgeries of surfaces when integral discs of $\xi$ intersect the surface in a cycle. In \secref{s:example} we describe an example of a tight confoliation on $T^3$ which violates the Thurston-Bennequin inequalities while we prove \thmref{t:tight on balls} in \secref{s:rigid}.

In \secref{s:discussion} we discuss overtwisted stars and establish the Thurston-Bennequin inequalities for tight confoliations without overtwisted stars. Moreover, we prove that symplectically fillable confoliations do not admit overtwisted stars.

Throughout this article $M$ will be a connected oriented $3$-manifold without boundary and $\xi$ will always denote a smooth oriented plane field on $M$. Moreover, we require $M$ to be compact.

{\em Acknowledgements:} The author started working on this project in the fall of 2006 during a stay at Stanford University, the financial support provided by the "Deutsche Forschungsgemeinschaft" is gratefully acknowledged. It is a pleasure for me to thank Y. Eliashberg for his support, hospitality and interest.  Moreover, I would like to thank V.~Colin and J.~Etnyre for helpful conversations.


\section{Characteristic foliations, non-integrability and tightness } \mlabel{s:facts}



In this section we recall some definitions, notations and well known facts which will be used throughout this paper. Most notions discussed here are generalizations of definitions which are well-known in the context of contact structures (cf. for example \cite{Aeb}, \cite{intro-etnyre}, \cite{giroux} and the references therein). 

\subsection{Characteristic foliations on surfaces}

We consider an embedded oriented surface $F$ in a confoliated $3$-manifold $(M,\xi)$ and we assume that $\xi$ is cooriented. The singular foliation $F(\xi):=\xi\cap TF$ is called the {\em characteristic foliation} of $F$. The leaves of the characteristic foliation are examples of {\em Legendrian curves}, ie. curves tangent to $\xi$. 

The following convention is used to orient $F(\xi)$: Consider $p\in F$ such that $F(\xi)_p$ is one-dimensional. For $X\in F(\xi)(p)$ we choose $Y\in\xi(p)$ and $Z\in T_pF$ such that $X,Y$ represents the orientation of $\xi(p)$ and $X,Z$ induces the orientation of the surface. Then $X$ represents the orientation of the characteristic foliation if and only if $X,Y,Z$ is a positive basis of $T_pM$.

With this convention, the characteristic foliation points out $F$ along boundary components of $F$ which are positively transverse to $\xi$. An isolated singularity of $F(\xi)$ is called {\em elliptic} respectively {\em hyperbolic} when its index is $+1$ respectively $-1$. A singularity is {\em positive} if the orientation of $\xi$ coincides with the orientation of $F$ at the singular point and {\em negative} otherwise. Given an embedded surface $F\subset M$ we denote the number of positive/negative elliptic singularities by $e_\pm(F)$ and the number of positive/negative hyperbolic singularities is $h_\pm(F)$. 


\subsection{(Non-)Integrability} \mlabel{ss:confolmeaning}

The condition that $\xi$ is a confoliation can be interpreted in geometric terms. The following interpretation can be found in \cite{confol}.  

Let $D$ be a closed disc of dimension $2$ and $\xi$ a positive confoliation transverse to the fibers of $\pi : D\times\R\lra D$. Then $\xi$ can be viewed as a connection. We assume in the following that this connection is complete, ie. for every differentiable curve $\sigma$ in $D$ there is a horizontal lift of $\sigma$ starting at a given point in the fiber over the starting point of $\sigma$.

We consider the holonomy of the characteristic foliation on $\pi^{-1}(\partial D)$ 
\begin{equation} \label{e:holonomy}
h_{\partial D} : \pi^{-1}(p)\simeq\R\lra\R\simeq\pi^{-1}(p)
\end{equation}
where $h_{\partial D}(x)$ is defined as the parallel transport of $x\in\R$ along $\partial D$.
\begin{lem}[Lemma 1.3.4. in \cite{confol}] \mlabel{l:neg-curv}
If the confoliation $\xi$ on $\pi : D\times\R\lra D$ defines a complete connection, then $h_{\partial D}(x)\le x$ for all $x\in\pi^{-1}(p)$ and $p\in\partial D$. Equality holds for all $x\in\pi^{-1}(p)$ if and only if $\xi$ is integrable. 

If $D=D\times\{0\}$ is tangent to $\xi$, then the germ of the holonomy is well defined without any completeness assumption and $h_{\partial D}(x)\le x$ for all $x$ in the domain of $h$. The germ of $h_{\partial D}$ coincides with the germ of the identity if and only if a neighbourhood of $D$ is foliated by discs. 
\end{lem}
Of course, the second part of the lemma applies to the case when on considers only the part lying above or below $D\times\{0\}\subset D\times\R$. A consequence of \lemref{l:neg-curv} is the following generalization of the Reeb stability theorem to confoliations.
\begin{thm}[Proposition 1.3.9. in \cite{confol}] \mlabel{t:Reeb stability}
Let $M$ be a closed oriented manifold carrying a positive confoliation $\xi$. Suppose that $S$ is an embedded sphere tangent to $\xi$. Then $(M,\xi)$ is diffeomorphic to the product foliation on $S^2\times S^1$ by spheres.
\end{thm}
Foliations by spheres appear as exceptional case in some theorems. They will therefore be excluded from the discussion.

Another useful geometric interpretation of the confoliation condition can be found on p.~4 in \cite{confol} (and many other sources): Let $X$ be a Legendrian vector field and $F$ a surface transverse to $X$. The slope of line field $F_t(\xi)$ on the image of $F$ under the time-$t$-flow of $X$ is monotone in $t$ if and only if $\xi$ is a confoliation. This interpretation is useful when one wants extends confoliations along flow line which are Legendrian where the confoliation is already defined. 

We define the {\em fully foliated part} of a confoliation $\xi$ on $M$ as the complement of 
$$
\{x\in M | \textrm{ there is a Legendrian curve connecting }x\textrm{ to }H(\xi) \}.
$$
If $\gamma$ is a Legendrian curve in a leaf of $\xi$ and $A\simeq\gamma\times(-\delta,\delta), \delta>0$ an annulus transverse to the leaf such that $\gamma=\gamma\times\{0\}$, then we will consider several types of holonomy $h_A$ of the characteristic foliation on $A$. 
\begin{itemize}
\item We say that there is {\em linear holonomy} or {\em non-trivial infinitesimal holonomy} along $\gamma$ if $h'_A(0)\neq 0$. 
\item The holonomy is {\em sometimes attractive} if there are sequences $(x_n),(y_n)$ which converge to zero such that $x_n>0>y_n$ and   
\begin{align*} 
h_A(x_n) <x_n,   h_A(y_n) >y_n   \textrm{ for all } n\in\N.
\end{align*}
\end{itemize}

\subsection{Tightness of confoliations} \mlabel{s:tightness of confol}
In this section we summarize several facts about tight confoliations. We shall always assume that $\xi$ is a tight confoliation but it is not a foliation by spheres.

If $(M,\xi)$ is tight and $D\subset M$ is an embedded disc such that $\partial D$ is tangent to $\xi$ and $\xi\eing{\partial D}$ is transverse to $TD$, then the disc $D'$ whose existence is guaranteed by \defref{d:tight confol} is uniquely determined. Otherwise there would be a sphere tangent to $\xi$ and by \thmref{t:Reeb stability} $\xi$ would be a foliation by spheres. But we explicitly excluded this case. 

The definition of tightness refers to smoothly embedded discs but of course it has implications for discs with piecewise smooth boundary and slightly more generally for unions of discs.

\begin{lem} \mlabel{l:limit cycles on spheres imply integral disc}
Suppose that $(M,\xi)$ is a tight confoliation and $S\subset M$ is an embedded sphere such that the characteristic foliation $S(\xi)=TS\cap\xi$ has only non-degenerate hyperbolic singularities along a connected cycle $\gamma$ of $S(\xi)$.  Then there are immersed discs $D_i', i=1,\ldots k$ in $M$ which are tangent to $\xi$ and 
$$ 
\partial\left(\bigcup_{i=1}^k D_i\right)=\partial D.
$$
\end{lem}
This follows by considering $C^\infty$-small perturbations of $S$ such that $\gamma$ is approximated by closed leaves of the characteristic foliation of the perturbed sphere. We will continue to say that a disc bounds the cycle $\gamma$ although the ``disc'' might have corners or be a pinched annulus, for example. 

The most important criterion to prove tightness is \thmref{t:fillable confol are tight}. It is based on the following definition.
\begin{defn} \mlabel{d:symp filling}
A positive confoliation $\xi$ on a closed oriented manifold $M$ is {\em symplectically fillable} if there is a compact symplectic manifold $(X,\omega)$ such that 
\begin{itemize}
\item[(i)] $\omega\big|_\xi$ is non-degenerate and
\item[(ii)] $\partial X=M$ as oriented manifolds where $X$ is oriented by $\omega\ww\omega$.
\end{itemize}
\end{defn}
In this definition we use the ``outward normal first'' convention for the orientation of the boundary. There are several different notions of symplectic fillings and the \defref{d:symp filling} is often referred to as weak symplectic filling. It is clear from \thmref{t:fillable confol are tight} (and \thmref{t:no polygons if filled}) that the existence of a symplectic filling is an important property of a confoliation. 

Note that if $(M,\xi)$ is symplectically fillable, then the same is true for confoliations $\xi'$ which are sufficiently close to $\xi$ in the $C^0$-topology. 

\thmref{t:fillable confol are tight} can sometimes be extended to non-compact manifolds. Then one obtains the following consequence. 
\begin{prop}[Proposition 3.5.6. in \cite{confol}] \mlabel{p:complete connection} 
If a confoliation $\xi$ is transverse to the fibers of the projection $\R^3\lra \R^2$ and if the induced connection is complete, then $\xi$ is tight. 
\end{prop}
In \cite{confol} one can find an example which shows that the completeness condition can not be dropped.


\section{Properties and modifications of characteristic foliations} \mlabel{s:manipulation}


The characteristic foliations on embedded surfaces in manifolds with contact structures has several properties reflecting the positivity of the contact structure. Moreover, there are methods to manipulate the characteristic foliation by isotopies of the surface. Similar remarks apply when $\xi$ is a foliation. In this section we generalize this to the case when $\xi$ is a confoliation. If $\xi$ is tight, then there are more restrictions on characteristic foliation. Some of these additional restrictions shall be discussed in \secref{s:rigid}.      

\subsection{Neighbourhoods of elliptic singularities}

With our orientation convention positive elliptic singular points lying in the contact region are sources. The following lemma shows that this statement can be interpreted such that it generalizes to confoliation.

\begin{lem} \mlabel{l:little discs}
Let $(M,\xi)$ be a confoliated manifold and $F$ an immersed surface whose characteristic foliation has a non-degenerate positive elliptic singularity $p$. 

There is an open disc $p\in D\subset F$ such that each leaf of the characteristic foliation on $D$ is either a circle or there is a closed transversal of $F(\xi)$ through the leaf. If $p$ is positive respectively negative and $\partial D$ is transverse to $F(\xi)$, then $F(\xi)$ points outwards respectively inwards.
\end{lem}

\begin{proof}
We fix a defining form $\alpha$ for $\xi$ on a neighbourhood of $p$. If $d\alpha(p)\neq 0$, then $p$ lies in the interior of the contact region and the claim follows from \cite{giroux}. When $d\alpha(p)=0$, then $F(\xi)$ is transverse to the gradient vector field $R$ of a Morse function which has a critical point of index $0$ or $2$ at $p$. 

In the following we assume that $p$ is positive and $R$ points away from $p$ and coorients $\xi$ away from $p$ (the other cases are similar). The Poincar{\'e} return map characteristic foliation is well defined on a small neighbourhood of $p$ in a fixed radial line starting at the origin (cf. \cite{marsden} for example) and by our orientation convention $F(\xi)$ is oriented clockwise near $p$. We want to show that Poincar{\'e} return map is non-decreasing when the orientation of the radial line points away from $p$. In the following we assume that the Poincar{\'e} return map is not the identity because in that situation our claim is obvious. 

Let $D\subset F$ be a small disc containing $p$ such that $\partial D$ is transverse to $F(\xi)$. Fix a vector field $Z$ coorienting both $F$ and $\xi$. We write $D_z$ for the image of $F$ under the time $z$-flow of $Z$. We may assume that the tangencies of $D_z$ and $\xi$ are exactly the points on the flow line $\gamma_p$ of $Z$ through $p$. 

We extend $R$ to a vector field on a neighbourhood of $p$ tangent to $D_z$ such that it remains transverse to $\xi$ on $U\setminus\gamma_p$. Then the vector field $T=zZ+R$ is transverse to $\xi$ on $\{z\ge 0\}\setminus\{p\}\subset U$. The flow of $T$ exists for all negative times $t$ and every flow line of $T$ approaches $p$ as $t\to-\infty$. Since $d\alpha(p)=0$ there are local coordinates $x,y$ on $D$ around $p$ such that $p$ corresponds to the origin and 
\begin{equation} \label{e:alpha}
\alpha=dz+\left(xdx+ydy\right)+\widetilde{\alpha}
\end{equation}
where $\widetilde{\alpha}$ denotes a $1$-form such that $\widetilde{\alpha}/(x^2+y^2)$ and $\widetilde{\alpha}/z$ remain bounded when one approaches the origin. 

We choose a closed embedded disc $D'$ in $\{z\ge 0\}$ which is transverse to $T$ and $D$ such that $\partial D'=\partial D$ and $D\cup D'$ bound a closed half ball $B$. The half ball is identified with a Euclidean half ball of radius $1$ and we fix spherical coordinates $\rho,\vartheta,\phi$ (where $\rho$ denotes the distance of a point from the origin, $\vartheta$ is the angle between $\gamma_p$ and the straight line connecting the point with the origin) such that $T$ corresponds to $\rho\partial_\rho$. In this coordinate system
\begin{equation} \label{e:pullback alpha}
\alpha=\cos(\vartheta) d\rho+\rho\sin(\vartheta)\left(-d\vartheta+\sin(\vartheta)d\rho+\cos(\vartheta)\rho d\vartheta \right)+\widetilde{\alpha}
\end{equation}
and $\widetilde{\alpha}/(\rho^2\sin^2(\vartheta))$ and $\widetilde{\alpha}/(\rho\cos(\vartheta))$ remain bounded when one approaches the origin. 

Consider a closed disc $D''$ lying in the interior of $D'$. We identify the union of all flow lines of $T$ which intersect $D''$ with $D''\times(0,1]$ such that the second factor corresponds to flow lines of $T$. On $D''\times(0,1]$ the factor $\cos(\vartheta)$ is bounded away from $0$. By \eqref{e:pullback alpha} the plane field $\ker(\alpha)$ extends to a smooth plane field on $D''\times[0,1]$ such that $D''\times\{0\}$ is tangent to the extended plane field. Therefore $\ker(\alpha)$ extends to a continuous plane field on $(D'\times[0,1])\setminus(\partial D'\times\{0\})$ which is a smooth confoliation on $D'\times(0,1]$. 

The holonomy of the characteristic foliation on $\partial D''\times[0,1]$ is non-increasing by \lemref{l:neg-curv} when $\partial D''\times\{0\}$ is oriented as the boundary of $D''$. Our orientation assumptions at the beginning of the proof imply that the characteristic foliation on $\partial D'\times(0,1]$ is oriented in the opposite sense. This implies that the Poincar{\'e}-return map of the characteristic foliation around $p$ is non-decreasing.
\end{proof}


\subsection{Legendrian polygons} \mlabel{ss:legpoly}

In the proof of rigidity theorems for tight confoliations and also in \secref{s:discussion} we well use the notion of basins and Legendrian polygons. In this section we adapt the definitions from \cite{El}.

\begin{defn} \mlabel{d:leg poly}
A {\em Legendrian polygon} $(Q,V,\alpha)$ on a compact embedded surface $F$ is a triple consisting of a connected oriented surface $Q$ with piecewise smooth boundary, a finite set $V\subset \partial Q$ and a differentiable map $\alpha : Q\setminus V \lra F$ which is an orientation preserving embedding on the interior such that
\begin{itemize}
\item[(i)] corners of $Q$ are mapped to singular points of $F(\xi)$,
\item[(ii)] smooth pieces of $\partial Q$ are mapped onto smooth Legendrian curves on $F$,
\item[(iii)] for points $v\in V$ the image $\alpha(b_\pm)$ of the two segments $b_\pm\subset\partial Q\setminus V$ which end at $v$ have the same $\omega$-limit set $\Gamma_v$ and $\Gamma_v$ is not a singular point.
\end{itemize}
A {\em pseudovertex} is a point $x\in\partial Q$ such that $\alpha(x)$ is a hyperbolic singularity and $\alpha|_{\partial Q}$ is smooth at $\alpha(x)$. 
\end{defn}
A hyperbolic singularity $\alpha(x)$ on $\alpha(\partial Q)$ can be a pseudovertex only if both unstable or both unstable leaves are contained in $\alpha(\partial Q)$.

The points in $V$ should be thought of as missing vertices in the boundary of $Q$. \figref{b:legpoly} shows the image $\alpha(Q)$ of a Legendrian polygon $(Q,V,\alpha)$ where $Q$ is a disc, $V=\{v\}\subset\partial Q$ and the corresponding ends of $\partial Q\setminus\{v\}$ are mapped to leaves of the characteristic foliation whose $\omega$-limit set is the closed leaf $\gamma_v$. There are three pseudovertices.
\begin{figure}[htb] 
\begin{center}
\includegraphics{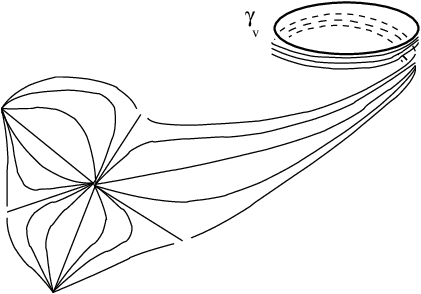}
\end{center}
\caption{\label{b:legpoly}}
\end{figure}
The following definition generalizes the notion of injectivity of a Legendrian polygon to the context of confoliations.
\begin{defn} \mlabel{d:identify}
A Legendrian polygon $(Q,V,\alpha)$ {\em identifies edges} if there are edges $e_1,\ldots,e_l, l\ge 2$ in $\partial Q$ such that $\alpha(e_1)\cup\ldots\cup\alpha(e_l)$ is a cycle containing the image of the pseudovertices lying $e_1,\ldots,e_l$ and leaves of the characteristic foliation such that
\begin{itemize}
\item[(i)] the preimage of each point of the cycle $\gamma_{e_1\ldots e_l}$ except the image of pseudovertices has exactly one element while
\item[(ii)] the preimage of points on the segments and of the images of the pseudovertices consists of exactly two elements.
\end{itemize} 
A Legendrian polygon which does not identify edges is called {\em injective}. 
\end{defn}
Notice that $\alpha$ may identify vertices even if $(Q,V,\alpha)$ is injective. An example of a Legendrian polygon which identifies three edges such that $\gamma_{e_1e_2e_3}$ is not trivial is shown \figref{b:identify}.
\begin{figure}[htb] 
\begin{center}
\includegraphics{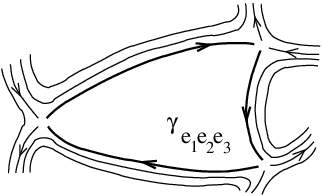}
\end{center}
\caption{\label{b:identify}}
\end{figure}

Because $F$ is compact and the singularities of $F(\xi)$ are isolated the limit sets of individual leaves of the characteristic foliation on $F$ belong to one and only one of the following classes (cf. Theorem 2.6.1. of \cite{flows})
\begin{itemize}
\item fixed points, 
\item closed leaves, 
\item cycles consisting of singular points and leaves connecting them and 
\item quasi-minimal sets, ie. closures of non-periodic recurrent trajectories. 
\end{itemize}
At this point we use the smoothness of $\xi$ (smoothness of class $C^2$ would suffice).

\begin{lem} \mlabel{l:leg poly}
Let $F\subset M$ be a surface and $\xi$ a confoliation on $M$ such that $\partial F$ is transverse to $\xi$ and the characteristic foliation points inwards along $\partial F$. Assume that $U\subset F$ is a submanifold of dimension $2$ such that every boundary component is either is tangent to $F(\xi)$ or transverse to $\xi$ and the characteristic foliation points outwards.

Let $B(U)$ be the union of all leaves of $F(\xi)$  which intersect $U$. Then $\overline{B(U)}$ has the structure of a Legendrian polygon.
\end{lem}

\begin{proof}
A preliminary candidate for $(Q,V,\alpha)$ is $Q_0:=U, V_0=\emptyset$ and $\alpha$ the inclusion of $Q_0$. We will define vertices and edges of $Q$ and we will glue $1$-handles to components of $\partial Q_0$. The existence of $\alpha$ will be immediate once the correct polygon with all pseudovertices, corners and elliptic singularities and $V$ are defined.

Each intersection of $\partial U$ with a stable leaf of a hyperbolic singularity of $F(\xi)$ defines a vertex of $Q_0$. We obtain a subset $P_0\subset\partial Q_0$ which will serve as a first approximation for the set of pseudovertices. For $p\in P_0$ we denote the corresponding hyperbolic singularity of $F(\xi)$ by $\alpha(p)$.

First we consider the boundary components $\Gamma$ of $Q_0$ which are transverse to $F(\xi)$ and $\Gamma\cap P_0=\emptyset$. All leaves of $F(\xi)$ passing through $\Gamma$ have the same $\omega$-limit set $\Omega(\Gamma)$ (cf. Proposition 14.1.4 in \cite{katok}). 
 
We claim that $\Omega(\Gamma)$ is an elliptic singularity or a cycle: Assume that $\Omega(\Gamma)$ is quasi-minimal. According to Theorem 2.3.3 in \cite{flows} there is a recurrent leaf $\gamma$ which is dense in $\Omega(\Gamma)$. There is a short transversal $\tau$ of $F(\xi)$ such that $|\gamma\cap\tau|\ge 2$ and there are leaves of $F(\xi)$ passing through $\Gamma$ which intersect $\tau$ between two points $p_1,p_2$ of $\gamma\cap\tau$. Because $\gamma$ is recurrent it cannot intersect $\Gamma$. Let $I\subset\tau$ be the maximal open segment lying between $p_1,p_2$ such that the leaves of $F(\xi)$ induce a map from $I$ to $\Gamma$. It follows (as in Proposition 14.1.4. in \cite{katok}) that the boundary points of $I$ connect to singular points of $F(\xi)$ which have to be hyperbolic by our assumptions. These hyperbolic singularities are part of a path tangent to $F(\xi)$ which connects $\Gamma$ with $\tau$ and this path passes only through hyperbolic singularities. This is a contradiction to our assumption $\Gamma\cap P_0=\emptyset$.

Thus if $P_0\cap\Gamma=\emptyset$, then there are two cases depending on the nature of $\Omega(\Gamma)$.
\begin{itemize}
\item If $\Omega(\Gamma)$ is an elliptic singularity respectively a closed leaf of $F(\xi)$, then we place no vertices on $\Gamma$ and $\alpha$ maps $\Gamma$ to the elliptic point respectively the closed leaf while $\alpha=\alpha_1$ outside a collar of $\Gamma$. 
\item If $\Omega(\Gamma)$ is a cycle containing hyperbolic points, then we place a corner on $\Gamma$ for each time the cycle passes through a hyperbolic singularity. The map $\alpha|_{\Gamma}$ is defined accordingly.  
\end{itemize}

Next we consider a boundary component $\Gamma$ of $Q_0$ which is transverse to $F(\xi)$ and contains an element $p$ of $P_0\cap\Gamma$. Let $\eta$ be an unstable leaf of the corresponding hyperbolic singularity $\alpha(p)$ of $F(\xi)$ and $\Omega(\eta)$ the $\omega$-limit set of $\eta$. Depending on the type of $\Omega(\eta)$ we distinguish four cases. 
\begin{itemize}
\item[(i)] $\Omega(\eta)$ is an elliptic singular point. Then we place an elliptic singularity on $\Gamma$ next to the pseudovertex. 
\item[(ii)] $\Omega(\eta)$ is a cycle of $F(\xi)$ or a quasi-minimal set. Then we place a point $v$ on $\Gamma$ and add this vertex to to the set of virtual vertices $V_0$. 
\item[(iii)] $\Omega(\eta)$ is a hyperbolic point and $\alpha(p)$ is part of a cycle. Some possible configurations in this case are shown in \figref{b:teardrop} (except the top right part). More precisely, the configurations in \figref{b:teardrop} correspond to the case when there are are at most two different hyperbolic singularities of $F(\xi)$ which are connected. This assumption is satisfied for surfaces in a generic $1$-parameter family of embeddings and it would suffice for our applications. 

In the present situation we add a $1$-handle to $Q_0$ along $\Gamma$. This defines a new polygon $Q_1$. We define $\alpha_1: Q_1\lra F$ such that one of two new boundary components is mapped to the cycle containing $\alpha(p)$ and we place a corner on this connected component of $\partial Q_1$ for each time the cycle passes trough a hyperbolic singularity. In particular $p$ is no longer a pseudovertex. Outside a collar of $\Gamma$ we require $\alpha=\alpha_1$.   

\item[(iv)] $\Omega(\eta)$ is a hyperbolic singularity and $\alpha(p)$ is not part of a cycle. Then we place a corner on $\Gamma$ which corresponds to $\Omega(\eta)$. We continue with the unstable leaf $\eta'\subset\overline{B_\omega(\Gamma)}$ of $\Omega(\eta)$ and place corners or vertices on $\Gamma$ depending on the nature of the $\omega$-limit set of $\eta'$. One possible configuration is shown in the top right part of \figref{b:teardrop}. 
\end{itemize}

\begin{figure}[htb]
\begin{center}
\includegraphics{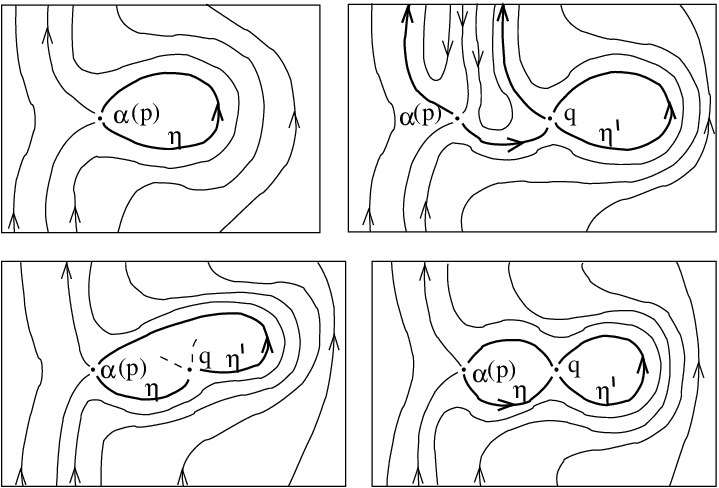}
\end{center}
\caption{\label{b:teardrop}}
\end{figure}
All unstable leaves of hyperbolic singularities in $F(\xi)$ which correspond to elements of $P_0\cap\Gamma$ can be treated in this way. 

We iterate the procedure (starting from the choice of pseudovertices) until no new $1$-handles are added and we have treated all occurring boundary components. This process is finite because each hyperbolic singularity can induce the addition of at most one $1$-handle and there are only finitely many hyperbolic singularities on $F$. In the end we obtain a polygon $Q$. The existence of a finite set $V\subset \partial Q$ and the immersion $\alpha: Q\setminus V\lra F$ with the desired properties follows from the construction.
\end{proof}


\subsection{The elimination lemma} 

There are several possibilities to manipulate the characteristic foliation on an embedded surface. Of course one can always perturb the embedding of the surface so that it becomes generic and that the singularities lie in the interior of the contact region $H(\xi)$ or in the interior of its complement. In addition to such perturbations we shall use two other methods.

The first method discussed in this section is called elimination of singularities and it is well known in the context of contact structures. The second method will be described in \secref{s:cutting}. 

By a $C^0$-small isotopy of the surface $F$ one can remove a hyperbolic and an elliptic singularity which are connected by a leaf $\gamma$ of $F(\xi)$ if the signs of the singularities agree. The characteristic foliation before the isotopy is depicted in \figref{b:vor-elim}. The segment $\gamma$ corresponds to the thickened segment in the middle of \figref{b:vor-elim}.

\begin{figure}[htb] 
\begin{center}
\includegraphics{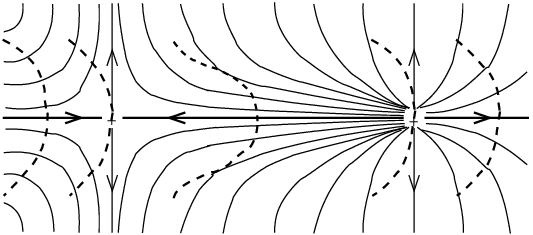}
\end{center}
\caption{\label{b:vor-elim}}
\end{figure}
After the elimination of a pair of singularities as in \lemref{l:elim} the characteristic foliation on a neighbourhood of $\gamma$ looks like in \figref{b:nach-elim}. 
\begin{figure}[htb] 
\begin{center}
\includegraphics{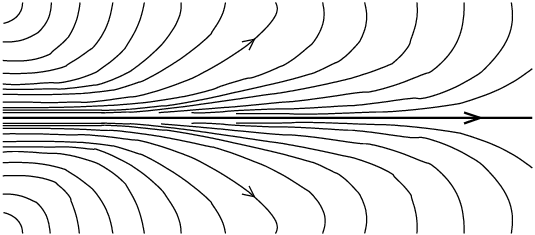}
\end{center}
\caption{\label{b:nach-elim}}
\end{figure}
The elimination of singularities plays an important role in Eliashberg's proof of \thmref{t:TB Ungl tight} for tight contact structures.

Below we give a proof of the elimination lemma which applies to confoliations under a condition on the location of the singularities. Usually the elimination lemma is proved using Gray's theorem but this theorem is not available in the current setting (this is explained in \cite{Aeb} for example). 

\begin{lem} \mlabel{l:elim}
Let $F$ be a surface in a confoliated manifold $(M,\xi)$. Assume that the characteristic foliation on $F$ has one hyperbolic singularity and one elliptic singularity of the same sign which are connected by a leaf $\gamma$ of the characteristic foliation. 

If the elliptic singularity lies in $H(\xi)$, then then there is a $C^0$-small isotopy of $F$ with support in a small open neighborhood $U$ of $\gamma$ such that the new characteristic foliation has no singularities inside of $U$. The isotopy can be chosen such that $\gamma$ is contained in the isotoped surface.
\end{lem}

Note that if $\xi$ is a foliation, then the situation of the lemma cannot arise since all leaves of the characteristic foliations in a neighbourhood of an elliptic singularity are closed. 

\begin{proof}[Proof of \lemref{l:elim}]
We assume that both singularities are positive. There is a neighbourhood $U$ of $\gamma$ with coordinates $x,y,z$ such that $\xi\eing{U}$ is defined by the $1$-form $\alpha=dz+a(x,y,z)dy$ such that the function $a$ satisfies $\partial_x a\ge0$. We assume that $\partial_z$ is positively transverse to $\xi$ and $F$, $\{z=0\}\subset F$ and the $x-$axis of the coordinate system contains $\gamma$.  

It follows that $\xi\eing{U'}$ can be extended to a confoliation $\xi_c$ on $\R^3$ which satisfies the assumptions of \lemref{l:neg-curv} if $U'\subset U$ is a ball and $\partial_x$ is tangent to $\partial U'$ along a circle. Since every step in the proof will take place in a fixed small neighbourhood of $\gamma$ we can apply \lemref{l:neg-curv} without any restriction. We choose $\eps>0$ so that $x\subset(-\eps,\eps)\subset U'$ for all $x$ in a neighbourhood $V\subset U'$ of $\gamma$. For a path $\sigma\subset V$ we will consider the hypersurface $T_\sigma=\sigma\times(-\eps,\eps)$. By our choices $T_\sigma(\xi)$ is transverse to the second factor of $T_\sigma$. 

Choose a smooth foliation $\II$ of a small neighbourhood (contained in $U$) of $\gamma$ in $F$ by intervals $I_s, s\in[-1,1]$ as indicated by the dashed lines in \figref{b:vor-elim}. We choose $\II$ such that it has the following properties.
\begin{itemize}
\item[(i)] Two intervals $I_{s_0},I_{s_1}$ pass through the singularities. 
One of them is tangent to the closure of the unstable separatrices of the hyperbolic singularity.
\item[(ii)] All intervals intersecting the interior of $\gamma$ have exactly two tangencies with the characteristic foliation on $F$. The intervals which do not intersect the closure of $\gamma$ are transverse to the characteristic foliation.
\item[(iii)] Let $\sigma$ by a path in $F$ which is shorter than $\delta$ with respect to a fixed auxiliary Riemannian metric. If $\delta>0$ is small enough, then the image of $(\sigma(0),0 )$ under the holonomy along $T_\sigma$ is defined. We assume that the length of each $I_s$ is smaller than $\delta$.  
\end{itemize}

We parameterize the leaf $I_s$ by $\sigma_s: [0,1] \lra F$ such that the intersection of $\gamma$ with $I_s$ is positive (or empty), ie. in \figref{b:vor-elim} the leaves of $\II$ are oriented towards the upper part of the picture. 

The following figures show neighbourhoods of $I_s$ in $T_s:=T_{\sigma_s}$ for certain $s\in[-1,1]$. In each of these figures the dotted line represents $I_s$, oriented from left to right. \figref{b:vorcrittrans} corresponds to a leaf $I_s$ which does not intersect $\gamma$. Then $I_s$ is nowhere tangent to the characteristic foliation on $T_s$. By our orientation conventions and the choice of $\II$ the slope of $\xi\cap T_s$ is negative along $I_s$.

\begin{figure}[htb] 
\begin{center}
\includegraphics{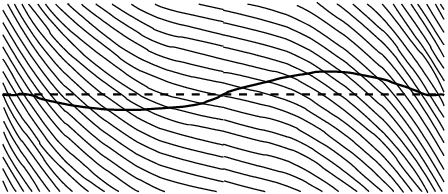}
\end{center}
\caption{\label{b:vorcrittrans}}
\end{figure}

The leaves $I_{s_0},I_{s_1}$ contain the singular points of the characteristic foliation on $F$. As shown in \figref{b:critpointtrans} there is exactly one tangency of $F$ and the characteristic foliation on $T_{s_0},T_{s_1}$. The slope of the characteristic foliation on $T_{s_0},T_{s_1}$ is negative along $I_{s_0},I_{s_1}$ except at the point of tangency.

\begin{figure}[htb] 
\begin{center}
\includegraphics{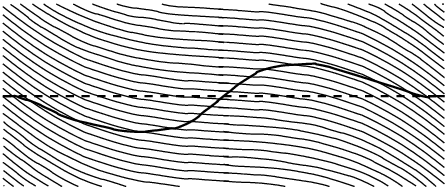}
\end{center}
\caption{\label{b:critpointtrans}}
\end{figure}

Finally, the leaves $I_s,s\in[s_0,s_1]$ intersect the interior of $\gamma$ and $I_s$ is tangent to $F(\xi)$ in exactly two points. This is shown in \figref{b:middletrans}. Between the two points of tangency, the slope of the characteristic foliation on $T_s$ is positive along $I_s$, it is zero at the tangencies and negative at the remaining points of $I_s$.

\begin{figure}[htb] 
\begin{center}
\includegraphics{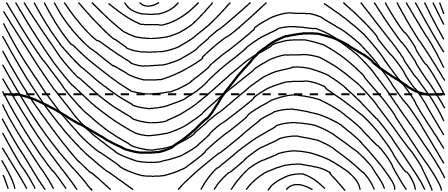}
\end{center}
\caption{\label{b:middletrans}}
\end{figure}
We want to find a smooth family of isotopies of the intervals $I_s$ within $T_s$ such that 
\begin{itemize}
\item[(i)] for all $s$ the isotopy is constant near the endpoints of $I_s$ and
\item[(ii)] after the isotopy, the intervals $I_s$ are transverse to the characteristic foliation on $T_s$. 
\end{itemize}

This will produce the desired isotopy of $F$. Such a family of isotopies exists if and only if the following condition (s) is satisfied for all $s\in[-1,1]$:

\smallskip

{\bf Condition (s):}   The image of $\sigma_s(0)\times\{0\}$ under the holonomy along $\sigma_s$ lies {\em below} the other endpoint $\sigma_s(1)\times\{0\}$ of $I_s$ or  the leaf of $T_s(\xi)$ which passes through $(\sigma_s(0),0)$ exits $T_s$ through $(\sigma_s,-\eps)\subset \partial T_s$.

\smallskip

Note that this condition is automatically satisfied for $s\in[-1,1]$ if $I_s$ does not intersect $\gamma$ or this intersection point is close enough to a singularity of the characteristic foliation.

If (s) is not satisfied for all $s$, then we will replace $\II$ by another foliation $\II'$ by intervals  $I_s'$ (the corresponding embeddings of intervals are denoted by $\sigma_s'$) as follows:

\begin{itemize}
\item[(i)] If $I_s$ does not intersect $\gamma$, then $\sigma_s=\sigma'_s$. $I_s'$ intersects $\gamma$ if and only if $I_s$ does.
\item[(ii)] $I_s'$ is tangent to the characteristic foliation on $F$ along two closed intervals (which may be empty or points). The complement of these two intervals is the union of three intervals such that each of these intervals is mapped to a curve of length $\le\delta$.
\item[(iii)] $I_s$ and $I'_s$ coincide on those intervals where the characteristic foliation on $T_s$ has negative slope for all $s\in[-1,1]$.
\item[(iv)] $\overline{I}_s\cup I'_s$ bounds a positively oriented disc (here $\overline{I}_s$ denotes the interval $I_s$ with the opposite orientation). 
\end{itemize}

In \figref{b:vor-elim2} the dashed line corresponds to $I_s'$ while the thick solid line represents $I_s$.

\begin{figure}[htb] 
\begin{center}
\includegraphics{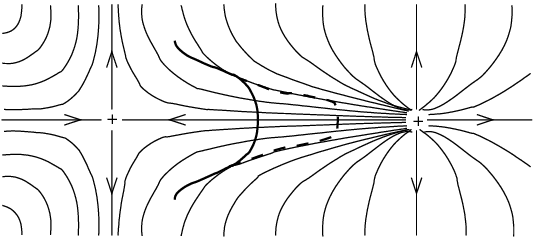}
\end{center}
\caption{\label{b:vor-elim2}}
\end{figure}

For $s\in(s_0,s_1)$ we define a curve $I''_s$ by replacing the segment of $I_s$ lying between the tangencies with $F(\xi)$ by two segments of leaves of $F(\xi)$ whose $\Alpha$-limit set is the elliptic singularity in $V$. Then the holonomy on $I''_s\times(-\eps,\eps)$ clearly satisfies the condition (s). This shows that for each $s$ one can choose $I'_s$ with the desired properties. 

Moreover, whenever $I_s$ satisfies (s) then so does $I'_s$ by \lemref{l:neg-curv}. It follows that we can choose the foliation $\II'$ such the leaf $I'_s$ of $\II'$ satisfies (s) for all $s\in[-1,1]$. The desired isotopy of $F$ can be constructed such that the surface is transversal to $\partial_z$ throughout the isotopy. 
\end{proof}

The following lemma is a partial converse of the elimination lemma. Because is only concerned with the region where $\xi$ is a contact structure we omit the proof. It can be found in \cite{El,giroux}.

\begin{lem} \mlabel{l:create}
Let $F\subset M$ be an embedded surface in a confoliated manifold and $\gamma\subset F$ a compact segment of a nonsingular leaf of the characteristic foliation on $F$ which lies in the contact region of $\xi$. 

Then there is a $C^0$-small isotopy of $F$ with support in a little neighbourhood of $\gamma$ such that after the isotopy there is an additional pair of singularities (one hyperbolic and ons  elliptic) having the same sign. The isotopy can be performed in such a way that $\gamma$ is still tangent to the characteristic foliation and connects the two new singularities.
\end{lem}

We end this section with mentioning a particular perturbation of an embedded surface $F$ which also appears in \cite{El}. Consider an injective Legendrian polygon $(Q,V,\alpha)$ such that there is an elliptic singularity $x$ of $F(\xi)$ such that $\alpha^{-1}(x)$ consists of more than one vertex of $Q$.  

Then $F$ can be deformed by a $C^0$-small isotopy near $x$ into a surface $F'$ such that there is a map $\alpha' : Q \lra F'$ with the same properties as $\alpha$ which coincides with $\alpha$ outside a neighbourhood of $\alpha^{-1}(x)$ and $\alpha'$ maps all vertices in $\alpha^{-1}(x)$ to different elliptic singularities of $F'(\xi)$, cf. \figref{b:split}.

\begin{figure}[htb] 
\begin{center}
\includegraphics{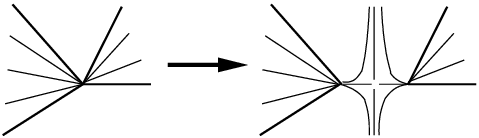}
\end{center}
\caption{\label{b:split}}
\end{figure}

\subsection{Modifications in the neighbourhood of integral discs} \mlabel{s:cutting}

The second method for the manipulation of the characteristic foliation on an embedded surface $F$ is by surgery of the surface along a cycle $\gamma$ which is part of an integral disc of $\xi$. The latter condition is satisfied when the confoliation is tight and $\gamma$ bounds a disc in $F$ (for example when $F$ is simply connected). 

While the elimination lemma is used in the proof of the Thurston-Bennequin inequalities for embedded surfaces in tight contact manifolds, the following lemmas adapt lemmas appearing in \cite{rous, Th} (cf. also \cite{cc}) which are used in the proof the the existence of the Roussarie-Thurston normal form for surfaces in $3$-manifolds carrying a foliation without Reeb components. The existence of this normal forms implies the Thurston-Bennequin inequalities for such foliations.

\begin{lem} \mlabel{l:cut} 
Let $F$ be a surface and $\gamma$ a closed leaf of the characteristic foliation on $F$ such that there is a disc $D$ tangent to $\xi$ which bounds $\gamma$ and has $F\cap D=\gamma$. 

Then there is a surface $F'$ which is obtained from $F$ by removing an annulus around $\gamma$ and gluing in two discs $D_+,D_-$. The discs can be chosen such that the $D_+(\xi),D_-(\xi)$ have exactly one elliptic singularity in the interior of $D_+,D_-$. 

If the germ of the holonomy $h_{\partial D}$ has non trivial holonomy along $\gamma$ on one side of $\gamma$, then we can achieve that the elliptic singularity on the disc on that side lies in the interior of the contact region and every leaf of the characteristic foliation on the new discs connects the singularity with the boundary of the disc.
\end{lem}  

\begin{proof}
We will construct the upper disc $D_+$ in the presence of non-trivial holonomy on the upper side of $\gamma\subset F$. The construction of the other disc is analogous.

Fix a closed neighbourhood $U\simeq D\times(-2\eps,2\eps), \eps>0$ of $D$ such that the fibers of $D\times(-\eps,\eps)$ are positively transverse to $\xi$. We assume $F\cap U=\partial D\times(-2\eps,2\eps)$ and we identify $D\times\{0\}$ with the unit disc in $\R^2$. 

By \lemref{l:neg-curv} there is $x\in D$ and $0<\eta<\eta'<\eps$ such that $x\times[\eta,\eta']$ is contained in the interior of the contact region of $\xi$. On $D$ we consider the singular foliation consisting of straight lines starting at $x$. For $t\in[\eta,\eta']$ let $D_t$ be the disc formed by horizontal lifts of leaves of the singular foliation on $D$ with initial point $(x,t)$. By Gray's theorem we may assume that $\xi$ is generic near $x\times[\eta,\eta']$. Then $D_t(\xi)$ is homeomorphic to the singular foliation by straight lines on $D$ and the singularity is non-degenerate for all $t\in[\eta,\eta']$.  

Let $\rho : [\eta,\eta']\lra[1/2,1]$ be a monotone function which is smooth on $(\eta,\eta']$ such that $\rho\equiv 1$ near $\eta'$ and the graph of $\rho$ is $C^\infty$-tangent to a vertical line at $(\eta,1/2)$. We denote the boundary of the disc of radius $\rho(t)$ in $D_t$ by $S_t$. The union of all $S_t,t\in[\eta,\eta']$ with the part of $D_\eta$ which corresponds to the disc with radius $1/2$ is the desired disc $D_+$. We remove the annulus $\partial D\times[0,\eta']$ from $F$ and add $D_+$.

By construction the only singular point of $D_+(\xi)$ is $(x,\eta)$, the singularity is elliptic and contained in the contact region. Its sign depends on the orientation of $F$. 

In order to show that all leaves of $D_+(\xi)$ accumulate at the elliptic singularity it is enough to show that there are no closed leaves on $D_+$. Assume that $\tau$ is a closed leaf of $D_+(\xi)$. Let $D_\tau$ be the disc formed by lifts of the leaves of the radial foliation on $D$ with initial point on $\tau$. 

The restriction of $\xi$ to $D\times[0,\eps]$ extends to a confoliation $\widetilde{\xi}$ on $\R^2\times\R$ which is a complete connection. By \propref{p:complete connection} $\widetilde{\xi}$ is tight. Hence $\tau$ must bound an integral disc of $\xi'$. Now $D_\tau$ is the only possible candidate for such a disc. But $D_\tau$ cannot be an integral disc of $\widetilde{\xi}$ because it intersects the contact region of $\widetilde{\xi}$ (or equivalently $\xi$) in an open set. This contradiction finishes the proof.
\end{proof}

The following two lemmas are analogues to the elimination lemma in the sense that we will remove pairs of singularities. Note however that new singularities can be introduced. In particular in \lemref{l:cut3} we will obtain a surface whose characteristic foliation is not generic. However this will play no role in later applications since the locus of the non-generic singularities will be isolated from the rest of the surface by closed leaves of the characteristic foliation.  

\begin{lem} \mlabel{l:cut2}
Let $F$ be a surface in a confoliated manifold, $D$ an embedded disc tangent to $\xi$ and  $D\cap F=\gamma$ is a cycle containing exactly one hyperbolic singularity $x_0$.

Then there is a surface $F'$ which coincides with $F$ outside of a neighbourhood of $\gamma$ and is obtained from $F$ by removing a tubular neighbourhood of $\gamma$ and gluing in two discs $D_+,D_-$. The characteristic foliation of $F'$ has no singularities on $D_-$ and one elliptic singularity on $D_+$ whose sign is the opposite of the sign of $x_0$. 
\end{lem}

\begin{proof}
The assumptions of the lemma imply that $x_0$ has a stable and an unstable leaf which do not lie on $D$.

Choose a simple curve $\sigma\subset D$ connecting $x_0$ to another boundary point $x_1$ of $D$ such that $\sigma$ is not tangent to a separatrix of $x$ and extend $\sigma$ to a Legendrian curve such that $x_0,x_1$ become an interior points of $\sigma$. Fix a product neighbourhood $U\simeq\widetilde{D}\times(-\eps,\eps)$ of $D$ with the following properties. 
\begin{itemize}
\item[(i)] $D$ is contained in the interior of the disc $\widetilde{D}\times\{0\}$.
\item[(ii)] There is a simple Legendrian curve $\sigma\subset \widetilde{D}$ containing $x_0$ in its interior and intersecting $\partial D$ respectively $\partial \widetilde{D}$ in two points such that $\gamma$ is nowhere tangent to $\sigma$ respectively $\partial\widetilde{D}$ is transverse to $\sigma$.
\item[(iii)] The fibers of the projection $\pi:\widetilde{D}\times(-\eps,\eps)\lra \widetilde{D}$ are transverse to $\xi$. 
\end{itemize}
Now consider $T_\sigma=\sigma\times(-\eps,\eps)$. The intersection $T_\sigma\cap F$ has a non-degenerate tangency with $T_\sigma(\xi)$ in $x_0$and meets $\sigma\times\{0\}$ transversely in $x_1$. We choose two points $y_0,y_1\in T_\sigma\cap F$ such that $x_0,x_1$ lie between $\pi(y_0)$ and $\pi(y_1)$, as indicated in \figref{b:cut2}.

\begin{figure}[htb] 
\begin{center}
\includegraphics{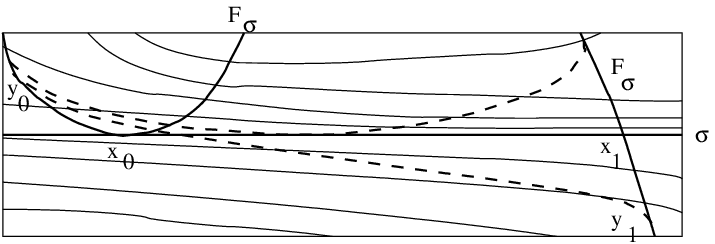}
\end{center}
\caption{\label{b:cut2}}
\end{figure}

The points $y_0,y_1$ can be connected by a curve $\hat{\sigma}\subset T_\sigma$ transverse to the characteristic foliation on this strip provided that $y_0,y_1$ are close enough to $\widetilde{D}$. Moreover, we may assume that $\hat{\sigma}$ is tangent to $F$ near its endpoints (cf. the lower dashed curve in \figref{b:cut2}).

The curve $\hat{\sigma}$ is going to be part of $D_-$. In order to finish the construction of $D_-$ we choose a foliation of $\widetilde{D}$ by a family $I_s, s\in\sigma$ of intervals that connect boundary points of $\widetilde{D}$ and are transverse to $\sigma$. The characteristic foliation on $T_{I_s}$ consists of lines which are mapped diffeomorphically to $I_s$ by $\pi$.

If $\hat{\sigma}$ was chosen close enough to $\widetilde{D}$, then there is a smooth family of curves $\hat{I}_s$ in $I_s\times(-\eps,\eps)$ which
\begin{itemize}
\item[(i)] intersect $\hat{\sigma}$ and are tangent to $\xi$ in these points,
\item[(ii)] are transverse to $\xi$ elsewhere and
\item[(iii)] are tangent to $F$ near $y_0,y_1$.
\end{itemize}
The choices we made for $\hat{\sigma}$ and $\hat{I}_s, s\in\sigma$ ensure that the union of all curves $\hat{I}_s$ is a disc $D_-$ which is transverse to $\xi$. 

The disc $D_+$ is obtained as in the proof of \lemref{l:cut}. The statement about the sign of the singularity of $D_+(\xi)$ follows from the construction. 
\end{proof}


\begin{lem} \mlabel{l:cut3}
Let $F\subset M$ be an embedded surface in a manifold carrying a confoliation $\xi$ such that $F(\xi)$ contains a hyperbolic singularity $x$ and the stable and unstable leaves of $x$ bound an annulus $A\subset F$ which is pinched at $x$. We assume that the pinched annulus is bounded by an integral disc $D$ of $\xi$ such that $\partial A=F\cap D$. 

Then there is an embedded surface $F'$ which is obtained from $F$ by removing a neighbourhood of $\gamma$ and gluing in an annulus $A'$ and a disc $D'$ such that $A'(\xi)$ has one of the following properties.
\begin{itemize}
\item[(i)] $A'(\xi)$ has no singularity. 
\item[(ii)] The singularities of $A'(\xi)$ form a circle and a neighbourhood in $F'$ of this circle is foliated by closed leaves of $F(\xi')$. 
\end{itemize}
The characteristic foliation on $D'$ has exactly one singularity which is elliptic and whose sign is opposite to the sign of $x$. 
\end{lem}

\begin{proof}
The disc $D$ in the statement of the lemma is an immersed disc which is an embedding away from two points in the boundary. These two points are identified to the single point $x$. Let $S^1\simeq\sigma\subset D$ be a simple closed curve in $D$ which meets $x$ exactly once.

We choose a solid torus $C=\sigma\times[-1,1]\times[-1,1]$ such that $\sigma=\sigma\times\{(0,0)\}$ and the foliation corresponding to the second factor is Legendrian while the foliation corresponding to the third factor is transverse to $\xi$. For $s\in[-1,1]$ let $A_s=\sigma\times\{s\}\times[-1,1]$. The torus is chosen such that $D\subset\sigma\times[-1,1]\times\{0\}$ and $F$ intersects $A_-=\sigma\times[-1,1]\times\{-1\}$ in two circles while $F\cap (\sigma\times[-1,1]\times\{1\})$ is a circle which bounds is disc in $\sigma\times[-1,1]\times\{1\}$.

If $C$ is thin enough, then a disc $D'$ which bounds $F\cap (\sigma\times[-1,1]\times\{1\})$ with the desired properties can be constructed as in the proof of \lemref{l:cut}.

Let $P_s:=\sigma(s)\times[-1,1]\times[-1,0],s\in S^1$. The characteristic foliation on $P_s$ consists of lines transverse to the last factor of $P_s$ and $\sigma(s)\times[-1,1]\times\{0\}$ is a leaf of $P_s(\xi)$ 

If $\xi$ one of the annuli $\sigma\times\{t\}\times(-1,0], t\in(-1,1)$ has non-trivial holonomy along $\sigma\times\{(t,0)\}$ or if $\sigma\times\{(t,0)\}$ is not Legendrian, then one can choose a curve $\sigma'$ in that annulus which is transverse to $\xi$. The annulus $A'$ is the union of curves in $P_s, s\in S^1$ which connect the two points of $F\cap(\sigma(s)\times[-1,1]\times\{-1\}$ and pass through $\sigma'\cap P_s$. These curves can be chosen such that they are transverse to $P_s(\xi)$ everywhere except in $\sigma'\cap P_s$. By construction $A'(\xi)$ has the property described in (i) of the lemma.

This construction also applies if we choose $\sigma'$ in annuli which are $C^\infty$-close to $\sigma\times\{t\}\times[-1,0]$ for a suitable $t\in [-1,1]$. If all annuli of this type have trivial holonomy along their boundary curve which is close to $\sigma\times\{(t,0)\}$, then  $\xi$ is a foliation on a neighbourhood of $\sigma$ in $\sigma\times[-1,1]\times[-1,0]$ by \lemref{l:neg-curv} whose holonomy along $\sigma$ is trivial. The same construction as in the previous case (with $\sigma'=\sigma$) yields an annulus $A'$ with the properties described in (ii).  
\end{proof}
\lemref{l:cut} and \lemref{l:cut2} suffice for \secref{s:rigid} because the embedded surfaces in that section are going to be simply connected. Then one can apply \lemref{l:cut2} to one  of the boundary components of the pinched annulus. 

In the lemmas of this section we have assumed that $F\cap D=\gamma$. In general $F$ and $D$ may intersect elsewhere. Since all singularities of the characteristic foliation on $\gamma$ are non-degenerate or of birth-death type, there is a neighbourhood of $\gamma$ in $D$ such that $\gamma$ is the intersection of $F$ with this neighbourhood. After a small perturbation with support outside of a neighbourhood of $\gamma$ we may assume that $F$ is transverse to $D$ on the interior of $D$. Now we can apply \lemref{l:cut} a finite number of times to circles in $F\cap D$ in order to achieve that the resulting surface intersects $D$ only along $\gamma$. Then we can apply the lemmas of this section.


\section{Tight confoliations violating the Thurston-Bennequin inequalities}  \mlabel{s:example}


The example given in this section shows that tightness (as defined in \defref{d:tight confol}) is a much weaker condition for confoliations compared to the rigidity of tight contact structures or foliations without Reeb components. It also shows that it may happen that {\em every} contact structure obtained  by a sufficiently small perturbation of a tight confoliation is overtwisted. This is in contrast to the situation of foliations without Reeb components: According to \cite{colin pert} every foliation without a Reeb component can be approximated by a tight contact structure.

The starting point for the construction of a tight confoliation violating the Thur\-ston-Benne\-quin inequalities is the classification of tight contact structures on $T^2\x I$ such that the characteristic foliation on $T_t=T^2\x \{t\}, t\in\{0,1\}$ is linear (cf. \cite{giroux2}). We fix an identification $T^2\simeq \R^2/\Z^2$ and the corresponding vector fields $\partial_1,\partial_2$. 

According to \cite{giroux2} (Theorem 1.5) there is a unique tight contact structure $\xi$ on $T^2\times I$ such that 
\begin{itemize}
\item[(i)] the characteristic foliation on $\partial(T^2\times I)$ is a pair of linear foliations whose slope is $2$ respectively $1/2$ on $T_0$ respectively $T_1$,
\item[(ii)] the obstruction for the extension of the vector fields which span the characteristic foliation on $\partial (T^2\times I)$ is Poincar{\'e}-dual to $(2,2)\in H_1(T^2;\Z)\simeq\Z^2$. 
\end{itemize}
\figref{b:movie} shows the characteristic foliation on $T^2\times\{t\}$ at various times and its orientation. The two curves in $T^2\times\{1/2\}$ where the characteristic foliation is singular represent the homology class $(2,2)\in H_1(T^2;\Z)$. 
\begin{figure}[htb] 
\begin{center}
\includegraphics{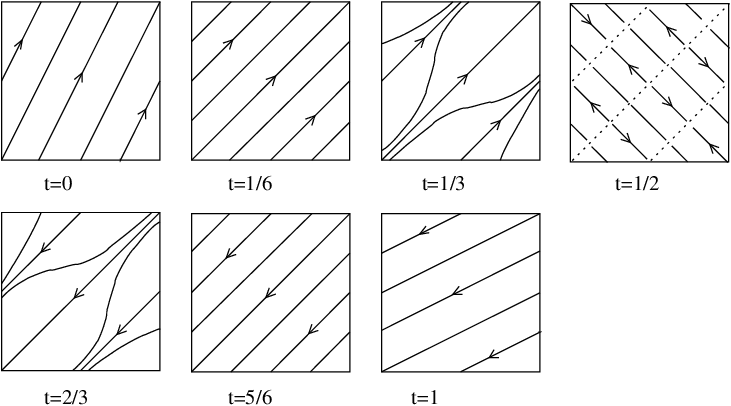}
\end{center}
\caption{\label{b:movie}}
\end{figure}
We may assume that the contact structure is $T^2$-invariant and tangent to $\partial_t$ on a neighbourhood of the boundary (cf. \cite{giroux}).  Then there are smooth functions $f_i,g_i, i\in\{0,1\}$ on this neighbourhood such that $\xi$ is spanned by $\partial_t$ and

\begin{align} \label{e:confol repr}
\begin{split}
f_0(t)\partial_1+g_0(t)\partial_2 & \textrm{ near } T^2\times\{0\} \\
f_1(t)\partial_1+g_1(t)\partial_2 & \textrm{ near } T^2\times\{1\}. 
\end{split}
\end{align}
Because $\xi$ is a positive contact structure, the functions $f_i,g_i$ satisfy the inequalities $f_i'(t)g_i(t)-g_i'(t)f_i(t)>0$ for $i\in\{0,1\}$ on their respective domains. 

We now modify $\xi$ to a confoliation $\tilde{\xi}$ on $V=T^2\times[0,1]$. For this replace the functions $f_i,g_i$ in \eqref{e:confol repr} by $\tilde{f}_i,\tilde{g}_i$ such that for $i=0,1$

\begin{itemize}
\item $\tilde{f}_i,\tilde{g}_i$ coincide with $f_i,g_i$ outside of small open neighbourhoods of $T^2\times\{i\}$   
\item there is $\tau>0$ such that $\tilde{f}_i'(t)\tilde{g}_i(t)-\tilde{g}_i'(t)\tilde{f}_i(t)>0$ if $t\in(\tau,1-\tau)$ and 
\item $\tilde{f}_i'(t)\tilde{g}_i(t)-\tilde{g}_i'(t)\tilde{f}_i(t)\equiv 0$   for $t\in[0,\tau]\cup[1-\tau,1]$
\item $\tilde{f}_i,\tilde{g}_i$ coincide with $f_i,g_i$ at $t=0,1$.
\end{itemize}

\begin{rem} \mlabel{r:tight interior}
From the proof of Theorem 1.5 in \cite{giroux2} it follows that the contact structure $\tilde{\xi}$ on $T^2\times(\tau,1-\tau)$ is tight.
\end{rem}
We write $\xi$ for the confoliation constructed so far. In the next step we will extend $\xi$ to a smooth confoliation on $T^2\times[-1,2]$ such that the boundary consists of torus leaves. 

Let $h$ be a diffeomorphism of $\R^+_0$ such that $h(s)<s$ for $s>0$ and all derivatives of $h(s)-s$ vanish for $s=0$. The suspension of this diffeomorphism yields a foliation on $S^1\times\R^+_0$ whose only closed leaf is $S^1\times\{0\}$ and all other leaves accumulate on this leaf. In this way we obtain a foliation on $S^1\times(S^1\times\R^+_0)$ such that the boundary is a leaf and the characteristic foliation on $S^1\times (S^1\times \{\sigma\})\simeq T^2\times\{\sigma\}, \sigma>0$ corresponds to the first factor. In particular it is linear. 

Using suitable elements of $\{A\in\mathrm{Gl}(2,\Z)|\det(A)=\pm 1\}$ we glue two copies of the foliation on $T^2\times[0,\sigma],\sigma>0$ to $T^2\times[0,1]$. We obtain an oriented confoliation on $T^2\times[-1,2]$ such that the boundary is the union of two torus leaves and we may assume the orientation of the boundary leaves coincides with the orientation of the fiber of $T^2\times[-1,2]$.  

After identifying the two boundary components by an orientation preserving diffeomorphism, we get a closed oriented manifold $M$ carrying a smooth positive confoliation which we will denote again by $\xi$. 

\medskip

{\it Claim: $\xi$ is tight.} 

We show that the assumption of the contrary contradicts \remref{r:tight interior}. Let $\gamma\subset M$ be a Legendrian curve which bounds an embedded disc $D$ in $M$ such that $\xi$ is nowhere tangent to $D$ along $\gamma$ and violates the requirements of \defref{d:tight confol}. By construction $\xi$ has a unique closed leaf $T$. If $\gamma$ is contained in $T$, then $\gamma$ bounds a disc in $T$ because $T$ is incompressible. Thus we may assume that $\gamma$ lies in the complement of $T$ and we can consider the manifold $M\setminus T= T^2\times(-1,2)$.

By \remref{r:tight interior}, $\gamma$ cannot be contained in $T^2\times(\tau,1-\tau)$. If $\gamma$ lies completely in the foliated region $T^2\times\big((-1,\tau]\cup[1-\tau,2)\big)$, then it bounds a disc in its leaf because all leaves are incompressible cylinders. 

It remains to treat the case when the $\gamma$ intersects the contact region and the foliated region. All leaves of $\xi$ in $M\setminus T= T^2\times(-1,2)$ are cylinders which can be retracted into the region $T^2\times[0,\tau)\cup(1-\tau,1]$. Hence we may assume that $\gamma$ is contained in $T^2\times[0,1]$.

First we show that there is a Legendrian isotopy of $\gamma$ such that the resulting curve is transverse to the boundary of the contact region $B=T^2\times\{\tau,1-\tau\}$. A similar isotopy will be used later, therefore we describe it in detail. 

Let $T^2\times(0,\tau')$ with $0<\tau<\tau'$ be a neighbourhood of one component of $B$ where $\xi$ can be defined by the $1$-form 
$$
\alpha_0=dx_1-\frac{\tilde{f}_0(t)}{\tilde{g}_0(t)}dx_2.
$$
We consider the projection $\mathrm{pr} : T^2\times[0,\tau']\lra S^1\times[0,\tau']$ such that the fibers are tangent to $\partial_1$. Note that $d\alpha_0$ is the lift of the $2$-form 
$$
\omega=\frac{\tilde{f}_i'(t)\tilde{g}_i(t)-\tilde{g}_i'(t)\tilde{f}_i(t)}{\tilde{g}_0^2(t)}dx_2\ww dt.
$$ 
The fibers of $\mathrm{pr}$ are transverse to $\xi$. Let $\hat{\gamma}$ be a segment of $\gamma$ which is contained in $T^2\times[0,\tau']$ and whose endpoints do not lie on $B$. 

If $\hat{\gamma}$ is contained in the foliated part of $\xi$, then we isotope $\hat{\gamma}$ within its leaf such that the resulting curve is disjoint from $T^2\times\{\tau\}$ and the isotopy does not affect the curve on a neighbourhood of its endpoints. 

Now assume that some pieces of $\hat{\gamma}$ are contained in the contact region of $\xi$. Then $\mathrm{pr}(\hat{\gamma})$ passes through the region of $S^1\times(\tau,\tau']$ where $\omega$ is non-vanishing. We consider an isotopy of the projection of $\hat{\gamma}$ which is fixed near the endpoints and the area of the region bounded by $\hat{\gamma}$ is zero for all curves in the isotopy. By Stokes theorem this implies that one obtains closed Legendrian curves when $\hat{\gamma}$ is replaced by horizontal lifts of curves of the isotopy (with starting point on $\gamma$).

Hence we may assume that $\gamma$ is transverse to $T^2\times\{\tau\}$ and $\gamma$ is decomposed into finitely many segments whose interior is completely contained in either the contact region or the foliated region of $\xi$.  

Let $\gamma_0\subset\gamma$ be an arc with endpoints in the contact region of $\xi$ such that $\gamma_0$ contains a exactly one sub arc of $\gamma$ lying in the foliated region. Because $\gamma_0$ is embedded, it bounds a compact half disc in a leaf tangent to $\xi$ and we can choose $\gamma_0$ such that the half disc does not contain any other segment of $\gamma$. 

Now we isotope $\gamma_0$ relative to its endpoints such that after the isotopy this segment lies completely in the contact region of $\xi$. As above we deform $\mathrm{pr}(\gamma_0)$ through immersions such that the resulting arc $\hat{\gamma}_0$ has the following properties
\begin{itemize} 
\item the integral of $\omega$ over the region bounded by $\hat{\gamma}_0$ and $\mathrm{pr}(\gamma_0)$ is zero and the same condition applies to every curve in the isotopy,
\item $\hat{\gamma}_0$ is completely contained in $S^1\times(\tau,\tau']$.
\end{itemize}
Then the horizontal lift of $\hat{\gamma}_0$ can be chosen to have the same endpoints as $\gamma_0$ and we can replace $\gamma_0$ by this horizontal lift. The resulting curve is Legendrian isotopic to $\gamma$ but it the number of pieces which lie in the foliated region has decreased by one. 

After finitely many steps we obtain a Legendrian isotopy between $\gamma_0$ and a closed Legendrian curve which lies completely in the interior of the contact region. The Thurston-Bennequin invariant of the resulting curve is still zero. But this is impossible because the contact structure on $T^2\times(\tau,1-\tau)$ is tight.

\medskip

{\it Claim: If $M=T^3$, then $\xi$ violates b) of \thmref{t:TB Ungl tight}.}

The trivialization of $\xi$ induced by the characteristic foliation on $T^2\times\{0,1\}$ extends to the complement of $T^2\times[0,1]$ in $T^3$. The obstruction for the extension of the trivialization from $T^2\times\{0,1\}$ to $T^2\times[0,1]$ is Poincar{\'e}-dual to $(1,1)\in H_1(T^2\times[0,1])$. Hence $e(\xi)$ is Poincare-dual to $(2,2,0)\in H^1(T^2)\oplus \Z$ where the second factor corresponds to the homology of the second factor of $T^3\simeq T^2\times S^1$. This means that $\xi$ violates the Thurston-Benneuqin inequalities since these inequalities imply $e(\xi)=0$ because every homology class in $t^3$ can be represented by a union of embedded tori.    


An example of a torus in $(T^3,\xi_T)$ which violates the Thurston-Bennequin inequality can be described very explicitly. Let $T_0$ be the torus which is invariant under the $S^1$-action transverse to the fibers and it intersects each fiber in a curve of slope $-1$, hence this curve represents $(1,-1)\in H_1(T^2)$ when $T_0$ is suitably oriented. It follows from the description of $\xi$ given above, that $\tau=T_0\cap (T^2\times\{1/2\})$ is Legendrian and the characteristic foliation on $T_0$ has exactly four singular points which lie on $\tau$ and have alternating signs.

Moreover, $T_0\cap T$ is a Legendrian curve and $\xi$ is transverse to all tori $T^2\times\{t\}, t\in(-1,2)$ except in the singular points on $T_0\cap(T^2\times\{1/2\})$. \figref{b:starfish} shows a singular foliation homeomorphic to the one on $T_0$. 
\begin{figure}[htb]
\begin{center}
\includegraphics{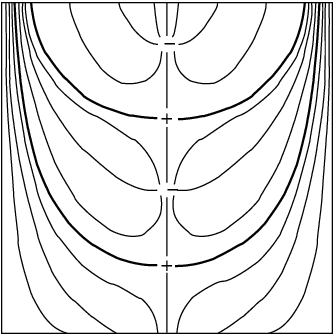}
\end{center}
\caption{\label{b:starfish}}
\end{figure}
We choose the orientation of $T_0$ such that $e(T_0)=-4$. In order to find an example of a surface with boundary which violates the inequality c) from \thmref{t:TB Ungl tight} it suffices to remove a small disc containing one of the elliptic singularities in $T_0$.   

Finally, note that according to \cite{confol} every positive confoliation can be approximated (in the $C^0$-topology) by a contact structure, it follows that tightness is {\em not} an open condition in the space of confoliations with the $C^0$-topology.  Actually $\xi$ can be approximated by contact structures which are $C^\infty$-close to $\xi$. This can be seen by going through the proof of Theorem 2.4.1 and Lemma 2.5.1 in \cite{confol}: By construction the holonomy of the closed leaf on $T_0$ is attractive, therefore it satisfies conditions which imply the conclusion of Proposition 2.5.1, \cite{confol} (despite of the fact that the infinitesimal holonomy is trivial).  The main part of this lemma is stated in \lemref{l:reminder on holonomy approx} together with an outline of the proof. 

Thus tightness is not an open condition for confoliations in general. This answers question 1 from the section 3.7 in \cite{confol} (when tightness is defined as in \defref{d:tight confol}).


\section{Rigidity results for tight confoliations} \mlabel{s:rigid} 


The example from the previous section shows that tight confoliations are quite flexible objects compared to tight contact structures and foliations without Reeb components. In this section we establish some restrictions on the homotopy class of plane fields which contain tight confoliations. 

The first restriction is the Thurston-Bennequin inequality for simply connected surfaces. Note that this imposes no restriction on the Euler class $e(\xi)$ of a tight confoliation $\xi$ on a closed manifold $M$ unless the prime decomposition of $M$ contains $(S^1\times S^2)$-summands. The second restriction on the homotopy class of $\xi$ is a consequence of
\begin{thm} \mlabel{t:tight on balls}
Let $M$ be a manifold carrying a tight confoliation $\xi$ and $B\subset M$ a closed embedded ball in $M$. There is a neighbourhood of $\xi$ in the space of plane fields with the $C^0$-topology such that $\xi'\eing{B}$ is  tight for every contact structure $\xi'$ in this neighbourhood.
\end{thm}
The proof of this theorem is given in \secref{s:perturb}. Let us explain an application of \thmref{t:tight on balls} which justifies the claim that \thmref{t:tight on balls} is a rigidity statement about tight confoliations. 

By \thmref{t:El-Th approx} every confoliation on a closed manifold can be $C^0$-approximated by a contact structure unless it is a foliation by spheres. Hence \thmref{t:tight on balls} can be applied to every confoliation. Recall the following theorem.
\begin{thm}[Eliashberg, \cite{El}] \mlabel{t:ball class}
Two tight contact structures on the $3$-ball $B$ which coincide on $\partial B$ are isotopic relative to $\partial B$. 
\end{thm} 
It follows from this theorem that two tight contact structures on $S^3$ are isotopic and therefore homotopic as plane fields. In contrast to this every homotopy class of plane fields on $S^3$ contains a contact structure which is not tight. Thus the following consequence of \thmref{t:tight on balls} shows that there are restrictions on the homotopy classes of plane fields containing tight confoliations.
\begin{cor}
Only one homotopy class of plane fields on $S^3$ contains a positive tight confoliation. 
\end{cor}
\begin{proof}
Let $\xi$ be a tight confoliation on $S^3$. It is well known that every foliation of rank $2$ on $S^3$ contains a Reeb component, cf. \cite{No}. Thus $H(\xi)$ is not empty. We choose $p\in H(\xi)$ and a ball $B\subset H(\xi)$ around $p$.

According to \cite{confol} $\xi$ can be $C^0$-approximated by a contact structure $\xi'$ on $S^3$ such that $\xi$ and $\xi'$ coincide on $B$. By \thmref{t:tight on balls} the restriction of $\xi'$ to $S^3\setminus B$ is tight and by a result from \cite{colin} $\xi'$ is a tight contact structure on $S^3$ which is homotopic to $\xi$.
\end{proof}
More generally, \thmref{t:tight on balls} together with \thmref{t:ball class} implies that the homotopy class of a tight confoliation $\xi$ as a plane field is completely determined by the restriction of $\xi$ to a neighbourhood of the $2$-skeleton of a triangulation of the underlying manifold. 


\subsection{The Thurston-Bennequin inequality for discs and spheres} \mlabel{s:tb sphere disc}

In this section we prove the Thurston-Bennequin inequalities for a tight confoliation $\xi$ in the cases where $F$ is a sphere or a disc (with transverse boundary). For this we adapt the arguments in \cite{El}. We shall discuss why Eliashberg's proof cannot be adapted for non-simply connected surfaces in tight confoliations after the proof \thmref{t:discs and spheres}. Recall that the self-linking number $\selfl(\gamma,F)$ of a null-homologous knot $\gamma$ which is positively transverse to $\xi$ with respect to a Seifert surface $F$ satisfies $e(\xi)[F]=-\selfl(\gamma,F)$ where $e(\xi)[F]$ corresponds to the obstruction for the extension the characteristic foliation near $\partial F$ to a trivialization of $\xi\eing{F}$.   

\begin{thm} \mlabel{t:discs and spheres}
Let $(M,\xi)$ be a manifold with a tight confoliation. Then
\begin{itemize}
\item[a)] $e(\xi)[S^2]=0$ for every embedded $2$-sphere $S^2\subset M$ and
\item[b)] $\selfl(\partial D, D)\le -1$ for every embedded disc whose boundary is positively transverse to $\xi$.
\end{itemize}
\end{thm}

\begin{proof}
We perturb the surface such that it becomes generic and the elliptic singularities lie in the interior of $H(\xi)$ or in the interior of the foliated region. Furthermore, we will assume in the following that there are no connections between different hyperbolic singularities of characteristic foliations. 

We show $e(\xi)[D]\ge 1$ for every disc as in b). By the Poincar{\'e} index theorem
\begin{align} \label{e:chie}
\begin{split} 
\chi(D) & = e_+(D) + e_-(D) -h_+(D) - h_-(D) \\
e(\xi)(D) & = e_+(D) - e_-(D) - h_+(D) + h_-(D).
\end{split}
\end{align}
Subtracting these equalities we obtain $\chi(D)- e(\xi)[D]=2(e_--h_-)$. In order to prove the b) it suffices to replace $D$ by an embedded disc $D'$ with $e(\xi)[D]=e(\xi)[D']$ such that $D'$ contains no negative elliptic singularities. 

Because $\xi$ is tight and $D$ is simply connected each cycle of $D(\xi)$ is the boundary of an integral disc. We can apply \lemref{l:cut} or \lemref{l:cut2} to such discs to obtain a new embedded disc $D'$. By (iii) of \defref{d:tight confol} $e(\xi)[D]=e(\xi)[D']$. 

We now choose particular cycles of $D(\xi)$ to which we apply \lemref{l:cut} and \lemref{l:cut2}: Define $\gamma\le\gamma'$ for two cycles $\gamma,\gamma'$ of the characteristic foliation if $\gamma'$ bounds an embedded disc containing $\gamma$. We apply \lemref{l:cut} and \lemref{l:cut2} to cycles which are maximal with respect to $\le$. This means in particular that the holonomy of maximal cycles which are closed leaves of $D(\xi)$ is not trivial on the outer side of the cycle.

Hence we obtain a disc $D'$ whose characteristic foliation does not have closed cycles and all elliptic singularities are contained in $H(\xi)$. In particular there are no integral discs of $\xi$ which pass though elliptic singularities of the characteristic foliation of $D$. Moreover, $e(\xi)[D]=e(\xi)[D']$. From now on we will write $D$ instead of $D'$. 

Adapting arguments from \cite{El} we eliminate one negative elliptic singularity $y$. Let $U$ be a disc such that $\partial U$ is transverse to $D(\xi)$ and $y\in U$. According to \lemref{l:leg poly} there is a Legendrian polygon $(Q,V,\alpha)$ covering $\overline{B(U)}$. In the present situation $V=\emptyset$ since $D(\xi)$ has no cycles or exceptional minimal sets.   
Note that $B(U)\subset D$ because the characteristic foliation is pointing outwards along $\partial D$.  After a small perturbation of $D$ we may assume that $\alpha$ identifies vertices of $\partial Q$ only if adjacent edges are also identified, for elliptic vertices this is illustrated in \figref{b:split}. In this situation all boundary components of $\partial\overline{B(y)}$ are embedded piecewise smooth circles.  

Recall that $D(\xi)$ contains no cycles. Then every boundary component $\gamma_{o}$ of $\overline{B(y)}$ therefore contains an elliptic singularity (which has to be positive). If all singularities of $D(\xi)$ on $\gamma_{o}$ are positive, then we obtain a contradiction to the tightness of $\xi$. Hence $\gamma_{o}$ contains a negative singularity which has to be hyperbolic. According to our assumptions it is a pseudovertex of the Legendrian polygon, ie. its unstable leaf ends at $y$ while the other unstable leaf never meets $B(y)$.

Therefore the application of the elimination lemma (\lemref{l:elim}) does not create new cycles on the disc. We continue with the elimination of negative elliptic singularities until $e_-=0$. This finishes the proof of b)

Now we come to the prove of a). First we use \lemref{l:cut} and \lemref{l:cut2} in order to decompose $S$ into a disjoint union of embedded spheres such that there are no cycles which contain hyperbolic singularities. In the following we consider each sphere individually, so we continue to write $S$. If $S(\xi)$ contains a closed leaf, then the claim follows immediately from the definition of tightness: Let $D_1,D_2\subset S$ be the two discs with $\partial D_1=\gamma=\partial D_2$. Then there is an integral disc $D'$ of $\xi$ such that $\partial D'=\gamma$. We orient $D'$ such that $D_1\cup D'$ is a cycle and denote by $-D'$ the disc with the opposite orientation. Then $[S]=[D_1\cup D']+[(-D')\cup D_2]$ and the claim follows from (iii) of \defref{d:tight confol} applied to $D_1,D_2$: 
$$
e(\xi)[S]=e(\xi)[D_1\cup D'] + e(\xi)[(-D')\cup D_2]=0.
$$
Finally if $S(\xi)$ has neither closed leaves or cycles, then one can prove a) using b) when one considers complements of small discs around positive or negative elliptic singularities.  
\end{proof}

Consider a Legendrian polygon $(Q,V,\alpha)$ in $F\subset M$ when $\xi$ is a contact structure on $M$. Generically the characteristic foliation on $F$ is of Morse-Smale type (cf. \cite{giroux}). In particular there are no quasi-minimal sets. If the set of virtual vertices of the Legendrian polygon $(Q,V,\alpha)$ associated to $U$ is not empty, then by \lemref{l:create} one can create of a canceling pair of singularities along on $\gamma_v$ for $v\in V$ such that all leaves which accumulated on $\gamma_v$ now accumulate on an elliptic or a hyperbolic singularity.

For this reason the case $V\neq\emptyset$ plays essentially no role when $\xi$ is a contact structure. If the $\omega$-limit set of $\gamma$ is contained in the fully foliated part of $\xi$, then it not possible to apply \lemref{l:create} (cf. \secref{s:example}). It is at this point where the proof of the Thurston-Bennequin inequalities for tight contact structures fails when one tries to adapt the arguments from \cite{El} to tight confoliations and surfaces which are not simply connected.

We finish this section with a remark that will be useful later.
\begin{rem} \mlabel{r:d+} 
Let $\xi$ be a tight confoliation. For an embedded surface $F\subset M$ we define $d_\pm(F)=e_\pm(F)-h_\pm(F)$ for open subsets of $F$. Note that if $F$ is a sphere, then $d_+(F)=d_-(F)=1$ by \thmref{t:discs and spheres} and $\chi(F)=2$.

Part b) \thmref{t:discs and spheres} can be strengthened: It is not only possible to replace $D$ be a disc with the same boundary and $e(\xi)[D]=e(\xi)[D']$ such that $D'(\xi)$ has no negative elliptic singularities. Consider $\alpha$-limit set of stable leaves of positive hyperbolic singularities of $D'$. Since $D'(\xi)$ contains no cycles the $\alpha$-limit set is generically a positive elliptic singularity. Thus we may eliminate all negative elliptic and all positive hyperbolic singularities from $D'(\xi)$. This implies the following inequalities:
\begin{align*}
d_-(D) & =e_-(D)-h_-(D)= e_-(D')-h_-(D')\le 0 \\
d_+(D) & =e_+(D)-h_+(D)= e_+(D')-h_+(D')\ge 0
\end{align*}
In a later application we shall consider discs such that $\partial D$ is negatively transverse to $\xi$. Then the two inequalities above will be interchanged. 
\end{rem}


\subsection{Perturbations of tight confoliations on balls}  \mlabel{s:perturb}


The proof \thmref{t:tight on balls} is given in the following sections. It has two main ingredients: First we generalize taming functions on spheres to confoliations. We show that the characteristic foliation on an embedded sphere $S$ can be tamed if $\xi$ is tight and that this remains true for contact structures $\xi'$ which are close enough to $\xi$. Then we apply arguments from \cite{giroux2} to conclude that $\xi'|B$ is tight if $\xi'$ is a contact structure.

In the following sections $\xi$ will always be an oriented tight confoliation on $M$ and $S$ denotes an embedded oriented sphere. We do not consider foliations by spheres.


\subsubsection{Properties of characteristic foliations on spheres}

The tightness of $\xi$ leads to restrictions on the signs of hyperbolic singularities on $\gamma$. \lemref{l:signs in cycles on spheres} is concerned with signs of hyperbolic singularities on cycles of $S(\xi)$ when $\xi$ is a tight confoliation. To state it we need the following definition:
\begin{defn}
A cycle connected $\gamma$ of $S(\xi)$ is an {\em internal subcycle} if there is another cycle $\gamma'$ of $S(\xi)$ such that $\gamma\cap\gamma'$ is not empty and the integral disc which bounds $\gamma'$ contains the integral disc which bounds $\gamma$.

A leaf $\gamma$ of $S(\xi)$ is called {\em internal} if there are two cycles of $S(\xi)$ which bound discs tangent to $\xi$ whose interiors are disjoint. We say that a hyperbolic singularity on $\gamma$ is {\em essential} if it is not lying on an internal subcycle of $\gamma$. 

The union of singular points and cycles of $S(\xi)$ will be denoted by $\Sigma(S)$. This set is compact.
\end{defn}
An example of an internal subcycle is shown in \figref{b:fol-taming-ex}. Note that one can create internal cycles intersecting a fixed cycle of $S(\xi)$ with arbitrary sign using an inverse of the construction explained in \lemref{l:cut2}.

If a connected cycle $\gamma$ of $S(\xi)$ contains hyperbolic singularities, then the holonomy along $\gamma$ can be defined at most on one side. The one-sided holonomy is defined if and only if there is an immersion of a disc $D$ into $S$ which is an embedding on $\ring{D}$ and $\partial D$ is mapped onto $\gamma$ such that the image of $\ring{D}$ does not contain a stable or unstable leaf of a hyperbolic singularity on $\gamma$.  We will say that $D$ is a disc in $S$ although some points on the boundary may be identified. 

The singularities on $\gamma$ can be decomposed into two classes

\begin{align*}
A(\gamma) & = \{ \textrm{hyperbolic singularities on }\gamma \textrm{ such that } \gamma \textrm{ contains}\\ 
& \quad  \textrm{ both stable leaves}\} \\
B(\gamma) & = \{ \textrm{hyperbolic singularities on }\gamma \textrm{ such that } \gamma \textrm{ contains}\\
& \quad  \textrm{ only one of the two stable leaves}\}.
\end{align*}
Let $\gamma$ be a cycle of $S(\xi)$ and $D\subset S$ a disc with $\partial D=\gamma$ whose interior does not contain a stable leaf of a hyperbolic singularity on $\gamma$. Then the one-sided holonomy along $\gamma$ is well defined. Because $\xi$ is tight there is a disc $D'$ tangent to $\xi$ such that $\partial D' = \gamma$. We orient $D'$ using the orientation of $\xi$. 
\begin{defn} \mlabel{d:pot}
We say that $\gamma$ is {\em potentially attracting} if
\begin{itemize}
\item[(i)] $D$ lies below respectively above $D'$ (with respect to the coorientation of $\xi$) in a neighbourhood of $D'$ and
\item[(ii)] the orientation of $\gamma$ is opposite respectively equal to the orientation of $\partial D'$. 
\end{itemize}
In the opposite case, $\gamma$ is {\em potentially repulsive}.   
\end{defn} 
According to \lemref{l:neg-curv} the holonomy along potentially repulsive respectively attractive cycles is non-repelling respectively non-attracting. The terminology of \defref{d:pot} is introduced to deal with the case when the holonomy is trivial (and therefore non-repelling and non-attracting at the same time). 

\begin{lem} \mlabel{l:signs in cycles on spheres}
Let $\gamma$ be a cycle of $S(\xi)$ containing a hyperbolic singularity and such that the one-sided holonomy is defined. 

Then all essential singularities in $A(\gamma)$ have the same sign and all essential singularities in $B(\gamma)$ have the opposite sign. The one-sided holonomy is potentially attractive (respectively repulsive) if and only if all singularities in $A(\gamma)$ are negative (respectively positive) and all singularities in $B(\gamma)$ are positive (respectively negative).

The signs of the non-essential singularities in $A(\gamma)$ respectively $B(\gamma)$ is opposite to the sign of the  essential singularities in $A(\gamma)$ respectively $B(\gamma)$.  
\end{lem}

\begin{proof}
Let $D\subset S$ be the disc in $S$ with $\partial D=\gamma$ such that the one-sided holonomy is defined on the side of $\gamma$ where $D$ is lying. Because $\xi$ is tight, there is a disc $D'$ tangent to $\xi$ which bounds $\gamma$. Consider a tubular neighbourhood of $D'$ which contains a collar of $\partial D$ and the collars lies on  one side of $D'$ in the tubular neighbourhood.

The statement about the signs of singularities now follows by looking how $D$ approaches $D'$ near the tangencies and the relation between the signs and the holonomy is a consequence of our orientation conventions and \lemref{l:neg-curv}. \end{proof} 

The following proposition is a generalization of Lemma 4.2.1 in \cite{El}. It will play an important role in the proof of \thmref{t:tight on balls}.

\begin{prop} \mlabel{p:basins in spheres}
Let $\xi$ be a tight confoliation on $M$ and $S\subset M$ an embedded sphere such that the singularities of $S(\xi)$ are non-degenerate. 
Let $U\subset S$ be a connected submanifold of dimension $2$ such that $\partial U$ is transverse to $S(\xi)$ and $S(\xi)$ points outwards along $\partial U$. Each connected component $\Gamma$ of the boundary  the associated Legendrian polygon $(Q,V,\alpha)$ has the following properties. 
 
\begin{itemize}
\item[(i)] If there is a negative elliptic singularity $x$ on $\alpha(\Gamma)$ such that $\alpha(Q)$ is not a neighbourhood of $x$ or a cycle $\gamma_v$ with $v\in V\cap \Gamma$ such that $\alpha(Q)$ is not a one-sided neighbourhood of $\gamma_v$, then $\alpha(\Gamma)$ contains a positive pseudovertex. 
\item[(ii)] If $d_+(U)=1$ and $(Q,V,\alpha)$ identifies the edges $e_1,\ldots,e_l$ of $\Gamma$, then $\alpha$ maps the pseudovertices on $e_1,\ldots,e_l$ to negative hyperbolic singularities of $S(\xi)$.
\end{itemize}  
\end{prop}

\begin{proof}
It was shown in \lemref{l:leg poly} that $\overline{B(U)}$ is covered by a Legendrian polygon $(Q,V,\alpha)$. Recall that $\alpha$ is defined only on $\Gamma\setminus(\Gamma\cap V)$, but we shall denote $\alpha(\Gamma\setminus(\Gamma\cap V))$ by $\alpha(\Gamma)$. 

First we reduce the situation to the case when $V=\emptyset$. By the theorem of Poincar{\'e}-Bendixon, the $\omega$-limit sets corresponding to points of $V$ are cycles. Because $\xi$ is tight, these cycles bound integral discs of $\xi$ and we can apply \lemref{l:cut} or \lemref{l:cut2}. 
Since the discs bounding these cycles may intersect $U$ it is also necessary to consider cycles in $U$. 

Let $v\in V$ and $D_v$ the integral disc of $\xi$ which bounds $\gamma_v$ and $\gamma_i$ a cycle of $S(\xi)$ which is contained in $D_v$. We assume that the disc $D_i\subset D_v$ bounded by $\gamma_i$ intersects $S$ only along $\gamma_i$.  The cycle $\gamma_{i}$ is either contained in $U$ or in the complement of $U$. 

We begin with the case $\gamma_i\subset U$. In this case we obtain two embedded spheres $S',S''$ by cutting along $\gamma_{i}$. When we use \lemref{l:cut} for this the subset $U\subset S$ induces two subsets $U'\subset S', U''\subset S''$ such that $U'$ respectively $U''$ contains one positive respectively one negative singularity in addition to singularities which were already present in $S$, $\partial U'$ respectively $\partial U''$ is transverse to $S'(\xi)$ respectively $S''(\xi)$ and the characteristic foliation points outwards.  The pseudovertices of the Legendrian polygons associated to the basins of $U',U''$ coincide with the pseudovertices of $(Q,\alpha,V)$.
If $d_+(U)=1$, then
\begin{align}
\begin{split} \label{e:d+ spaltung}
d_+(U')+d_+(U'')& = d_+(U)+1 \\
d_+(S'\setminus U')+d_+(U') & = d_+(S')=1 \\
d_+(S''\setminus U'')+d_+(U'') & = d_+(S'')=1.
\end{split}
\end{align}
Notice that $(S'\setminus U')\cup (S''\setminus U'')=S\setminus U$ and $\partial(S\setminus U)$ is negatively transverse to $S_\xi$. It follows from \remref{r:d+} that $d_+(S'\setminus U')\le 0$ and $d_+(S''\setminus U'')\le 0$. Together with \eqref{e:d+ spaltung} this implies $d_+(U')=d_+(U'')=1$. 

If we applied \lemref{l:cut2} and the hyperbolic singularity was positive respectively negative, then $h_+(U'\cup U'')=h_+(U)-1$ respectively $e_+(U'\cup U'')=e_+(U)+1$ and one of the sets, say $U'$ coincides with $U$. Then $d_+(U)=1$ implies $d_+(U'')=1$. 

When $\gamma_{i}$ lies in the complement of $U$, cutting along $\gamma_{i}$ will not affect $U$ or $d_+(U)$ but the basin of $U$ can change: We might remove a virtual vertex, or after the surgery process some boundary components of the Legendrian polygon might be mapped to a negative elliptic singularity while they were accumulated on a cycle before. The pseudovertices are not affected. Note also that if $\alpha(Q)$ is a one--sided neighbourhood of a cycle $\gamma_v$, then the Legendrian polygon which results from the surgery along $\gamma_v$ will be a neighbourhood of the negative elliptic singularity which results from surgery process. (Recall that $\gamma_v$ has well defined attractive one--sided holonomy on the side of $\alpha(Q)$). 

After finitely many steps we obtain a finite union of embedded spheres $S_j$ and subsets $U_j$ with the same properties as $U$ such that the associated Legendrian polygon $(Q_j,V_j,\alpha_j)$ satisfies $V_j=\emptyset$. Therefore is suffices to prove the claim when $\overline{B(U)}$ is covered by a Legendrian polygon $(Q,V,\alpha)$ with $V=\emptyset$. Let $\Gamma$ be a boundary component of $Q$. 

We now prove (i). Let $x\in\alpha(\Gamma)$ be an elliptic singularity such that $\overline{\alpha(Q)}$ is not a neighbourhood of $x$. Then the connected component of $\partial(\alpha(Q))$ containing $x$ is a piecewise smooth closed curve $c$. After a perturbation of the sphere we may assume that $c$ does not contain corners, $x\in H(\xi)$ and $c$ is embedded (cf. \figref{b:split}). If all singularities on $c$ were negative, then we would get a contradiction to the tightness of $\xi$ since no integral surface of $\xi$ can meet $x$. Since all elliptic singularities on $c\subset\alpha(\partial Q)$ are attractive and therefore negative there must be a positive pseudovertex on $c$. 

It remains to prove (ii). Assume $d_+(U)=1$ and let $x_1,\ldots,x_l,l\ge 2$ be the pseudovertices on the edges $e_1,\ldots,e_l\subset\Gamma$. 

When $\alpha(e_i)=\alpha(e_j)$ for $i\neq j$, then $l=2$. Let $\eta,\eta'$ be the two stable leaves of $\alpha(x_1)$. After a small perturbation of $S$ in the complement of $U$ we may assume that the $\alpha$-limit sets of $\eta,\eta'$ are contained in $U$. 

If $\alpha(e_i)\neq\alpha(e_j)$ for all $i\neq j$, then let $\alpha(x_i),\alpha(x_j)$ be two hyperbolic singularities which lie on the cycle associated to identified edges (cf. \defref{d:identify}) and are connected by a piecewise smooth simple oriented path $\sigma$ in the complement of $U$ consisting of leaves of $S(\xi)$ and hyperbolic singularities (as corners) such that $\sigma$ starts at $\alpha(x_i)$ and ends at $\alpha(x_j)$ without passing through images of other pseudovertices. After a small perturbation of $S$ in the neighbourhood of $\alpha(x_j)$ we obtain a sphere $S'$ such that the  $\alpha$-limit sets $\Alpha(\eta),\Alpha(\eta')$ of the two stable leaves $\eta,\eta'$ of $\alpha(x_i)$ are contained in $U$. 

We may assume that neither $\Alpha(\eta)$ or $\Alpha(\eta')$ is a hyperbolic singularity or a singularity of birth-death type. By the Poincar{\'e}-Bendixon theorem $\Alpha(\eta)$ is either an elliptic singularity or a cycle. The same is true for $\Alpha(\eta')$. Using \lemref{l:cut} and \lemref{l:cut2} we can ensure that $\Alpha(\eta)$ is an elliptic singularity, which has to be positive. Note that $\eta,\eta'$ lie in the same connected component of the two spheres obtained by the surgery along cycles in $U$.

For the same reason we may assume that the $\alpha$-limit set of each stable leaf of hyperbolic singularities in $U$ is an elliptic singularity in $U$. Under these conditions the hypotheses $d_+(U)=1$ implies that the graph formed by positive singularities (except birth-death type singularities) and stable leaves of hyperbolic singularities is a connected tree.  

Both stable leaves of $\alpha(x_1)$ together with the simple path on the tree $\Gamma$ connecting $\Alpha(\eta)$ with $\Alpha(\eta')$  form a simple closed curve $\gamma$ on $S$ which is Legendrian. All singularities on $\gamma$ except $\alpha(x_i)$ are positive by construction. Moreover, $\gamma$ contains an elliptic singularities which lies in $H(\xi)$. If $\alpha(x_i)$ is positive we obtain a contradiction to the tightness of $\xi$ since $c$ cannot bound an integral disc of $\xi$. 
\end{proof}

In order to apply the previous proposition efficiently it remains to show that either one of the two parts of \propref{p:basins in spheres} can be used or $\Gamma\subset\partial Q$ does not contain any pseudovertices at all. This is done in the following lemma.

\begin{lem} \mlabel{l:edges of poly}
In the situation of \propref{p:basins in spheres} $\partial Q$ has more connected components or one of the following statements holds for each connected component $\Gamma$ of $\partial Q$. 
\begin{itemize}
\item[(i)] There is a connected component $\Gamma$ of $\partial Q$ such that $\alpha(\Gamma)$ is an elliptic singularity and $\alpha(Q)$ is a neighbourhood of $x$ or $\alpha(\Gamma)$ is a cycle and $\alpha(Q)$ is a one-sided neighbourhood of that cycle.
\item[(ii)] $\alpha(\Gamma)$ contains a cycle of $S(\xi)$ such that $\alpha(Q)$ is not a one-sided neighbourhood of $\alpha(\Gamma)$ or $\alpha(\Gamma)$ contains an elliptic singularity such that $\alpha(Q)$ is not a neighbourhood of $x$.
\item[(iii)] $\alpha$ identifies edges on $\Gamma$.   
\end{itemize}
\end{lem}

\begin{proof}
After a small perturbation of $S$ we may assume that all negative elliptic singularities on $S$ lie in $H(\xi)$ or the interior of the complement of $H(\xi)$. As in the proof of the previous proposition the problem can be reduced to the case when $\Gamma\cap V=\emptyset$. 

We show that if (i) and (ii) do not hold for $\Gamma$, then (iii) applies to $\Gamma$. 
In the following discussion we ignore corners on $\alpha(\Gamma)$ if two of their separatrices lie in the complement of $\overline{\alpha(Q)}$.

Let $x_1\in\alpha(\Gamma)$ be an elliptic singularity. Since $\alpha(\Gamma)\neq x_1$ there is an unstable leaf $\eta_1'$ of a pseudovertex $y_1$ which ends at $x_1$. Let $\eta_1$ be the other unstable leaf of $y_1$. 

If the $\alpha$-limit set of $\eta_1$ is a negative elliptic singularity, then $y_1$ is contained in the interior of $\alpha(Q)$ and the two edges of $\Gamma$ which correspond to $y_1$ are identified by $\alpha$. Otherwise the $\omega$-limit set of $\eta_1$ is a hyperbolic singularity $y_2$ and we can assume that $y_2$ is a pseudovertex of $\Gamma$. There is a unique unstable leaf $\eta_2$ of $y_2$ which is not contained in the interior of $\alpha(Q)$. In particular the $\omega$-limit set of $\eta_2$ cannot by an elliptic singularity. Thus the $\omega$-limit set of $\eta_2$ is the image $y_3$ of a pseudovertex of $Q$. 
If $y_3=y_1$, then $\alpha$ identifies the edges corresponding to $y_1$ and $\eta_1,\eta_2$ form a non-trivial cycle of $S(\xi)$. 

Otherwise we continue as above until a pseudovertex appears for the second time. This happens after finitely many steps since $\Gamma$ contains only finitely many pseudovertices. If we obtained a sequence $y_1,y_2,\ldots,y_r, r\ge 2$ with $y_1=y_r$, then $\alpha$ identifies the edges corresponding to the pseudovertices $y_1,\ldots,y_{r-1}$. Thus if (i) and (ii) do not apply, then (iii) is true. 
\end{proof}


\subsubsection{Taming functions for characteristic foliations on spheres}

Taming functions for characteristic foliations were introduced by Y.~Eliashberg in \cite{El}. In this section we extend the definition of taming functions so that it can be applied to spheres embedded in manifolds carrying a tight confoliation.  

Let $S$ be an embedded sphere in a confoliated manifold such that the singularities of the characteristic foliation $S(\xi)$  are non-degenerate or of birth-death type. This assumption holds in particular for spheres in a generic $1$-parameter family of embeddings. In addition we may assume that there are at most two different hyperbolic singularities which are connected by their stable/unstable leaves. 

\begin{defn} \mlabel{d:taming fct} 
Let $U\subset S$ be a compact submanifold of dimension $2$ in $S$ whose boundary is piecewise smooth and does not intersect $\Sigma(S)$. Moreover, we assume that every connected component $\Gamma\subset\partial U$ satisfies one of the following conditions:
\begin{itemize} 
\item[(1)] $\Gamma$ is either transverse or tangent to $S(\xi)$.
\item[(2)] $\Gamma$ intersects one respectively two stable leaves of hyperbolic singularities of $S(\xi)$ (these singularities may be part of a cycle, cf. \figref{b:levelset} or $U$ is a neighbourhood of a hyperbolic singularity). Each smooth segment of $\Gamma$ intersects exactly one separatrix of a hyperbolic singularity in $U$ and each segment is transverse to $S(\xi)$.
\item[(3)] $U$ is disc and a neighbourhood of a birth-death type singularity of $S(\xi)$ such that $\partial U$ consists of two smooth segments transverse to $S(\xi)$.     
\end{itemize}
A function $f: U\lra \R$ is a {\em taming function} for $S(\xi)$ if it has the following properties.
\begin{itemize}
\item[(o)] If a component $\Gamma\subset\partial U$ belongs to the class (1), then $f$ is assumed to be constant along $\Gamma$. If $\Gamma$ is of class (2) or (3) we require that $f\eing{\Gamma}$ has exactly one critical point in the interior of each of the smooth segments of $\Gamma$.
\item[(i)] The union of the singular points of $S(\xi)$ with all points on internal leaves coincides with the set of critical points of $f$. The function is strictly increasing along leaves of $S(\xi)$ which are not part of a cycle and $f$ is constant along cycles of $S(\xi)$.
\item[(ii)] Positive respectively negative elliptic points of $S(\xi)$ are local minima respectively maxima of $f$. 
\item[(iii)] If the level set $\{f=C\}$ contains only hyperbolic singularities, then as $C$ increases the number of closed connected components of $\{f=C\}$ changes by $h_-(\{f=C\})-h_+(\{f=C\})$.
\end{itemize}
\end{defn}
Requirement (i) in \defref{d:taming fct} is slightly more complicated than one might expect. \figref{b:fol-taming-ex} gives an example of a sphere $S$ in $\R^3$ equipped with the foliation by horizontal planes. A part of the characteristic foliation is indicated in the right part of \figref{b:fol-taming-ex} where the cycle containing the internal subcycle is thickened. If one requires that singular points of $S(\xi)$ should coincide with critical points of the taming function, then $S(\xi)$ cannot be tamed although the confoliation in question is tight.
\begin{figure}[htb]
\begin{center}
\includegraphics{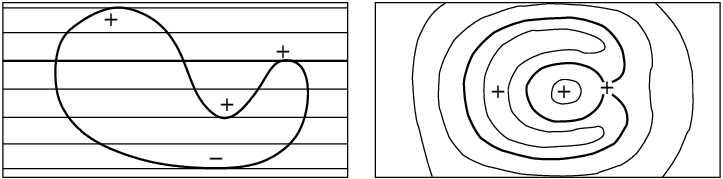}
\end{center}
\caption{\label{b:fol-taming-ex}}
\end{figure}

Assume that $(X,\omega)$ is a symplectic filling of $(M,\xi)$ and a compatible almost complex structure on $M$ is fixed such that $\xi$ consists of complex lines. By Theorem 1 of \cite{hind} an embedded $2$-sphere $S\subset M$ can be filled by holomorphic discs when the embedding of $S$ satisfies several technical conditions. The singular foliation in the formulation of Theorem 1 in \cite{hind} is very similar to the singular foliation formed by level sets of a taming function. The appearance of internal cycles should be compared with Remark 2 in \cite{hind}.


\subsubsection{Construction and deformations of taming functions}

Let $S\subset M$ an embedded oriented $2$-sphere. The tightness of $\xi$ leads to several restrictions on the combinatorics of the cycles of $S(\xi)$ and their holonomy. This will be used to construct a taming function for $S(\xi)$. 

Recall that the orientations of $S$ and $\xi$ induce an orientation of $S(\xi)$ and integral surfaces of $\xi$ are oriented by $\xi$. If $\gamma$ is a cycle of $S(\xi)$, then by tightness there is an integral disc $D_\gamma$ of $\xi$ such that $\partial D_\gamma=\gamma$ but the orientation of $\partial D_\gamma$ as boundary of $D_\gamma$ does not coincide with the orientation of $\gamma$ in general. Recall also that $D_\gamma$ is uniquely determined because $\xi$ is not a foliation by spheres.

For a $2$-dimensional submanifold $U\subset S$ with piecewise smooth boundary we define the following quantities:
\begin{align*}
d_+(U) & = e_+(U)-h_+(U) \\
N_-(U) & = \textrm{Number of connected components }\Gamma\textrm{ of }\partial U\textrm{ where }S(\xi)\\  & \quad\textrm{ points transversally into }U \textrm{ or } \Gamma \textrm{ is tangent to } S(\xi)  \\
& \quad  \textrm{ and }\Gamma \textrm{ is potentially repulsive on the side of }U. \\
N_s(U) & = \textrm{Number of boundary components of }\partial U \textrm{ through which}\\
& \quad\textrm{ stable leaves of negative hyperbolic singularities enter.}\\
P_s(U) & = \textrm{Number of stable leaves of positive hyperbolic singularities in }U\\
& \quad\textrm{ which intersect }\partial U.
\end{align*}
These quantities will be used in the construction of taming functions.

\begin{lem} \mlabel{l:taming functions near cycles} 
For each path connected component $\Sigma_0$ of $\Sigma(S)$ there is a neighbourhood $U_0$ of $\Sigma_0$ and a taming function $f : U_0 \lra \R$ such that no connected component of $\partial U_0$ is tangent to $S(\xi)$ and
\begin{equation} \label{e:d+}
d_+(U_0)=1-N_-(U_0)-P_s(U_0)-N_s(U_0).
\end{equation}
\end{lem}

\begin{proof}
We will construct $U_0$ and $f: U_0\lra\R$ inductively. The starting point are connected cycles $\gamma$ and singularities of $S(\xi)$ in $\Sigma_0$ which belong to the following classes.
\begin{itemize}
\item[(i)] Positive elliptic singularities and hyperbolic or birth-death type singularities which do not belong to a cycle.
\item[(ii)] Closed leaves with sometimes attractive (non-trivial) one-sided holonomy.
\item[(iii)] Cycles $\gamma$ containing hyperbolic singularities which satisfy the following conditions:
\begin{itemize}
\item The only cycle of $S(\xi)$ containing $\gamma$ is $\gamma$.
\item If $\gamma_0\subset\gamma$ is a subcycle with potentially attractive one-sided holonomy, then this one-sided holonomy is not trivial.
\end{itemize}
\end{itemize}
If the positive elliptic singularity $y$ in (i) is dynamically hyperbolic, then it is a source and there is a taming function on a neighbourhood $U$ whose boundary is transverse to $S(\xi)$. If the elliptic singularity is not dynamically hyperbolic, then one obtains a taming function using the holonomy of an interval $[0,\eta), \eta>0$ which is transverse to $S(\xi)$ except at $y$ and $y$ corresponds $0$ (cf. \lemref{l:little discs}). If the holonomy is non-trivial, then we can choose the domain $U$ of the taming function such that $\partial U$ is transverse to $S(\xi)$. Otherwise we choose $U$ such that $\partial U$ is a closed leaf of $S(\xi)$. Moreover, $U$ satisfies \eqref{e:d+}.

If $x$ is a hyperbolic singularity or a singularity of birth-death type, then the existence of a taming function on a neighbourhood $U$ which satisfies \eqref{e:d+} is obvious. 

For a closed leaf $\gamma$ of $S(\xi)$ as in (ii) we choose an embedded interval $(-\eta,\eta),\eta>0$ transverse to $S(\xi)$ such that $0$ corresponds to a point in $\gamma$ and $(-\eta,0]$ corresponds to the side where the holonomy of $\gamma$ is sometimes attractive. This choice determines $f$ along the transverse segment and $f$ can be extended to a taming function on a neighbourhood of $\gamma$. If the holonomy on the side $\{f\ge 0\}$ is non-trivial (respectively trivial) we choose $U$ to be an annulus with transverse boundary (respectively such that $\partial U\cap\{f>0\}$ is a leaf of $S(\xi)$ and the other component of $\partial U$ is transverse to $S(\xi)$). Thus $N_-(U)=1$ and $U$ contains no singular points of $S(\xi)$. This means that \eqref{e:d+} holds for $U$. 

Now let $\gamma$ be a cycle containing hyperbolic singularities. For each subcycle with potentially attractive (respectively repelling) one-sided holonomy chose a transversal $(-\eps,0]$ (respectively $[0,\eps)$) with $0$ lying on $\gamma$ and construct taming functions on collars of discs bounding the subcycle. When the germ of the one-sided holonomy is nontrivial, then we can choose the boundary corresponding boundary component of the domain $U$ of $f$ to be transverse to $S(\xi)$, otherwise we can choose the boundary of the domain to be tangent to a leaf of $S(\xi)$. 

If $\gamma$ contains a corner such that only one stable leaf of the hyperbolic singularity is part of $\gamma$, then the levelsets of $f$ near $\gamma$ can be chosen as suggested in \figref{b:levelset}. The thick curve represents a critical level of $f$ while the dashed curve corresponds to a regular level of $f$.

\begin{figure}[htb]
\begin{center}
\includegraphics{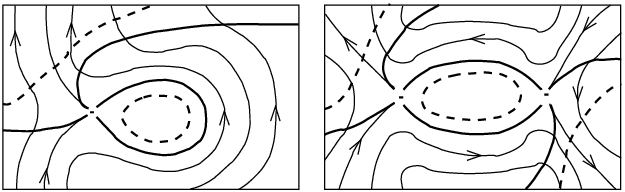}
\end{center}
\caption{\label{b:levelset}}
\end{figure}

By construction $f$ is constant along cycles and increasing along leaves of $S(\xi)$ which are not part of cycles. Singular points of $S(\xi)$ clearly are critical points of $f$. In order to show that requirement (i) of \defref{d:taming fct} is satisfied by $f$ we consider an internal leaf $\gamma_0\subset\gamma$.

Let $D_{0,1},D_{0,2}\subset S$ be discs which lie on opposite sides of $\gamma_0$ and contain no subcycle of $\gamma$ in their interior. Because $\gamma$ is an internal leaf $\ring{D}_{0,1}$ respectively $\ring{D}_{0,2}$ can not contain a stable or unstable leaf of a hyperbolic singularity on $\partial D_{0,1}$ respectively $\partial D_{0,2}$. Therefore the one-sided holonomy along $\partial D_{0,1}$ and $\partial D_{0,2}$ is well defined and by \lemref{l:neg-curv} the holonomy along $\partial D_{0,1}$ is potentially attractive if and only if the same is true for the holonomy along $\partial D_{0,2}$. Hence $f$ has a local minimum respectively maximum at every point of $\gamma_0$ when the holonomy is potentially repulsive respectively attractive. 

Using induction on the number of hyperbolic singularities in $\gamma$ we now prove requirement (iii) from \defref{d:taming fct} and \eqref{e:d+} for $f :U\lra\R$. We have already treated the case when $\gamma$ contains no hyperbolic singularity. 

Given a cycle $\gamma$ and a fixed hyperbolic singularity $x_0$ we isotope $S$ in a neighbourhood of $x_0$. We choose the isotopy such that segments of $S(\xi)$ in $S\cap S'$ which ended at $x_0$ before the perturbation are now connected be non-singular segments of $S'(\xi)$. In this way obtain a cycle $\gamma'$ on $S'$ which contains one singularity less than $\gamma$ and it may happen that $\gamma'$ is not connected. 

In order to construct an isotopy with the desired properties one moves $x_0$ away from the integral surface of $\xi$ which contains the cycle $\gamma$. When $x_0$ is part of an internal cycle or not all stable/unstable leaves of $x_0$ are contained in $\gamma$ one has to move $x_0$ into the interior of an integral surface of $\xi$ and then slightly above or below the integral surface with respect to the coorientation of $\xi$. Choosing to push upwards or downwards one can make sure that on obtains a cycle on the perturbed surface which is contained in the interior of the integral surface of $\xi$ which contains $\gamma$. \figref{b:shift} shows one particular instance of the isotopy in a neighbourhood of $x_0$. In that figure, we move $x_0$ downwards. In the left part of the figure all lines are part of $S$ while in the right part they straight line do not belong to $S'$. The cycles $\gamma$ respectively $\gamma'$ correspond to the thickened lines in the left respectively right part of \figref{b:shift}.

\begin{figure}[htb]
\begin{center}
\includegraphics{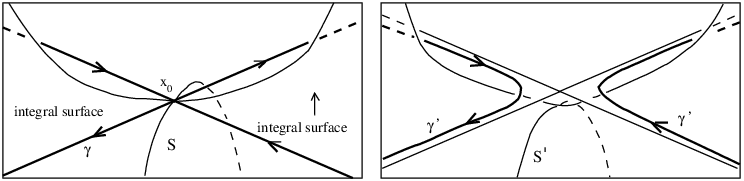}
\end{center}
\caption{\label{b:shift}}
\end{figure}


If there is a hyperbolic singularity $x_0\in\gamma$ such that $\gamma$ contains only one stable leaf of $x_0$, then $x_0$ is automatically an essential singularity on $\gamma$ and our orientation convention and the choice of the function in \figref{b:levelset} together with \lemref{l:signs in cycles on spheres} imply that the behavior of the level sets of $f$ near $x_0$ is compatible with requirement (iii) of \defref{d:taming fct}. 

In order to prove \eqref{e:d+} we perturb $S$. After an isotopy of $S$ in a neighbourhood of $x_0$ we obtain a cycle $\gamma'$ which contains one singularity less than $\gamma$ and the singularity we removed had a stable leaf which was not part of $\gamma$. We construct the function $f'$ on $U'\supset\gamma'$ as above. When $x_0$ is positive, then 
\begin{align*}
d_+(U') & = d_+(U) + 1 & N_-(U') & = N_-(U) \\
N_s(U') & = N_s(U)     & P_s(U') & = P_s(U)-1.
\end{align*}
Therefore \eqref{e:d+} holds for $U$ if and only if it holds for $U'$. If $x_0$ is negative we have to distinguish two cases: In the first case, the stable leaf of $x_0$ is the only stable leaf of a negative hyperbolic singularity intersecting the connected component of $\partial U$. Then 
\begin{align*}
d_+(U') & = d_+(U)   & N_-(U') & = N_-(U)+1 \\
N_s(U') & = N_s(U)-1 & P_s(U') & = P_s(U).
\end{align*}
If there are other stable leaves of other hyperbolic singularities of $\gamma$ which intersect the same connected component of $U$ as the stable leaf of $x_0$, then 
\begin{align*}
d_+(U') & = d_+(U)  & N_-(U') & = N_-(U)  \\
N_s(U') & = N_s(U)  & P_s(U') & = P_s(U).
\end{align*}
Again the validity of \eqref{e:d+} for $U$ follows from \eqref{e:d+} for $U'$. For the proof of \eqref{e:d+} we may assume from now on that all stable and unstable leaves of all hyperbolic singularities on $\gamma$ are contained in $\gamma$. In particular $N_s=P_s=0$ in the sequel.

Let $x_0\in\gamma$ be an essential hyperbolic singularity. We shall discuss the configuration shown in the left part of \figref{b:shift}. The other configurations can be handled in the same manner. The vertical arrow in \figref{b:shift} indicates the coorientation of $\xi$, the other arrows indicate orientations of leaves of $S(\xi)$ and $S'(\xi)$. In addition we assume that the stable leaf on the right (resp. left) hand side is connected in $\gamma\setminus\{x_0\}$ to the unstable leaf on the right (resp. left) hand side. 

In this situation $\gamma$ is split into two connected components $\gamma',\gamma''$ by the isotopy. For both connected components there is an integral disc of $\xi$ which bounds a cycle containing parts of one stable leaf of $x_0$. The two integral discs have disjoint interiors.  

Therefore there is one disc $D_b\subset S$ with well defined one-sided holonomy below the integral surface of $\xi$ and $x_0\in D_b$ and by \lemref{l:neg-curv} this holonomy is potentially attractive. There are two discs with well defined one-sided holonomy lying above the integral surface and each of the upper discs contains exactly one stable leaf of $x_0$ in its boundary while the lower disc contains both stable leaves of $x_0$ in its boundary. The one-sided holonomies along the boundary of each of the two discs  if potentially repulsive.  This is exactly the behavior prescribed by (iii) of \defref{d:taming fct}. 

We choose neighbourhoods $U',U''$ of $\gamma',\gamma''$ which satisfy \eqref{e:d+}. The relation between $d_+(U)$ and $d_+(U'),d_+(U'')$ is given by
\begin{align*}
d_+(U) & = d_+(U')+d_+(U'') & N_-(U) & = N_-(U')+N_-(U'')-1.
\end{align*}
Hence \eqref{e:d+} is true for $U$ because it is satisfied for $U',U''$. The other configurations can be handled in a similar manner. 
  
Now we assume that $x_0$ is a hyperbolic singularity such that one stable leaf is part of an internal cycle and the other one is part of a subcycle of $\gamma$ which is not internal (if there are internal subcycles, then there must be singularities with this property because $\gamma$ is connected).

Let $\gamma_{0,1},\gamma_{0,2}$ be the stable and unstable leaves of $x_0$ which are internal.  There is a disc $D_0\subset S$ whose boundary contains $\gamma_{0,1},\gamma_{0,2}$ such that the one-sided holonomy along $\partial D$ is well defined. If it is potentially attractive respectively repulsive, then $x_0$ is positive respectively negative by \lemref{l:signs in cycles on spheres}.

The remaining pair of separatrices is part of a cycle with well defined one-sided holonomy. It is potentially attractive if and only if the holonomy along $\partial D_0$ is potentially repulsive (cf. \lemref{l:signs in cycles on spheres}).

By a small isotopy we can obtain a connected cycle $\gamma'$ or two connected cycles $\gamma',\gamma''$ on the perturbed sphere $S'$ with one singularity less than $\gamma$ such that $\gamma_{0,1}, \gamma_{0,2}$ (ie. the segments lying outside of the support of the perturbation of $S$) are connected by a leaf of $S'(\xi)$ and the same is true for the other pair of separatrices of $x_0$. \figref{b:fol-taming-ex} shows a cycle which decomposes into a pair of connected cycles. The discussion above shows that $f:U\lra\R$ satisfies (iii) of \defref{d:taming fct} if the same is true for $f' : U'\lra\R$ and $f'':U''\lra\R$. 

We construct a taming function on a neighbourhood of the perturbed cycle. The following table summarizes the relations from \lemref{l:signs in cycles on spheres} between the invariants $d_+,N_-$ associated to $\gamma$ with the invariants for the perturbed cycle. 

\smallskip
\begin{tabular}{|l|c|c|}
\hline
                       & $x_0$ is positive & $x_0$ is negative \\
\hline
$\gamma$ remains connected & \begin{tabular}{c} $d_+=d_+'-1$ \\ $N_-=N_-'+1$ \end{tabular}   
& \begin{tabular}{c} $d_+=d_+'$ \\ $N_-=N_-'$ \end{tabular}  \\
\hline
$\gamma$ splits into two cycles  &\begin{tabular}{c} $d_+=d_+'+d_+''-1$ \\ $N_-=N_-'+N_-''$ \end{tabular} & \begin{tabular}{c} $d_+=d_+'+d_+''$ \\ $N_-=N_-'+N_-''-1$ \end{tabular} \\
\hline
\end{tabular}

\smallskip
Therefore \eqref{e:d+} holds for the neighbourhood $U$ of $\gamma$ and $f : U \lra \R$ has the desired properties. 

This finishes the first step in the construction of a taming function on a neighbourhood of $\Sigma_0$. If all components of $\partial U$ are transverse to $S(\xi)$, then $U_0:=U$ and $f$ tames $S(\xi)$ on $U_0$. Otherwise we iterate the above construction.

Assume we have constructed a taming function $f: U\lra\R$ and $\Gamma\subset\partial U$ is a closed leaf of $S(\xi)$ with trivial holonomy. By construction the holonomy is potentially attractive on the side of $\Gamma$ which is contained in $U$. Then there is a cylinder $S^1\times(0,1)\subset S$ such that $S(\xi)$ corresponds to the foliation by the first factor and $\overline C$ consists of two cycles $\gamma_0, \gamma_1$ such that $\gamma_0\subset U$ and $\gamma_1$ lies in the complement of $U$. We choose $C$ maximal among cylinders with these properties. Then $\gamma_1$ can not be a closed leaf with trivial holonomy. Therefore $\gamma_1$ belongs to one of the following classes. 
\begin{itemize}
\item[(i)] $\gamma_1$ is a negative elliptic singularity or a closed leaf such that the holonomy on the side which is not contained in $C$ is non-trivial and potentially repulsive. In this case it is easy to extend $f$ to a taming function on $U\cup \overline{C}$ such that \eqref{e:d+} is satisfied. 
\item[(ii)] $\gamma_1$ is a cycle containing hyperbolic singularities. If we did not yet define a taming function near $\gamma_1$, then we apply the above procedure to construct a taming function $g: V\lra\R$ on a set $V$ with $U\cap V=\emptyset$. In particular, $V$ satisfies \eqref{e:d+}. We add a constant to $g$ to ensure that $g\eing{\gamma_1}>f\eing{\Gamma}$. Then we extend $g\cup f: U\cup V\lra\R$ to a taming function on $U\cup V\cup C$. Note that $N_-(U\cup V\cup C)= N_-(U)+N_-(V)-1$. From this it follows that \eqref{e:d+} holds for $U\cup V\cup C$. 
\end{itemize}   
After finitely many steps we have constructed a taming function on a neighbourhood $U_0$ of $\Sigma_0$ with the desired properties. It is clear how to adapt the construction in the presence of birth-death type singularities.
\end{proof}

The following lemma implies that the existence of a taming function on a neighbourhood $U$ of $\Sigma$ is a property which is stable under $C^0$-small perturbations of $\xi$ if $U$ is small enough. For the statement of \lemref{l:stab} recall that for a given cycle in $S$ there is a unique integral disc of $\xi$ whose boundary is the cycle. 

\begin{lem} \mlabel{l:stab}
Let $\Sigma_0$ be a path connected component of $\Sigma(S)$ and $\tsig_0$ the union of all discs tangent to $\xi$ which bound cycles in $\Sigma_0$. There is a neighbourhood $\tsig_0\subset W\subset M$ and $\eps>0$ such that for every confoliation $\xi'$ on $M$ which is $\eps$-close (in the $C^0$-topology) there is a confoliation $\xi_c'$ on $\R^3$ which is transverse to the fibers of $\R^3\lra\R^2$ and complete as connection together with an embedding
$$
\varphi : \left(W,\xi'\eing{W}\right) \lra \left(\R^3,\xi_c'\right)
$$
such that $\varphi_*(\xi')=\xi_c'$. In particular, if $\xi'$ is a contact structure, then $\xi'\eing{W}$ is tight. 
\end{lem}

\begin{proof}
Note that the integral discs which bound a cycle depend continuously on the cycle because the integral discs are uniquely determined. On $\tsig_0$ we define an equivalence relation as follows: $x\sim y$ for $x,y\in \tsig_0$ if and only if there is a piecewise smooth path in $\tsig_0$ tangent to $\xi$ which connects $x$ and $y$. 

The space $T:=\tsig_0/\sim$ should be thought of as a directed graph: Discs bounding singular cycles and closed leaves with non-trivial holonomy correspond to vertices while edges of $T$ correspond to families of integral discs of $\xi$ which bound a maximal connected cycle in $\Sigma_0$. (Because a disc in $\tsig_0$ may be part of a bigger disc in $\tsig_0$, a point in $\tsig_0/\sim$ does not correspond to a unique cycle of $S(\xi)$ in general. This happens for example in \figref{b:fol-taming-ex}.) The orientation of an edge is induced by the coorientation of $\xi$.

$T$ is a connected tree because $\tsig_0$ is connected and $S$ is a sphere. We embed $T$ in the $y,z$-plane in $\R^3$ such that $dz$ is consistent with the orientation of the edges of $T$. 

Let $\LL$ be the foliation on $\R^3$ by straight lines parallel to the $x$-axis and $\ZZ$ the foliation by planes parallel to the $x,y$-plane. 
We replace $T$ by a family of discs tangent to $\ZZ$: For each vertex of $T$ we choose a collection of discs $D_i$ such that
\begin{itemize}
\item each $D_i$ is tangent to the leaf of $\ZZ$ containing the vertex,
\item $\cup_i D_i$ is homeomorphic to the union of integral discs in $M$ which bound the corresponding cycle in $M$ and $\cup_i D_i$ intersects the original tree $T$ in exactly one point. 
\end{itemize}

Then we connect the discs which correspond to vertices of $T$ by families of discs tangent to $\ZZ$ as prescribed by the edges of $T$, ie. by the configuration of integral discs in $M$. This is done in such a way that outside of a small neighbourhood of the discs which correspond to vertices of the tree each leaf of $\LL$ intersects at most one disc and this intersection is connected. (In the presence of some configurations of critical points on cycles in $\Sigma_0$ it may be impossible to satisfy the last requirement everywhere without violating the requirement that each leaf of $\LL$ intersects at most one disc.)

So far we have obtained an embedding $\varphi_0: \tsig_0\lra\R^3$ with $\varphi_{0*}(\xi)=\ZZ$ and the Legendrian foliation $\varphi_{0*}^{-1}(\LL)$ on $\tsig_0$. We extend this foliation to a Legendrian foliation $\LL_0$ on an open neighbourhood $\tsig$ of $\tsig_0$ and we extend the embedding $\varphi_0$ such that the extended Legendrian foliation is mapped to $\LL$, the extension of $\varphi_0$ is the desired embedding $\varphi : \tsig\lra\R^3$ 
but we still have to find the right domain and the neighbourhood $W$.

We may assume that $\tsig$ was chosen such that the intersection of each leaf of $\LL$ with $\varphi(\tsig)$ is connected and $\varphi_*(\xi)$ is transverse to $\partial_z$. By construction $\varphi_*\left(\xi\eing{\tsig}\right)$ is the kernel of the $1$-form $\alpha=dz+f(x,y,z)dy$ with $\partial_x f\ge 0$ and $f\equiv 0$ on $\tsig_0$. 

By extending $f$ to a function on $\R^3$ we can extend $\alpha$ to a $1$-form $\alpha_c$ on $\R^3$ whose kernel is a confoliation $\xi_c$ with the desired properties: If we extend $f$ to a function on $\R^3$ with $\partial_xf\ge 0$ and  $f\equiv 0$ for $|z|$ big enough, then $\xi_c$ is a complete connection. 

For each plane field $\zeta$ on $\varphi(\tsig)$ such that $\zeta$ is transverse to $\partial_z$ we define a foliation $\LL(\zeta)$ which is tangent to the projection of $\partial_x$ to $\zeta$ along $\partial_z$. There is a neighbourhood $W\subset M$ of $\tsig_0$ and $\eps>0$ with the following properties:
\begin{itemize}
\item If $\xi'$ is $\eps$-close to $\xi$, then $\varphi_*(\xi)$ is transverse to $\partial_z$.
\item For every plane field $\xi'$ which is $\eps$-close to $\xi$ there is an open set $W'$ with $\tsig_0\subset W\subset W'\subset U$ such that the intersection of $\varphi(W')$ with leaves of $\LL(\varphi_*(\xi'))$ is connected. 
\end{itemize}  

This implies the claim of the lemma: If a confoliation $\xi'$ is sufficiently close to $\xi$ in the $C^0$-topology, then we can extend $\varphi_*(\xi'\eing{W})$ by extending (as above) the confoliation $\varphi_*(\xi\eing{W'})$ along leaves of a foliation $\LL'$ of $\R^3$ by lines transverse to the planes $\{x=\textrm{const}\}$ and which coincides with $\LL$ outside of $\varphi(\tsig)$. Thus we have found a confoliation $\xi'_c$ on $\R^3$ with the desired properties.

The statement about the tightness of $\xi'\eing{W}$ follows from \propref{p:complete connection}.  
\end{proof}

Next we show that the taming functions which we have constructed on pieces of $S$ in \lemref{l:taming functions near cycles} can be combined to obtain a taming function on a given generically embedded sphere. 

\begin{prop} \mlabel{p:taming functions on spheres}
If $(M,\xi)$ is tight and $S$ is an embedded sphere such that $S(\xi)$ has isolated singularities which are either non-degenerate or of birth-death type, then $S$ admits a taming function. 
\end{prop}

\begin{proof}
We construct $f$ inductively in a finite number of steps. By \lemref{l:taming functions near cycles} we can cover the compact set $\Sigma(S)$ by a finite collection of open sets $\mathbb{U}_0=\{U_1,\ldots,U_l\}$ with $U_j\subset S$ such that there is a taming function $f_j$ on $U_j, j=1,\ldots,l$ and the sets $U_j$ are pairwise disjoint. Recall that 
\begin{equation} \label{e:d+2}
d_+(U_j)=1-N_-(U_j)-P_s(U_j)-N_s(U_j)
\end{equation}
for all $j=1,\ldots,l$. For later applications we assume that each $U_j\in\mathbb{U}_0, j=1,\ldots,l$ has the property described in \lemref{l:stab} for $\eps_j>0$.

We define a partial order $\preceq$ on $\mathbb{U}_0$ as follows: $U_{j}\preceq U_{k}$ if and only if either $j=k$ or $U_{k}$ has a boundary component which bounds a disc in $S$ not containing $U_k$ and a leaf of the characteristic foliation coming from $U_j$ enters $U_k$ through this boundary component. 

By definition every cycle of $S(\xi)$ which intersects $U_j$ is completely contained in $U_{j}$. This implies that $U_{j}\preceq U_{k}$ and $U_{k}\preceq U_{j}$ if and only if $j=k$ and there is a set $U_j\in\mathbb{U}_0$ which is minimal with respect to $\preceq$. 
All connected components of $\partial U_j$ are transverse to $S(\xi)$ and the characteristic foliation points outwards along the boundary. Moreover, \eqref{e:d+2} implies $d_+(U_j)=1$. 

Let $f_j$ be a taming function on $U_j$ and consider the basin $B(U_j)$ of $U_j$. According to \lemref{l:leg poly} the closure of $B(U_j)$ is covered by a Legendrian polygon $(Q_j,V_j,\alpha_j)$. We consider four cases which correspond to the conclusion of \lemref{l:edges of poly}. Let us assume that there are no birth-death type singularities. This assumption will be removed below. 

{\em Case (o)}: $Q_j$ has more boundary components than $U_j$. This means that in the construction of $(Q_j,V_j,\alpha_j)$ in \lemref{l:leg poly} we did attach $1$-handles to $U_j$ (recall that we used $U_j$ as a starting point for the construction of $Q_j$). 

Let $\gamma_j$ be the stable leaf of a hyperbolic singularity $h_j$ such that $\gamma_j$ leaves $U_j$ and $h_j$ is a corner in a cycle $\eta$. This cycle is contained in one of the sets $U_{i(\eta)}\in\mathbb{U}_0$. Let $f_i$ be a taming function on $U_{i(\eta)}$. Now we extend $f_j$ to a taming function on a neighbourhood $U'_j$ of $\gamma_j\cup U_j\cup U_{i(\eta)}$ (it may be necessary to add a sufficiently large constant to $f_{i(\eta)}$).

The extended function tames the characteristic foliation on its domain and the new boundary component of $U'_j$ can be chosen transverse to $S(\xi)$. By construction 
\begin{align*}
N_-\left(U'_j\right) & = N_-(U_{i(\eta)}) \\
P_s\left(U'_j\right) & = \left\{ 
\begin{array}{ll} P_s(U_{i(\eta)})-1 & \textrm{ if }h_j \textrm{ is positive} \\
              P_s(U_{i(\eta)})   & \textrm{ if }h_j \textrm{ is negative} \\
\end{array}\right. \\
N_s\left(U'_j\right) & = \left\{ 
\begin{array}{ll} N_s(U_{i(\eta)})   & \textrm{ if }h_j \textrm{ is positive} \\
              N_s(U_{i(\eta)})-1   & \textrm{ if }h_j \textrm{ is negative.} \\
\end{array}\right.
\end{align*}
This implies $d_+(U'_j) = 1 - N_-(U'_j)-P_s(U'_j)-N_s(U'_j)$. 

In the following cases we consider a fixed connected component $\Gamma\subset \partial Q_j$ which was not covered in case (o).

{\em Case (i)}: $\alpha_j(\Gamma)$ is an elliptic singularity and $\alpha_j(Q_j)$ is a neighbourhood of $x$ or $\alpha_j(\Gamma)$ is a cycle and $\alpha_j(Q_j)$ is a one-sided neighbourhood of that cycle.

Let us start with the case when $\alpha_j(\Gamma)$ is an elliptic singularity. Because it is attractive, it must be negative and it is contained in $U_{i(\Gamma)}$ with $i(\Gamma)\neq j$. One can easily extend $f_j$ to a taming function on the union $U'_j$ of $U_j\cup U_{i(\Gamma)}$ with all leaves passing through $\Gamma$. Obviously \eqref{e:d+2} holds for $U'_j$.

If $\alpha_j(\Gamma)$ is a closed leaf or a cycle, then $\alpha_j(\Gamma)$ belongs to one of the sets $U_{i(\Gamma)}$ with $i(\Gamma)\neq j$. After eventually adding a constant to the taming function on $U_{i(\Gamma)}$ one obtains a taming function on the union of the flow lines leaving $U_j$ through $\Gamma$ with $U_j$ and $U_{i(\Gamma)}$. As before we denote the new domain by $U'_j$. From 
\begin{align*}
N_-\left( U'_j \right) & = N_-(U_{i(\Gamma)})-1\\
P_s\left( U'_j \right) & = P_s(U_{i(\Gamma)})  \\
N_s\left( U'_j \right) & = N_s(U_{i(\Gamma)}).
\end{align*}   
it follows that $d_+(U'_j)  = 1 - U_-( U'_j) -P_s(U'_j)-N_s(U'_j)$.

{\em Case (ii)}: $\alpha_j(\Gamma)$ contains an elliptic singularity such that $\alpha_j(Q_j)$ is not a neighbourhood of this singularity or there is $v_j\in V_j\cap\Gamma$ such that $\gamma_{v_j}$ is a cycle of $S(\xi)$ and $\alpha_j(Q_j)$ is not a one sided neighbourhood of $\gamma_{v_j}$ or 

According to \propref{p:basins in spheres} there is a positive pseudovertex $x$ on $\alpha_j(\Gamma)$ such that $\alpha_j(Q_j)$ is not a neighbourhood of $x$. Let $\eta$ be the stable leaf of $x$ which is not contained in $\alpha_j(Q_j)$. The $\alpha$-limit set of $\eta$ is contained in a set $U_{i(\eta)}$ while $x\in U_{i(x)}$. We obtain a taming function on the union of $U'_j$ of  $U_j\cup U_{i(\eta)}\cup U_{i(x)}$ with a neighbourhood of the stable leaves of $x$ (after adding a constant to the taming function on $U_{i(x)}$). 

Because $x$ is positive the requirements in the definition of taming functions are satisfied. Moreover, we can choose the domain $U'_j$ of the taming function such that its the new boundary component is transverse to $S(\xi)$. The equality $d_+(U'_j)=1-N_-(U'_j) - P_s(U'_j)- N_s(U'_j)$ follows from 
\begin{align*}
N_-\left(U'_j\right) & = N_-\left(U_{i(\eta)}\right) \\ 
P_s\left(U'_j\right) & = P_s\left(U_{i(\eta)}\right) \\
N_s\left(U'_j\right) & = N_s\left(U_{i(\eta)}\right) 
\end{align*}
and the fact that $x$ is positive.

{\em Case (iii)}: (0)-(ii) do not hold for $(Q_j,V_j,\alpha_j)$. Then $\alpha_j$ identifies edges on $\Gamma$ by \lemref{l:edges of poly}. We shall use the notation from the proof of that lemma. 

Let $e_1,\ldots,e_l$ be edges on $\Gamma$ which are obtained as in the proof of \lemref{l:edges of poly}. The cycle $\eta\subset\alpha_j(e_1)\cup\ldots\cup\alpha_j(e_l)$ is contained in $U_{i(\eta)}\in\mathbb{U}_0$ and we denote the stable leaves of the pseudovertices on $\eta$ which are not part of $\eta$ by $\sigma_1,\ldots,\sigma_l$. Let $U_j'$ be the union of $U_j\cup U_{i(\eta)}$ with neighbourhoods of $\sigma_1,\ldots,\sigma_l$. No other stable leaves of hyperbolic singularities enter $U_{i(\eta)}$ and all pseudovertices on $\eta$ are negative. After we add a sufficiently big constant to $f_{i(\eta)}$ we obtain a taming function $f'_j$ on $U'_{j}$. By construction we have
\begin{align*}
N_-\left(U'_j\right) & = N_-\left(U_{i(\eta)}\right) \\ 
P_s\left(U'_j\right) & = P_s\left(U_{i(\eta)}\right) \\
N_s\left(U'_j\right) & = N_s\left(U_{i(\eta)}\right)-1.
\end{align*}
These equalities immediately imply \eqref{e:d+}. 

We have now considered all cases occurring in \lemref{l:edges of poly}. Next we remove the assumption that there is not birth-death type singularity. Assume that in the step above we encounter a birth-death type singularity $x$. Then $x$ is contained in a set $U_{i(x)}$ from $\mathbb{U}_0$. In an intermediate step we extend $f$ to the union $U_j^{int}$ of $U\cup U_{i(x)}$ with the leaves of $S(\xi)$ which connect $U_{i(x)}$ to $U$. Then we continue as before with $U_j^{int}$ instead of $U_j$. 

Now we remove $U_j$ together with all $U_i$ which are contained in $U'_j$ from the collection $\mathbb{U}_0$ and we add $U'_j$. This yields a new collection of of subsets $\mathbb{U}_1$ such that on each domain in $\mathbb{U}_1$ we have a defined a taming function. Notice that the number of sets in $\mathbb{U}_1$ is strictly smaller than the number of sets in $\mathbb{U}_0$. 

We iterate the procedure after replacing $\mathbb{U}_0$ with $\mathbb{U}_1$. After finitely many steps we obtain a taming function on $S$.
\end{proof}


So far we have established the existence of a taming function on embedded spheres such that $S(\xi)$ has only non-degenerate or birth-death type singularities. Now we consider an embedding of a family of spheres $S^2\times[0,1]$ in $M$ and a $C^0$-approximation of $\xi$ by a confoliation $\xi'$. After a $C^\infty$-small perturbation of $S^2\times[0,1]$ each sphere $S_t=S^2\times\{t\}$ becomes generic. We want to show that the characteristic foliation $S_t(\xi')$ admits a taming function if the confoliation $\xi'$ is close enough to $\xi$ in the $C^0$-topology.

\begin{prop} \mlabel{p:deform taming fct}
There is a $C^0$-neighbourhood of $\xi$ such that for every confoliation $\xi'$ in that neighbourhood $S_t(\xi')$ admits a taming function for all $t\in[0,1]$ if $S_t$ is generic with respect to $\xi'$ for all $t$.

If $\xi'$ is a contact structure, then $S_t(\xi')$ admits a taming function which is strictly increasing along all leaves of $S_t(\xi')$.  
\end{prop}

\begin{proof}
We show that if $\xi'$ is close enough to $\xi$ in the $C^0$-topology and $S_t(\xi)$ has only non-degenerate singularities or singularities of birth death type, then the iteration process used for the construction of a taming function in \propref{p:taming functions on spheres} can be carried out to yield a taming function for $S_t(\xi')$. For this we first reconsider the proof of \propref{p:taming functions on spheres} in order to show the existence of $\eps>0$ with the desired properties for a fixed sphere $S_t$ and then we argue that $\eps$ can be chosen independently from $t\in[0,1]$.  

Recall that in the proof of \propref{p:taming functions on spheres} we required that all sets $U_j\in\mathbb{U}_0$ appearing in the initial stage of the construction are contained in a set $W_j$ with the stability property described in \lemref{l:stab} for $\eps_j>0$: The restriction of $\xi'$ to $W_j$ is tight when $\xi'$ is $\eps_j$-close to $\xi$. 

Moreover, we chose the $U_j$ such that each smooth segment in $\partial U_j$ is transverse to $S(\xi)$. This remains true when $\xi'$ is $\eps_j$-close to $\xi$ when $\eps_j>0$ is small enough. The iteration process in the proof of \propref{p:taming functions on spheres} stops after finitely many steps and we choose $\eps>0$ so small that each smooth segment contained in the boundary of a set in $\mathbb{U}_0,\mathbb{U}_1,\ldots$ is transverse to $S(\xi')$ when $\xi'$ is $\eps$-close to $\xi$. This requirement ensures also that the combinatorics of the extensions of $f$ is the same for $S_t(\xi)$ and $S_t(\xi')$. 

It remains to show that we can choose $\eps>0$ independently from $t\in[0,1]$. For this note that $\Sigma=\cup_t\Sigma(S_t)$ is compact. Thus a finite number of sets $W_j$ obtained from \lemref{l:stab} suffice to cover $\Sigma$. If $\tau$ is sufficiently close to $t$, then $S_\tau(\xi)$ is very close to $S_t(\xi)$ in  the $C^\infty$-topology and the combinatorics of extensions of a taming function for $S_t(\xi)$ and $S_\tau(\xi)$ coincide, ie. we connect subsets $U_j(t)$ of $S_t$ which are very close to subsets $U_j(\tau)$ of $S_\tau$ in the same order (with the possible but irrelevant exception of birth-death type singularities). 

When the above procedure for the choice of $\eps$ for $S_t$ yields $\eps_t>0$, then $\eps_t/2$ has the desired property with respect to the characteristic foliation on $S_{\tau'}$ when $\tau'$ is close enough to $t$. Since $[0,1]$ is compact, this proves the claim. 
\end{proof}

\subsubsection{Proof of \thmref{t:tight on balls}}   
For the proof of \thmref{t:tight on balls} we combine the results from the previous sections with results from \cite{giroux2}. 

Let $B\subset B_1\subset M$ be an embedded closed ball in a manifold $M$ with a tight confoliation $\xi$. We assume that the interior of $B_1$ contains points where $\xi$ is a contact structure since otherwise \thmref{t:tight on balls} follows immediately from \lemref{l:stab}. Moreover, we assume that $\partial B_1$ is generic. 

Let $B_0$ be a ball in the contact region whose characteristic foliation has exactly two singular points and the leaves of the characteristic foliation connect the two singularities. The existence of such a ball follows from the fact that every contact structure is locally equivalent to the standard contact structure $\mathrm{ker}(dz+xdy)$ on $\R^3$. Moreover, there is an open neighbourhood of $\xi|_{B_0}$ such that every confoliation in this neighbourhood is tight on $B_0$.

Let $\xi'$ be a contact structure on $B_1$. If $\xi''$ is a contact structure and  sufficiently close to $\xi'$ in the $C^\infty$-topology, then $\xi'|B$ is diffeomorphic to the restriction of $\xi''$ to a closed ball in $B_1$. Therefore it is enough to prove \thmref{t:tight on balls} for generic perturbations. 

We fix a generic identification $B_1\setminus\ring{B}_0\simeq S^2\times[0,1]$ such that $\partial B_i=S_i, i=0,1$. Because the confoliation $\xi$ is assumed to be tight, $S_t(\xi)$ can be tamed for all $t$. By \propref{p:deform taming fct} this remains true for generic confoliations $\xi'$ which are sufficiently close to $\xi$ in the $C^0$-topology.

Recall that an embedded surface in a contact manifold is called {\em convex} if there is a vector field transverse to the surface such that the flow of the vector field preserves the contact structure. According to \cite{giroux} convexity is a $C^\infty$-generic property, so we may assume that $\partial B_0$ and $\partial B_1$ are convex with respect to $\xi'$.

We will show that $\xi'$ can be isotoped on $S^2\times[0,1]$ relative to the boundary such that all leaves of the product foliation on $S^2\times[0,1]$ become convex with respect to the isotoped contact structure. Since $\partial B_0$ is convex and $\xi'$ is tight on a neighbourhood of $\partial B_0$ this implies that $\xi'|_B$ is tight by Theorem 2.19 in \cite{giroux2} (and the gluing result in \cite{colin}). 

In order to prove the existence of the desired isotopy of $\xi'$ we use the following lemma. Our formulation is a slight modification of Lemma 2.17 in \cite{giroux2} in the case $F\simeq S^2$.
\begin{lem} \mlabel{l:make convex}
Let $(M,\xi')$ be a contact manifold. Assume that the characteristic foliation on each sphere $S_t$ from the family $S^2\times[0,1]\subset M$ admits a taming function and $S_0,S_1$ are convex. Then there is a contact structure $\xi''$ such that 
\begin{itemize}
\item $\xi'$ and $\xi''$ are isotopic relative to the boundary and
\item the characteristic foliation of $\xi''$ on $S_t$ has exactly $\chi(S)=2$ singular points and  $S_t$ is convex with respect to $\xi''$ for all $t\in[0,1]$.
\end{itemize}
\end{lem}
The original statement of Giroux of this lemma contains tightness as an assumption. However the proof of Lemma 2.17 of \cite{giroux2} requires only properties of the characteristic foliation on $S_t,t\in[0,1]$ which follow from the existence of taming functions.
 
More specifically, the proof of Lemma 2.17 in \cite{giroux} yields a proof of \lemref{l:make convex} after the following modification: As we have already explained we may assume that the characteristic foliation of $\xi'$ on $S_t$ can also be tamed for all $t\in[0,1]$ by \propref{p:deform taming fct}. Moreover, because $\xi'$ is a contact structure, the taming functions are strictly increasing along leaves of the characteristic foliation. Therefore the following statements hold:
\begin{enumerate}
\item There is no closed cycle on $S\times\{t\}, t\in[0,1]$. 
\item The graph $\Gamma_t^+$ ($\Gamma_t^-$) on $F\times\{t\}$ formed by positive (negative) singular points and stable (unstable) leaves of positive (negative) hyperbolic singularities is a tree. 
\end{enumerate}
Using these two observations one obtains a proof of \lemref{l:make convex} from the proof of Lemma 2.17 in \cite{giroux2}. This finishes the proof of \thmref{t:tight on balls}.


\section{Overtwisted stars} \mlabel{s:discussion}

In this section we introduce overtwisted stars. Their definition is given in the next section and it is motivated by the discussion of the confoliation $(T^3,\xi_T)$ in \secref{s:example}. The absence of overtwisted stars in a tight confoliations implies all Thurston-Bennequin inequalities and we show that symplectically fillable confoliations do not admit overtwisted stars (in addition  to the fact that they are tight).

\subsection{Overtwisted stars and the Thurston-Bennequin inequalities} \mlabel{s:ot star}

As we have already mentioned the point where Eliashberg's proof of the Thurston-Bennequin inequalities fails in the case of tight confoliations is the following: Given an embedded surface $F$ and a tight confoliation $(M,\xi)$, there may be leaves of $F(\xi)$ which come from an elliptic singularity and accumulate on closed leaves $\gamma$ (or on quasi-minimal sets) of the characteristic foliation such that $\gamma$ is part of the fully foliated set of $\xi$. Even if all singular points on $\partial B(x)$ have the same sign it may be impossible to construct a disc from $B(x)$ which has the properties of the disc $D$ appearing in \defref{d:tight confol}.

This suggests the following definition of overtwisted stars on generically embedded surfaces $F$.

\begin{defn} \mlabel{d:overtwisted star}
An overtwisted star in the interior of a generically embedded compact surface $F\not\simeq S^2$ is the image of a Legendrian polygon $(Q,V,\alpha)$ with the following properties.
\begin{itemize}
\item[(i)] $Q$ is homeomorphic to a disc and $\alpha(\partial Q)$ contains singularities of $F(\xi)$.
\item[(ii)] All singularities of $F(\xi)$ on $\alpha(\partial Q\setminus V)$ have the same sign. There is a single singularity in the interior of $\alpha(Q)$; it is elliptic and its sign is opposite to the sign of the singularities on $\alpha(\partial Q)$.
\item[(iii)] If $\gamma_v$ is a cycle, then it does not bound an integral disc of $\xi$ in $M$.
\end{itemize}
\end{defn}
The torus shown \figref{b:starfish} contains two overtwisted stars. Note that the polygon is not required to be injective. Requirement (i) implies that either $V\neq\emptyset$ or $\alpha(\partial Q)$ contains an elliptic singularity of $F(\xi)$ and we may assume that this singularity is contained in $H(\xi)$. (Note that the elliptic singularity cannot lie in the interior of $M\setminus H(\xi)$. After a small perturbation and by \lemref{l:cut} the elliptic singularity lies in $H(\xi)$). In particular discs with the properties of $D$ in \defref{d:tight confol} are not overtwisted stars. 

If $\xi$ is a contact structure and $F\subset M$ is a generically embedded closed surface containing an overtwisted star $(Q,V,\alpha)$, then $\xi$ cannot be tight since $\xi$ is convex by the genericity assumption (therefore all $\gamma_v, v\in V$ are cycles) and has a homotopically trivial dividing curve (this terminology is standard in contact topology; because we shall not really use it we refer the reader to \cite{giroux} or \cite{honda}). This argument does not apply when $F\simeq S^2$. Since the definition of tightness in \defref{d:tight confol} can be applied efficiently to spheres and discs, the exceptional role of spheres in \defref{d:overtwisted star} will not play a role.

The following theorem is proved following Eliashbergs strategy from \cite{El} and \thmref{t:discs and spheres}.

\begin{thm} \mlabel{t:no stars imply tb}
Let $(M,\xi)$ be an oriented tight confoliation such that no compact embedded oriented surface contains an overtwisted star and $(M,\xi)$ is not a foliation by spheres. 

Every embedded surface $F$ whose boundary is either empty or positively transverse to $\xi$ satisfies the following relations. 
\begin{itemize}
\item[a)] If $F\simeq S^2$, then $e(\xi)[F]=0$. 
\item[b)] If $\partial F=\emptyset$ and $F\not\simeq S^2$, then $|e(\xi)[F]|\le-\chi(F)$.
\item[c)] If $\partial F\neq\emptyset$ is positively transverse to $\xi$, then $\selfl(\gamma,[F])\le -\chi(F)$.
\end{itemize}
\end{thm}

\begin{proof}
The claim a) was already covered in \thmref{t:discs and spheres}. For the proof of b) and c) we may assume that $F$ is a generic representative of the homology class $[F]$ which is incompressible (this means that the map $\pi_1(F)\lra\pi_1(M)$ which is induced by the inclusion $F\hookrightarrow M$ is injective).  Recall that if $\partial F$ is positively transverse to $\xi$, then $F(\xi)$ points out of $F$ along $\partial F$. Recall that 
\begin{equation*}
\chi(F)-e(\xi)[F]=2(e_--h_-)
\end{equation*}
by \eqref{e:chie}. If there is no negative elliptic singularity, then this 
implies $-e(\xi)[F]\le-\chi(F)$. If there is a negative elliptic singularity $x$, then we shall use the absence of overtwisted stars to eliminate $x$ without creating new negative elliptic singularities. Let $D_x$ be the maximal open disc in $F$ such that
\begin{itemize}
\item $\partial D_x=\overline{D}_x\setminus D_x$ is a cycle of $F(\xi)$ and
\item $x$ is the only singularity of $F(\xi)$ in the interior of $D$. 
\end{itemize} 
Unless $D_x\neq\emptyset$ there is an integral disc $D_x'$ of $\xi$ whose boundary is $\partial D_x$ because $\xi$ is tight. Moreover, the intersection of the interior of $D_x'$ with $F$ consists of homotopically trivial curves in $F$ (otherwise we get a contradiction to the incompressibility of $F$).

Thus we can cut $F$ using \lemref{l:cut}, \lemref{l:cut2} and \lemref{l:cut3} so that the resulting surface $F'$ is the union of spheres and a surface which is diffeomorphic to $F$ and incompressible. Because $e(\xi)[S]=0$ for embedded spheres $S$ we can ignore the spherical components and we denote the remaining surface by $F'$. It follows 
that $e(\xi)[F]=e(\xi)[F']$.   

If we used \lemref{l:cut2} or \lemref{l:cut3}, then we have reduced the number of negative elliptic singularities by one. Note that if we have applied \lemref{l:cut3}, then $F'$ might contain a circle of singularities. This means that $F'$ is non-generic near that circle. Since this circle is isolated from the rest of $F'$ by closed leaves of $F'(\xi)$ and the singularities on this circle do not contribute to $e(\xi)[F']$ or $\chi(F')$, these singularities will play no role in the following. Therefore we can pretend that $F'$ is generic and eliminate the remaining negative elliptic singularities.  

If we used \lemref{l:cut}, then $F'$ contains a negative elliptic singularity $x'$. By construction $x'$ lies in $H(\xi)$. In the following we shall denote $x'$ again by $x$. 

The basin of $x$ is covered by a Legendrian polygon $(Q',V',\alpha')$ on $F'$. By the maximality property of $D_x$ the boundary of $Q'$ is not mapped to a cycle of $F'(\xi)$. If $\partial Q'$ has more than one connected component, then there is a hyperbolic singularity $y$ on $\alpha'(\partial Q')$ which is the corner of a cycle $\gamma_y$. If $y$ is negative, then we can eliminate the pair $x,y$. 

Now assume that $y$ is positive. If $\gamma_y$ does intersect $H(\xi)$, then we can perturb $F'$ in a small neighbourhood of a point on the cycle such that $y$ is no longer part of a cycle after the perturbation. If $\gamma_y$ does not intersect $H(\xi)$, then we push a part of the cycle into $H(\xi)$ by an isotopy of $F'$ without introducing new singularities of the characteristic foliation.

The isotopy is constructed as follows. Let $L$ be the maximal connected integral surface of $\xi$ which contains the cycle through $y$. We choose a simple curve $\sigma$ tangent to $\xi$ which connects the cycle to $H(\xi)$ and is disjoint from $F'$. This curve can be chosen close to the stable leaf of $y$ which is connected to $x\in H(\xi)$. We choose a vector field $X$ tangent to $\xi$ with support in a small neighbourhood of $\sigma$ such that $\sigma$ is a flow line of $X$ and $F'$ is transverse to $X$. We use the flow of $X$ to isotope $F'$ such that all unstable leaves of $y$ are connected to $H(\xi)$ after the isotopy. Since $X$ is transverse to $F'$ and tangent to $X$ the isotopy creates no new singular points of the characteristic foliation. \figref{b:slide} shows $L$ together with a part of the intersection $F'\cap L$. The curve $\sigma$ is represented by the thickened line while the shaded disc represents another part of $H(\xi)$ or non-trivial topology of $L$. 

\begin{figure}[htb]
\begin{center}
\includegraphics{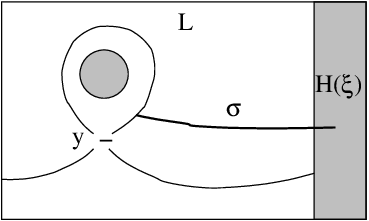}
\end{center}
\caption{\label{b:slide}}
\end{figure}

By this process we modified the basin of $x$ and the surface. Note that there are finitely many hyperbolic singularities on $F$ and the procedure described above does not create new ones. Therefore finitely many applications lead to a surface $F''$ with $e(\xi)[F]=e(\xi)[F'']$ such that the hyperbolic singularities of $F''(\xi)$ are also hyperbolic singularities of $F(\xi)$ and the basin of $x$ is homeomorphic to a disc. Also, the number of negative elliptic singularities did not increase. Note that $F''$ is not a sphere because $F''$ and $F$ have the same genus. Moreover, $F''$ has the following properties.

The basin of $x$ is covered by a Legendrian polygon $(Q'',V'',\alpha'')$ on $F''$ such that $Q''$ is a disc and $\alpha''(Q'')$ is not an elliptic singularity or a cycle of $F''(\xi)$. If necessary, we eliminate all elements of $v''$ with the property that $\gamma_{v''}$ is null homotopic in $F''$. 

Now the assumption of the theorem implies that $\partial Q''$ contains a negative pseudovertex. By \lemref{l:elim} we can isotope $F''$ to a surface containing less negative elliptic singularities than $F$ respectively $F''$. After finitely many steps we have eliminated all negative elliptic singularities. This finishes the proof of c) and one of the inequalities in b). The remaining inequality  in b) can be proved by eliminating all positive elliptic singularities. 
\end{proof}



\subsection{Overtwisted stars and symplectic fillings} \mlabel{s:symp fill ot star}

In this section we show that symplectically fillable confoliations do not admit overtwisted stars. In the proof we $C^0$-approximate a confoliation by another confoliation (cf. \thmref{t:El-Th approx}). Several techniques used in the proof are adaptations of constructions in \cite{confol}. Other useful references are \cite{pet} (where the proofs of Lemma 2.5.1 c) and Lemma 2.5.3 from \cite{confol} are carried out) and \cite{etnyre-notes}. For later use we summarize the proof of a lemma used to show \thmref{t:El-Th approx}.

\begin{lem}[Lemma 2.5.1 c) in \cite{confol}] \mlabel{l:reminder on holonomy approx}
Let $\gamma$ be a simple closed curve in the interior of an integral surface $L$ of $\xi$. If $\gamma$ has sometimes attractive holonomy, then in every $C^0$-neighbourhood of $\xi$ there is a confoliation $\xi'$ which 
\begin{itemize}
\item[(i)] is a contact structure on a neighbourhood of $\gamma$ and
\item[(ii)] coincides with $\xi$ outside a slightly larger neighbourhood. 
\end{itemize}  
\end{lem}

\begin{proof}
We only indicate the main stages of the construction. Fix a neighbourhood $V\simeq S^1_x\times[-1,1]_y\times[-1,1]_z$ and coordinates $x,y,z$ such that the foliation by the second factor is Legendrian, $S^1\times[-1,-1]\times\{0\}\subset L$ and $S^1\times\{(0,0)\}$ corresponds to $\gamma$. We assume that $\gamma$ has sometimes attractive holonomy. As in Lemma 2.1.1 of \cite{pet} the coordinates can be chosen such that 
\begin{itemize}
\item $\xi$ is defined by the $1$-form $\alpha=dz+a(x,y,z)\,dx$ with $\partial_ya\le 0$ and
\item  
there are sequences $\zeta_n'<0<\zeta_n$ converging to zero such that $a(x,0,\zeta'_n)<0<a(x,0,\zeta_n)$ for all $x$.
\end{itemize}
At this point we use the assumption that the holonomy along $\gamma$ is sometimes attractive. We fix a pair $\zeta',\zeta$ of numbers from the sequences $(\zeta_n),(\zeta')_n$. 

According to Lemma 2.2.1 in \cite{pet} and Lemma 2.5.3 in \cite{confol} there is a diffeomorphism $g: [-1,1]\lra[-1,1]$ such that 
\begin{itemize}
\item[(i)] $g$ is the identity outside of $V:=(\zeta',\zeta)$ and
\item[(ii)] $g'(z)a(x,0,z)<a(x,0,g(z))$ for all $(x,0,z)\in S^1\times\{0\}\times V$.
\end{itemize}
It follows that $g$ converges uniformly to the identity as $\zeta,\zeta'\to 0$, but no claim is made with respect to the $C^1$-topology. The graph of $g$ is given in \figref{b:graph-f} (cf. \cite{pet}). The parameters $a,b$ with $\zeta'<a<0<b<\zeta$ are chosen such that $a(x,0,z)\neq 0$ for $z\in[\zeta',a]\cup[b,\zeta]$. 

\begin{figure}[htb]
\begin{center}
\includegraphics{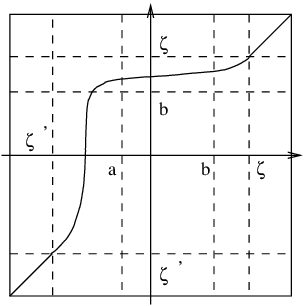}
\end{center}
\caption{\label{b:graph-f}}
\end{figure}

 In order to obtain the desired confoliation in a $C^0$-neighbourhood of $\xi$, one proceeds as follows.

{\em Step 1:} Replace $\xi$ on $S^1\times[-1/2,-1/4]\times V$ by the push forward of $\xi$ with the map $G$ which is defined by 
$$
G(x,y,z):=(x,y,u(y)g(z)+(1-u(y))z)
$$
where $u$ is a smooth non-negative function on $[-1/2,-1/4]$ such that $u\equiv 0$ near $-1/2$ and $u\equiv 1$ near $-1/4$. We extend $G$ to $M\setminus(S^1\times[-1/4,1/2]\times V)$ by the identity. As $\zeta,\zeta'\to 0$ the corresponding diffeomorphism $G$ converges to the identity uniformly but not with respect to the $C^1$-topology in general. Therefore $G_{*}(\xi)$ might not be $C^0$-close to $\xi$ on $S^1\times[-1/2,-1/4]\times V$. This will be achieved in the third step (at this point we follow the exposition on \cite{pet} closely). In the following step we replace the confoliation on $S^1\times[-1/4,1/2]\times V$. 

The dashed respectively the solid lines in \figref{b:reminder} show the characteristic foliations of $\xi'$ on neighbourhoods of $\gamma$ in $\{y=-1/4\}$ respectively on $\{y=1/2\}$ using dashed respectively solid lines in the simple case when $\gamma$ has attractive holonomy.   
\begin{figure}[htb] 
\begin{center}
\includegraphics{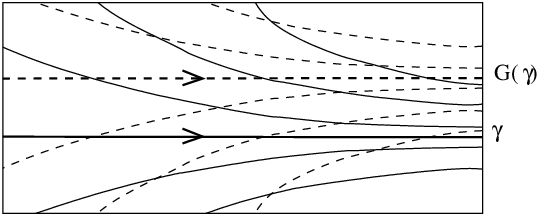}
\end{center}
\caption{\label{b:reminder}}
\end{figure}

{\em Step 2:} We extend $G_{*}(\xi)$ to a confoliation $\xi''$ on $M$ such that $\partial_y$ remains Legendrian: The plane field $\xi''$ rotates around the foliation $S^1\times[-1/4,1/2]\times V$ such that the characteristic foliation on $S^1\times\{-1/4,1/2\}\times V$ coincides with the characteristic foliation of $F_{n*}(\xi)$ on these annuli. This is possible by (ii) using the interpretation of the confoliation condition mentioned in \secref{ss:confolmeaning} (cf. \figref{b:reminder}). Note that $\xi''$ is a contact structure on the interior of $S^1\times[-1/4,1/2]\times V=:\widetilde{V}$.    

{\em Step 3:} We want to construct a diffeomorphism $\phi$ of $M$ with support in $V$ such that $\phi_{*}\xi''$ is $C^0$-close to $\xi$. For this one has to choose $V$ more carefully. This is carried out on p. 31--33 of \cite{pet}. The argument can be outlined as follows; cf. p.~16 in \cite{pet}: Assume that $r$ is chosen such that $V\subset[-r/2,r/2]$ and $\xi$ is $\eps$-close to the horizontal distribution on $S^1\times[-1,1]\times[-r,r]$. As we already mentioned $\xi''$ might be very far away from the horizontal distribution. Choose a very small number $\delta>0$ and a diffeomorphism $\varphi : [-r,r]\lra[-r,r]$ such that $\varphi([-r/2,r/2])\subset[-\delta,\delta]$. Then the push forward of the restriction of $\xi''$ to $S^1\times[-1/2,1/2]\times[-r,r]$ is $3\eps$-close to the horizontal distribution. One has to extend $\varphi$ such that this property is preserved.  
\end{proof}

We will need not only the statement of the lemma, but also the construction outlined in the proof since we need to understand how this modification of $\xi$ near a curve $\gamma$ with sometimes attractive holonomy affects the presence of overtwisted stars on embedded surfaces in $M$. The third step of the above proof is of course irrelevant at this point. 

\figref{b:near-tentacle} shows $F(\xi'')$ near a closed curve of $F(\xi'')$ in an embedded surface $F$ transverse to $\gamma$ after the second step of the proof of \lemref{l:reminder on holonomy approx}. The dot in the center of the figure represents $F\cap\gamma$ while the left inner rectangle represents the support of $G$. Finally, $\xi''$ is a contact structure in the inner rectangle on the right (this rectangle corresponds to the region $\widetilde{V}\cap F$ in the proof of \lemref{l:reminder on holonomy approx}). Recall that the characteristic foliation $F(\xi)$ was nearly horizontal in the region shown in \figref{b:near-tentacle}.

\begin{figure}[htb]
\begin{center}
\includegraphics{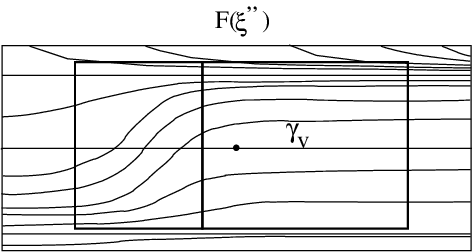}
\end{center}
\caption{\label{b:near-tentacle}}
\end{figure}

Note that if $\gamma$ even has non-trivial infinitesimal (or only attractive) holonomy, then the statement of \lemref{l:reminder on holonomy approx} can be sharpened in the sense that the lemma remains true for $C^\infty$-neighbourhoods of $\xi$ because the function $g : [-1,1]\lra[-1,1]$ can be chosen $C^\infty$-close to the identity. In the following we will consider only $C^0$-approximations. This allows us to choose the approximation of $\xi$ more freely. In particular we can preserve qualitative features of the characteristic foliation on surfaces transverse to $\gamma$.  

\begin{lem} \mlabel{l:approx with control1}
Let $\xi$ be a $C^k$-confoliation, $k\ge 1$, and $\gamma$ a simple Legendrian segment such that both endpoints of $\gamma$ lie in the contact region and $\gamma$ intersects $F$ transversely and at most once. 

Then every $C^k$-neighbourhood of $\xi$ contains a confoliation $\xi'$ such that $\xi'=\xi$ outside a neighbourhood of $\gamma$ and $\xi'$ is a contact structure on a neighbourhood of $\gamma$. Moreover, $F(\xi)=F(\xi')$. 
\end{lem}

\begin{proof}
The case $\gamma\cap F=\emptyset$ corresponds to Lemma 2.8.2. in \cite{confol}, the case $\gamma\cap F=\{p\}$ is very similar and only this case uses the assumption that both endpoints of $\sigma$ lie in $H(\xi)$. 
\end{proof}

The following lemma is standard in the setting of foliations: One can thicken a closed leaf to obtain a smooth foliation which is close to the original one and contains a family of closed leaves. Once there is such a family, one can modify the foliation such that a compact leaf whose holonomy was never sometimes attractive on one sides has sometimes attractive holonomy one one side after the modification. 

The main difficulty in the context of confoliations is the fact that now compact leaves of $\xi$ may have boundary.  

\begin{lem} \mlabel{l:thickening}
Let $(M,\xi)$ be a manifold with confoliation, $L\subset M$ a compact embedded surface tangent to $\xi$ and $F\subset M$ a closed oriented surface which is generically embedded and does not intersect $\partial L$. We require that each connected component of $\partial L$ can be connected to $H(\xi)$ by a Legendrian curve which is disjoint from $\ring{L}\cup F$. 

Then there is a smooth confoliation $\xi'$ which is $C^0$-close to $\xi$ such that $F(\xi')$ is homeomorphic to the singular foliation obtained from $F(\xi)$ by thickening the closed leaves of cycles of $F(\xi)$ which are also contained in $L$.    
\end{lem}

\begin{proof}
Let $I=[-1,1]$ and $J=[-1,0]$. We fix a tubular neighbourhood $U\simeq L\times I$ of $L=L\times\{0\}$.

For each boundary component $B_i$ of $L$ we choose $U_i\simeq S^1\times J\times I\subset M$ in the complement of $\ring{L}\cup F$. We assume that the third factor of $U_i$ is transverse to $\xi$ while the foliation $\mathcal{J}$ whose fibers correspond to the second factor is Legendrian and that $S^1\times\{(0,0)\}=B_{0,i}$ and $S^1\times\{(-1,0)\}=B_{-1,i}$ intersect $H(\xi)$. Let $A_{j,i}=S^1\times\{j\}\times I\subset \partial U_i$ for $j\in\{-1,0\}$.

Without loss of generality we may assume that $B_{-1,i}$ is completely contained in the contact region and transverse to $\xi$. Otherwise we apply \lemref{l:approx with control1} along segments of $B_{-1,i}$ and replace $U_i$ with a new set $U_i'$ with the desired property. 

We will now construct a confoliation $\xi'$ on $U\cup \bigcup_i U_i$ which coincides with $\xi$ near $\partial U$ and has the desired properties. 

The restriction of $\xi'$ to $U$ is defined in two steps. First we flatten $\xi$ in a neighbourhood $U\simeq L\times I$ using the push forward of $\xi$ using a smooth homeomorphism $g$ of $I$ which is $C^\infty$-tangent to the zero map and coincides with the identity outside a neighbourhood of $0$. 

We push forward $\xi$ on $L\times[0,1]$ respectively $L\times[-1,0]$ using a diffeomorphism $[0,1]\lra[\eps,1]$ respectively $[-1,-\eps]$. The confoliation on $(L\times [-1,-\eps])\cup (L\times [-\eps,\eps]) \cup (L\times[\eps,1] )\simeq U$ (with $\eps>0$), which is the product foliation on $L\times[-\eps,\eps]$, is smooth and contains a family of compact leaves. Moreover, we can choose the diffeomorphisms appearing in the construction such that $\xi\eing{U}$ is as close to $\xi'\eing{U}$ in the $C^0$-topology as we want.

We can choose $\xi'\eing{U}$ such that $A_{0,i}(\xi)$ and $A_{0,i}(\xi')$ coincide outside of the region where the slope of $A_{0,i}(\xi)$ is very small compared to the slope of $A_{-1,i}(\xi)$.  By construction the slope of $A_{0,i}(\xi')$ is much smaller than the slope of $A_{-1,i}(\xi)=A_{-1,i}(\xi)$. As in the second step in the proof of \lemref{l:reminder on holonomy approx} (or Lemma 2.5.1. of \cite{confol}) one can extend $\xi'$ to a smooth confoliation on $M$ such that $\xi'$ is close to $\xi$ (the foliation $\mathcal{J}$ corresponds to the $y$-coordinate in \cite{confol}). The claim about $F(\xi')$ follows immediately from the construction. 
\end{proof}

\begin{rem} \mlabel{r:force holonomy}
After a trivially foliated bundle $L\times[-\eps,\eps]$ is added to the confoliation, it is possible to replace the trivially foliated piece by a foliation on $L\times[-\eps,\eps]$ such that the boundary leaves $L\times\{\pm\eps\}$ have sometimes attractive holonomy on side lying in $L\times[-\eps,\eps]$. The following statements follow from the construction explained in \cite{confol} on p. 39. (This construction carries over to surfaces with boundary after the surface is doubled.) 

When the Euler characteristic of $L$ is negative, then one can replace the product foliation on $L\times[-\eps,\eps]$ by a foliation such that the holonomy along every homotopically non trivial curve in $L\times\{\eps\}$ or $L\times\{-\eps\}$ is sometimes attractive on one side. 

If the Euler characteristic of the compact surface with boundary $L$ is not negative, then $L$ is diffeomorphic to $S^2,D^2,T^2$ or $S^1\times I$. The case $S^2$ will not occur unless the confoliation in question is actually a product foliation by spheres. But these are excluded. If $L\simeq S^1\times I$, then the suspension of a suitable diffeomorphism yields the same result as in the case of $\chi(L)<0$ (without doubling the surface). The case $L\simeq D^2$ will be excluded by the last requirement of \defref{d:overtwisted star} in the application we have in mind. Finally, the case $L\simeq T^2$ is exceptional because of Kopell's lemma (cf. the footnote on p.~39 of \cite{confol}). But if $L=T^2$, then it is easy to arrange that the holonomy is attractive along a given homotopically non-trivial curve. 

This modification changes the characteristic foliation on $F$, but only an open set which was foliated by closed leaves and cycles before the perturbation. In particular overtwisted stars are not affected. 
\end{rem}

The following proposition from \cite{confol} adapts a famous result of Sacksteder \cite{sacksteder} to laminations so that it can be applied to the fully foliated part of confoliations. 

\begin{prop}[Proposition 1.2.13 in \cite{confol}] \mlabel{p:sacksteder}
Let $(M,\xi)$ be a $C^k$-confoliation, $k\ge 2$. All minimal sets of the fully foliated part of $\xi$ are either closed leaves or exceptional minimal sets. Each exceptional minimal set contains a simple closed curve along which $\xi$ has non-trivial infinitesimal holonomy. 

In particular exceptional minimal sets are isolated and there are only finitely many of them. 
\end{prop}
We denote the finite set consisting of the exceptional minimal sets of the fully foliated part of $\xi$ by $\mathcal{E}(\xi)$.

In the following $F$ will be an embedded surface containing an overtwisted star $(Q,V,\alpha)$. We write $\Omega_Q$ for $\cup_{v\in V}\gamma_v$. If $\gamma_v, v\in V$ is a cycle containing hyperbolic singularities of $F(\xi)$, then the confoliation $\xi$ can be modified such that the cycle has a neighbourhood which is foliated by closed leaves of the characteristic foliation of the modified confoliation (cf. \lemref{l:thickening}). We will therefore assume that $\gamma_v$ is either a closed leaf of $F(\xi)$ or a quasi-minimal set but not a cycle containing hyperbolic singularities. (By the definition of an overtwisted star, $\gamma_v$ is not an elliptic singularity.)   

\begin{lem} \mlabel{l:approx with control2}
Let $\xi$ be a confoliation and $F$ an embedded connected surface containing an overtwisted star $(Q,V,\alpha)$ and $v\in V$.
\begin{itemize}
\item[a)] If $\gamma_v$ is contained in a closed leaf of $\xi$, then in every $C^0$-neighbourhood of $\xi$ there is a confoliation $\xi'$ such that $F(\xi')$ contains an overtwisted star $(Q',V',\alpha')$ which is naturally identified with $(Q,V,\alpha)$ and $\gamma'_v, (v\in V'\simeq V)$ passes through the contact region of $\xi'$. 
\item[b)] Assume that $\gamma$ is contained in an exceptional minimal set, $\gamma$ has attractive linear holonomy, and $\gamma$ is transverse to $F$. Then every $C^0$-neighbourhood of $\xi$ contains a confoliation $\xi'$ such that $F(\xi')$ contains an overtwisted star which can be naturally identified with $(Q,V,\alpha)$ 
and $|\mathcal{E}(\xi')|<|\mathcal{E}(\xi)|$.
\end{itemize} 
\end{lem}

\begin{proof}
First we prove a). Let $L$ be the closed leaf containing $\gamma_v$. Since $\gamma_v$ is the $\omega$-limit set of leaves in $F(\xi)$ it has attractive holonomy on one side and $F\cap L$ consists of a family of cycles. In particular, $L\cap\alpha(Q)=\emptyset$ because an overtwisted star with virtual vertices does not contain closed cycles of the characteristic foliation. 

We use \lemref{l:thickening} and \remref{r:force holonomy} to ensure that $\gamma_v$ has sometimes attractive holonomy on both sides. Unfortunately this property 
is not stable under arbitrary isotopies of $\gamma_v$ in general. But by \lemref{l:neg-curv} there is an annulus $A\simeq \gamma_v\times[0,1]$ such that $\gamma_v=\gamma_v\times\{0\}=F\cap A$ and all curves in $A$ have attracting holonomy on the side where $\alpha(Q)$ approaches $\gamma_v$ while isotopies do not change the nature of the holonomy on the other side of $L$ since there the confoliation is actually a foliation. 

Therefore there is a small isotopy of $F$ which maps $(Q,V,\alpha)$ to an overtwisted star $(Q',V',\alpha')$ on the isotoped surface $F'$ such that $\gamma_v$ is mapped to $\gamma_v\times\{\eps\}$ where $0<\eps<1/2$. Then we can apply \lemref{l:reminder on holonomy approx} to $\gamma_v\times\{0\}$ and $\gamma_v\times\{2\eps\}$.

After this there is a Legendrian arc intersecting $F'$ exactly once in a point of $\gamma_v$ and both endpoints of this arc lie in the contact region. Hence this arc satisfies the assumptions of \lemref{l:approx with control1}. Therefore there is a confoliation $\xi'$ with the desired properties such that $F'(\xi)=F'(\xi')$. This finishes the proof of a).

Now we prove b). We shall use notations from the proof of \lemref{l:reminder on holonomy approx}. In the proof we will use the freedom in the choice of the function $g$ in the proof of \lemref{l:reminder on holonomy approx}. For this we need the fact that $\gamma$ has non-trivial infinitesimal holonomy since then there are only very few restriction on $g$ in the proof of \lemref{l:reminder on holonomy approx}, cf. also Lemma 2.5.2 in \cite{confol}.  

Fix a neighbourhood $U\simeq S^1_x\times[-1,1]_y\times[-1,1]_z$ such that $\gamma=S^1\times\{(0,0)\}$ and the coordinates $x,y,z$ have all the properties used in the proof of \lemref{l:reminder on holonomy approx}. In particular, the foliation by the second factor is Legendrian and coincides with $F(\xi)$ on $F\cap U$ while the third factor is positively transverse to $\xi$. We require that $U$ intersects $F$ only in neighbourhoods of points in $\gamma\cap\Omega_Q=:X$.

Let us make an orientation assumption in order to simplify the presentation: We assume that the orientation of the Legendrian foliation on $S^1\times[-1,1]\times[-1,1]$ given by the second factor coincides with the orientation of $F(\xi)$ near points of $\gamma\cap\gamma_v, v\in V$, ie. in \figref{b:near-tentacle} the foliation is oriented from left to right. When this assumption is not satisfied for some $y\in\gamma\cap\Omega_Q$, then one has to interchange the roles of $\hat{\tau}_-(y)$ and $\hat{\tau}_+(y)$ in some of the following arguments.

By transversality $\gamma$ intersects $F$ in a finite number of points. Since $\gamma$ is contained in the fully foliated part of $\xi$, $\gamma$ cannot intersect $\alpha(Q)$ since every point of $\alpha(Q)$ is connected to $H(\xi)$ by a Legendrian arc. We can ignore the points in $F\cap\gamma$ which do not belong to $\overline{\alpha(Q)}$ if we deform $\xi$ on a neighbourhood of $\gamma$ which is small enough.

Because $F$ is smoothly embedded and $\xi$ is $C^2$-smooth, $F(\xi)$ is also of class $C^2$. As we have already mentioned in \secref{ss:legpoly} the $\omega$-limit set $\gamma_v$ with $v\in V$ is either a quasi-minimal set or we may assume (after a small isotopy of $F$) that $\gamma_v$ is a closed leaf of $F(\xi)$. We distinguish the following cases.
\begin{itemize}
\item[(i)] $\gamma_v$ is quasi-minimal. Since there are interior points of $\alpha(Q)$ arbitrarily close to $\gamma_v$, there is no segment $\tau$ transverse to $F(\xi)$ such that $\tau\cap\gamma_v$ is dense in $\tau$. Then $\gamma_v\cap\tau$ is a Cantor set (cf. \cite{guiterrez}). The intersection between two different quasi-minimal sets cannot contain a recurrent orbit by Maier's theorem (Theorem 2.4.1 in \cite{flows}) and the number of quasi-minimal sets of $F(\xi)$ is bounded by the genus of $F$ according to Theorem 2.4.5. in \cite{flows}. 
\item[(ii)] $\gamma_v$ is a closed leaf of $F(\xi)$ whose holonomy is attractive on the side from which $\alpha(Q)$ accumulates on $\gamma_v$ while it is repulsive on the other side and $\alpha(Q)$ spirals onto $\gamma_v$ on the attractive side. In this case, $\alpha(Q)$ cannot enter a one-sided neighbourhood of $\gamma_v$ on the side where the holonomy is repulsive. 
\item[(iii)] $\gamma_v$ is a closed leaf of $F(\xi)$ whose holonomy is attractive on one side and either there is a sequence of closed leaves of $F(\xi)$ on the other side of $\gamma_v$ which converge to $\gamma_v$ or $\gamma_v$ has attractive holonomy on both sides.
\end{itemize}

If $\gamma_v$ belongs to class (iii) and $U$ is small enough (ie. contained in the interior of an annulus each of whose boundary is tangent to $F(\xi)$ or transverse to $F(\xi)$ such that $F(\xi)$ points into the annulus), then any modification of $F(\xi)$ with support in $U\cap F$ will result in a singular foliation on $F$ such that all leaves of the characteristic foliation which enter a neighbourhood of $\gamma_v$ containing $U$ will remain in $U$ forever even after the modification. When no singularities are created during the modification, then the modification replaces $(Q,V,\alpha)$ by an overtwisted star $(Q',V',\alpha')$ such that $|V|=|V'|$. In this case $\gamma_{v}\neq\gamma'_{v}$ but $\gamma_v'$ is a closed leaf of $F(\xi')$ which passes through $H(\xi')$ (by the proof of \lemref{l:reminder on holonomy approx}. We keep this case separated from the others although all three of them may occur in one single perturbation of $\xi$. 

The following argument is complicated due to a difficulty in case (ii). If $\alpha(Q)$ accumulates on $\gamma_v$ and the holonomy of $\gamma_v$ is repulsive on the side where points of $\gamma$ are pushed to by the diffeomorphism $G$ appearing in the proof of \lemref{l:reminder on holonomy approx}, then it is impossible to say something about the new $\omega$-limit set of leaves in $\alpha(Q)$ which accumulated on $\gamma_v$ unless $G$ is chosen carefully: It is possible that leaves which accumulated on $\gamma_v$ accumulate on $\gamma_{v'}$ when the characteristic foliation is modified near $\gamma_{v}$. However it is possible that $\gamma_{v'}$ is also changed when $\xi$ is replaced by $\xi'$. Therefore one has to treat all $v\in V$ such that $\gamma_v$ belongs to (i),(ii) simultaneously. 

For non-empty open intervals $\tau_-\subset[-1,0)$ and $\tau_+\subset(0,1]$ we write $\hat{\tau}_\pm(y):=\{y\}\times[-1,1]\times\overline{\tau_\pm}$ for $y\in\gamma$. We will fix $\tau_\pm$ in the following.

We require that $\tau_+$ is chosen such that the $\omega$-limit 
of a leaf intersecting $\hat{\tau}_+(y)$ is never a hyperbolic singularity for all $y\in X$. Because 
\begin{itemize}
\item there are only finitely many hyperbolic singularities on $F$ and 
\item $\alpha(Q)$ intersects every interval transverse to $\gamma_v$ in an open set (note that there are singular folioations on surfaces with dense quasiminimal sets; in particular stable leaves of hyperbolic singularities in such quasi-minimal sets may be dense in the surface)
\item $\alpha(\partial Q)$ is disjoint from $\gamma_v$ which intersect $\gamma$ even if $\gamma_v$ is quasi-minimal (this is true because every point of $\alpha(Q)$ is connected to $H(\xi)$ by a Legendrian curve while $\gamma$ is part of the fully foliated set)

\end{itemize}
this condition can be satisfied.
Next we impose additional restrictions on $\tau_-$: 

We choose $\tau_-$ such that no point in $\hat{\tau}_+(x), x\in X,$ is connected to $\hat{\tau}_-(y),y\in X,$ by a leaf of $F(\xi)$ which is disjoint from $\{(y,0)\}\times[\inf(\tau_-),\sup(\tau_+)]$. In other words, we require that leaves of $F(\xi)$ which come from $\hat{\tau}_+(x)$ do not intersect $\hat{\tau}_-(y)$ when they meet the piece of $\{(y,-1)\}\times[-1,1]\subset (U\cap F)$ which lies between the lower endpoint of $\hat{\tau}_-(y)$ and the upper endpoint of $\hat{\tau}_+(y)$ for the first time. In order to satisfy this condition it might be necessary to shorten $\tau_+$.  

Obviously there is a choice for $\tau_+,\tau_-$ which satisfies these requirements for $x,y\in X$ whenever the limit set $\gamma_v$ which corresponds to $y$ is not the $\omega$-limit set of leaves intersecting $\hat{\tau}_+(x)$. 

If $y$ is contained in a closed leaf of $F(\xi)$, then one can also satisfy the requirement for $x,y\in X$ provided that $\tau_+$ is so short that the translates of $\hat{\tau}_+(x)$ along leaves of $F(\xi)$ do not cover the segment $\hat{\tau}_-(y)$). We shorten $\tau_+$ whenever this is necessary. Finally, when $y$ is part of a quasi-minimal set and the leaves of $F(\xi)$ which intersect $\hat{\tau}_+(x)$ accumulate on this quasi-minimal set the above requirement can be satisfied by shortening $\tau_\pm$ again. Now one can construct $\tau_-$ in a finite number of steps and shortening $\tau_\pm$ at each step. 

Let $t_-\in\tau_-$. We fix the diffeomorphism $g:[-1,1]\lra[-1,1]$ in the proof of \lemref{l:reminder on holonomy approx} such that $g$ maps the entire interval $(t_-,\sup(\tau_+))$ into $\tau_+$ and the support of $g$ is contained in $(\inf(\tau_-),\sup(\tau_+))$. 
The role of the parameters $\zeta,\zeta'$ from the proof of \lemref{l:reminder on holonomy approx} is now played by $\sup(\tau_+),\inf(\tau_-)$.

If $\xi$ is modified by the procedure described in the proof of \lemref{l:reminder on holonomy approx} using the diffeomorphism $g$ chosen above, then one obtains a confoliation $\xi'$ such that all leaves of $F(\xi')$ starting at the elliptic singularity in the center of the original overtwisted whose $\omega$-limit set was $\gamma_v$ such that $\gamma_v\cap\gamma\neq\emptyset$ never meet a hyperbolic singularity of $F(\xi')$. 

Since all elliptic singularities on the boundary of the basin of the elleiptic singularity in $\alpha(Q)$ are automatically negative and all hyperbolic singularities on the boundary of the basin where already present in $\alpha(\partial Q)$ there is an overtwisted star $(Q',V',\alpha')$ and $V'$ can be viewed as a subset of $V$ by construction. Moreover, $|\mathcal{E}(\xi')|<|\mathcal{E}(\xi)|$.
\end{proof}

Now we can finally show that there are no overtwisted stars when $\xi$ is symplectically fillable.  

\begin{thm} \mlabel{t:no polygons if filled}
Let $(M,\xi)$ be a $C^k$-confoliation, $k\ge 2$, which is symplectically fillable. Then no oriented embedded surface contains an overtwisted star.  
\end{thm}

\begin{proof}
Let $(X,\omega)$ be a symplectic filling of $\xi$. Assume that $F$ is an embedded surface containing an overtwisted star $(Q,V,\alpha)$. It is sufficient to treat only the case of closed surfaces when the elliptic singularity in the interior of $\alpha(Q)$ is positive. 

In the first part of the proof we show how to reduce the number of virtual vertices. Because overtwisted stars are not required to be injective as Legendrian polygons, we show in a second step how to obtain an embedded disc violating \defref{d:tight confol} starting from an overtwisted star $(Q,\emptyset,\alpha)$. The confoliation is modified several times but all confoliations appearing in the proof will be $C^0$-close to $\xi$. In particular they are symplectically fillable. Therefore the assumption that $(M,\xi)$ admits an overtwisted star leads to a contradiction to \thmref{t:fillable confol are tight}.

Notice that in the presence of an overtwisted star $\xi$ cannot be a foliation everywhere. Therefore $M$ is not a minimal set of the fully foliated part of $\xi$ and $\xi$ is not a foliation without holonomy. 

{\em Step 1:} If $V\neq\emptyset$, then $\xi$ can be approximated by a confoliation which admits an overtwisted star with less virtual vertices than $(Q,V,\alpha)$. We fix $v_0\in V$. If $\gamma_0:=\gamma_{v_0}$ intersects $H(\xi)$, then an application of \lemref{l:create} yields a surface carrying an overtwisted star with less virtual vertices after a $C^0$-small isotopy of $F$. Now assume $\gamma_0\cap H(\xi)=\emptyset$. 

Let $L$ be the maximal connected open immersed hypersurface of $M$ which is tangent to $\xi$ and contains $\gamma_0$. If $L=\emptyset$, then there is a Legendrian segment $\sigma$ satisfying the hypothesis of \lemref{l:approx with control1}. After applying this lemma, $\gamma_v$ intersects the contact region of the modified confoliation and we are done.

Now assume $L\neq\emptyset$ and let $L^\infty$ be the space of ends of $L$. We say that an end $e\in L^\infty$ lies in $H(\xi)$ if for every compact set $K\subset L$ there is a Legendrian curve from $H(\xi)$ to the connected component of $L\setminus K$ corresponding to $e$.  

{\em Step 1a:} If $L^\infty\neq\emptyset$, then we approximate $\xi$ such that all ends of $L$ lies in the contact region of the modified confoliation. 
  
The set of ends in $H(\xi)$ is open in $L^\infty$, therefore its complement $L^\infty_{fol}$ is compact. To each $e\in L^{\infty}_{fol}$ we associate a minimal set $\mathcal{M}(e)\subset\lim_e L$ of the fully foliated part of $\xi$ (this is explained in \cite{cc}, p. 115). Recall that $M$ cannot be a minimal set of the fully foliated part of $\xi$. According to \cite{hh}, p.19, all minimal sets are either closed leaves or exceptional minimal sets. Note that we allow that $L$ is contained in $\mathcal{M}(e)$.

If $\mathcal{M}(e)$ is a closed leaf of $\xi$ whose holonomy along a curve $\gamma$ transverse to $F$ is sometimes attractive, then we can apply \lemref{l:approx with control2} (a) to $\gamma_v$ if there is $v\in V$ with $\gamma_v\subset \mathcal{M}(e)$. If $L$ contains no limit set of $\alpha(Q)$, then the procedure from the proof of \lemref{l:reminder on holonomy approx} can be applied directly to any curve $\gamma\subset \mathcal{M}(e)$ with sometimes attractive holonomy. We can ensure the existence of such a curve by \lemref{l:thickening} and \remref{r:force holonomy}.

If $\mathcal{M}(e)$ is an exceptional minimal set, then according to \propref{p:sacksteder} there is a simple closed curve $\gamma$ in a leaf $L_\gamma\subset \mathcal{M}(e)$ with non-trivial infinitesimal holonomy. Every curve in $L_\gamma$ which is isotopic to $\gamma$ through Legendrian curves has the same property by Lemma 1.3.17 in \cite{confol}. In particular we may assume that $\gamma$ is transverse to $F$. 

Using \lemref{l:approx with control2} (b) we approximate $\xi$ by a confoliation $\xi'$ such that $L_\gamma$ meets $H(\xi')$. 

If $\mathcal{M}(e)$ was an exceptional minimal set, this process might have changed the overtwisted star in the sense that type of the $\omega$-limit sets of virtual vertices may have changed. But recall that by the proof of \lemref{l:approx with control2} we can view $V'$ as a subset of $V$. We use $\gamma'_v$ to denote the $\omega$-limit set of leaves which start at the elliptic singularity in the center of the overtwisted star and accumulated on $\gamma_v, v\in V$ before the modification. 

We iterate the procedure from the very beginning with $v_0\in V'$ and with an integral surface of $\xi'$ containing $\gamma'_0$. Since $\mathcal{E}(\xi)$ is finite and $|\mathcal{E}(\xi')|<|\mathcal{E}(\xi)|$ this phenomenon can occur only finitely many times.

After finitely many steps no exceptional minimal sets will occur in the above procedure. In later applications of the above construction $\gamma_0'=\gamma_0$ and the maximal integral surface of $\xi'$ containing $\gamma_0'$ is contained in the maximal integral surface of $\xi$ containing $\gamma_0$. Because the inclusion induces a continuous mapping between the spaces of ends and by the compactness of $L^\infty_{fol}$ we are done after finitely many steps. We continue to write $F$ for the embedded surface, $\xi$ for the confoliation, and $(Q,V,\alpha)$ for the overtwisted star etc.

{\em Step 1b:} We isotope $F$ such that all quasi-minimal sets of the characteristic foliation on the resulting surface pass through the contact region. As we have already noted in the proof of \lemref{l:approx with control2}, $F(\xi)$ has only finitely many quasi-minimal sets (this number is bounded by the genus of $F$). Let $\gamma_w,w\in V$ be a quasi-minimal set of $F(\xi)$ which is disjoint from $H(\xi)$. 
 
According to Theorem 2.3.3 in \cite{flows} there is an uncountable number of leaves of $F(\xi)$ which are recurrent (in both directions) and dense in $\gamma_w$ while there is only a finite number of pseudovertices of $(Q,V,\alpha)$ and only finite number of virtual vertices. Therefore there is $p_w\in\gamma_w$ which can be connected to $H(\xi)$ by a Legendrian arc $\sigma$ transverse to $F$ such that $\sigma$ does not meet $\alpha(\partial Q)$ and $\sigma$ never intersects closed components of $\Omega_Q$. At this point we use the fact that every end of the union of integral hypersurfaces containing $\gamma_w$ lies in $H(\xi)$. If $\sigma$ intersects $\Omega_Q$ in some other quasi-minimal set $\gamma_{w'}, w'\in V$ before it meets $H(\xi)$, then we replace $\gamma_w$ by $\gamma_{w'}$. Thus we may assume that $\sigma$ meets $F$ in $p_w$ and nowhere else.  

By Lemma 2.8.2 in \cite{confol} there is a confoliation $\xi'$ $C^k$-close to $\xi$ such that $F(\xi')=F(\xi)$, $\sigma$ is tangent to $\xi$ and $\xi'$ and a neighbourhood of $p_w$ in $F$ lies in $\overline{H(\xi')}$. We will denote $\xi'$ again by $\xi$. 

Choose a neighbourhood $U\simeq\sigma\times[-1,1]\times[-1,1]$ of $\sigma$ such that $\sigma=\sigma\times\{(0,0)\}$ and $(\{p_w\}\times[-1,1]\times[-1,1])\subset F$. Moreover, we require that the foliation by the first factor is Legendrian while the foliation corresponding to the second factor is transverse to $\xi$ and $\ring{U}\subset H(\xi)$. Finally we assume that the foliation which corresponds to the second factor is Legendrian when it is restricted to $F$. 

Now we apply an isotopy to $F$ whose effect on the characteristic foliation on $F$ is the same as the effect of the map $G$ appearing in the proof of \lemref{l:reminder on holonomy approx}. We explain this under the following orientation assumptions (the other cases can be treated in the same way):

The orientation of $F(\xi)$ coincides with the second factor of $U\simeq\sigma\times[-1,1]\times[-1,1]$ and the coorientation of $F$ points away from $U$.  In \figref{b:near-tentacle} the left respectively right edge of the rectangle corresponds to $\{(p_w,-1)\}\times[-1,1]$ respectively $\{(p_w,1)\}\times[-1,1]$, the foliation is oriented from left to right, the coorientation of $\xi$ points upwards and the coorientation of $F$ points towards the reader. 

Choose $-1<x<0<y<1$ such that the points $(p_w,-1,x), (p_w,1,y)\in F$ 
\begin{itemize}
\item[(i)] do not lie on a stable or unstable leaf of a hyperbolic singularity and they are not connected by a leaf of $F(\xi)$. 
\item[(ii)] can be connected by a smooth Legendrian arc $\lambda$ in $U$ whose projection to $\sigma\times[-1,1]$ is embedded and $\lambda$ is $C^\infty$ tangent to $F$. Moreover, we assume that the projection of $\lambda$ to $\sigma\times[-1,1]$ is transverse to the first factor.  
\end{itemize}
The curve $\lambda$ and $x,y$ exist because of the orientation assumptions and \lemref{l:neg-curv}. Now fix $x',y'$ close to $x,y$ such that $x<x'<0<y'<y$. 

Using a flow along the first factor of $U$ we can move $\{p_w\}\times[-1,1]$ to a curve which is close to the projection of $\lambda$ to $\sigma\times[-1,1]$. When we apply this flow to $F$, the surface is pulled into $U$ and we obtain a surface $F'$ isotopic to $F$ which coincides with $F$ outside of $\{p_w\}\times(-1,1)\times(x,y)$.

By the assumptions on $\lambda$ we can choose $F'$ such that $F'((xi)$ compresses the transverse segment $\{(p_w,-1)\} \times (x',y)$ onto $\{(p_w,1)\}\times(y',y)$ such that no leaf of intersecting $\{(p_w,1)\}\times(y',y)$ is part of a stable or unstable leaf of $F(\xi)$. Moreover, we may assume that leaves which start at points of $\{(p_w,1)\}\times(y',y)$ meet the segment $\{(p_w,-1)\}\times[x',y]$ before the enter the region where $F'\neq F$ for the first time.           
The new $\omega$-limit set is now a closed leaf of $F'(\xi)$ which passes through $\{(p_w,1)\}\times(y',y)$.

This modification may have created quasi-minimal sets on $F'$ which were not present in $F(\xi)$. But if this happens, then the new quasi-minimal sets intersect the contact region by construction. Thus after finitely many steps (this number is bounded by the genus of $F$) we have isotoped $F$ such that all quasi-minimal sets of the characteristic foliation on the resulting surface pass through the contact region. Now we apply \lemref{l:create}. We obtain a surface $F''$ containing an overtwisted star $(Q'',V'',\alpha'')$ such that there is a natural inclusion $V''\subset V$ and all $\gamma_v, v\in V''$  are cycles of $F''(\xi)$. In the next step we treat the remaining virtual vertices. We will denote $F''$ by $F$, $Q''$ by $Q$, etc.

{\em Step 1c:} Let $\gamma_0$ be the limit set which corresponds to the virtual vertex $v_0\in V$ of an overtwisted star $(Q,V,\alpha)$. We assume that $\gamma_v$ is a cycle for all $v\in V$ and
all ends of the maximal integral surface $L_0$ containing $\gamma_0$ lie in the contact region. 

Choose a submanifold $L_0'\subset L_0$ of dimension $2$ such that $L_0'$ contains all closed components of $\Omega_Q\cap L_0$. Since each end of $L_0$ lies in $H(\xi)$ we can choose $L_0'$ so that each boundary component is connected to $H(\xi)$ by a Legendrian curve which does not intersect the interior of $L_0'$. After a $C^\infty$-small perturbation (we use again Lemma 2.8.1 from \cite{confol}) of $\xi$ we may assume that the boundary of $L_0'$ is contained in the contact region of the resulting confoliation $\xi'$. This perturbation might affect the characteristic foliation on $F$, but since the modification of the confoliation does not affect $\Omega_Q$ and all components of $\Omega_Q$ are cycles of $F(\xi)$ which are also present in $F(\xi')$, there still is an overtwisted star $(Q',V'\alpha')$ on $F$ together with a natural inclusion $V'\hookrightarrow V$. 

Now we can apply \lemref{l:thickening} and \remref{r:force holonomy}. From \lemref{l:approx with control2} a) we obtain a confoliation $\xi''$ which is $C^0$-close to $\xi'$ such that $F(\xi'')$ contains an overtwisted star $(Q'',V'',\alpha'')$ with $V''\subset V'$ and all $\omega$-limit sets $\gamma_w'', w\in V''$ which were contained in $L_0$ now intersect the contact region of $\xi''$. After an application of \lemref{l:create} we can reduce the number of virtual vertices. 

{\em Step 2:} We show that we can assume that the map $\alpha$ associated to the overtwisted star $(Q,\emptyset,\alpha)$ in $F$ is injective.

Assume that the Legendrian polygon $(Q,\emptyset,\alpha)$ is not injective. Then there are two edges $e_1,e_2$ of $Q$ such that $\alpha(e_1)=\alpha(e_2)$. (Recall that by our genericity assumption no two different hyperbolic singularities of $F(\xi)$ are connected by leaves. Therefore configurations like the one shown in \figref{b:identify} cannot appear.) 

Let $y$ be the image of the pseudovertex on $e_1$ by the map $\alpha$. Then $y$ is a negative hyperbolic singularity of $F(\xi)$. The $\omega$-limit sets of the stable leaves of $y$ are negative elliptic singularities $y_1,y_2$ in $\alpha(\partial Q)$ and we may assume that these singularities are contained in $H(\xi)$ (because they are $\omega$-limit sets, they do not lie in the interior of the foliated part of $\xi$). 

We eliminate $y_1$ and $y$ using \lemref{l:elim}. This reduces the number of edges of the polygon which are identified unless $y_1=y_2$. The case when $y_1=y_2$ requires slightly more work:

After perturbing the surface on a neighbourhood of $y_1$ we may assume that the two unstable leaves of $y$ form a smooth closed Legendrian curve $\gamma'$. We eliminate $y_1,y$ such that $\gamma'$ is a closed leaf of the characteristic foliation on the resulting surface. We obtained a Legendrian polygon $(Q',V',\alpha')$ on a surface $F'$ with $Q'\simeq D^2$ and $V'$ consists of all vertices of $Q'$ which were mapped to $y_1$ by $\alpha'$. By construction $\gamma_{v'}=\gamma'$ for all $v'\in V'$. 

Since $y_1\in H(\xi')$ we can approximate $\xi'$ by a confoliation $\xi''$ which coincides with $\xi'$ outside a tubular neighbourhood of $\gamma'$ and is a contact structure near $\gamma'$. This can be done without changing the characteristic foliation on the surface by \lemref{l:approx with control1}.

Next we apply a standard procedure from contact topology called folding to $\gamma'$. This is described in \cite{honda} (on p. 325). We obtain a surface $F''$ which contains an overtwisted star $(Q'',V'',\alpha')$ such that $V'$ consists of two elements with $Q''\simeq Q'$, $V''=V'$ but now elements of $V''$ correspond to different $\omega$-limit sets depending on which side of $\gamma'$ the corresponding leaves of $\alpha(Q)$ accumulated. 

In order to continue we create a pair of negative singularities along the closed leaves in $\overline{\alpha''(Q'')}$. We  eliminate all pseudovertices successively and we obtain a confoliation $\widetilde{\xi}$ on $M$ together with an overtwisted star $(\widetilde{Q},\widetilde{V}=\emptyset,\widetilde{\alpha})$ on a surface $\widetilde{F}$ which has no virtual vertices and is injective as a Legendrian polygon. $\widetilde{\alpha}$ becomes injective after finitely many perturbations of $\widetilde{F}$ as in \figref{b:split}. 

Because $\widetilde{\alpha}(\partial \widetilde{Q})$ passes through the contact region of $\widetilde{\xi'}$ the disc $D=\widetilde{\alpha}(\widetilde{Q})$ violates \defref{d:tight confol}. This concludes the proof of the theorem. 
\end{proof} 

This proof can be modified to yield a proof of \thmref{t:fillable confol are tight} using the well known fact that symplectically fillable contact structures are tight and without referring to results of R.~Hind in \cite{hind} which are used in \cite{confol}. Let us outline the argument.

Given a disc $D$ as in \defref{d:tight confol} assume first that the holonomy of $\partial D$ in $D$ is non-trivial. We try follow the construction above to find a confoliation $\xi'$ such that $\partial D$ remains Legendrian and $\xi'$ is $C^0$-close to $\xi$. This attempt must fail since otherwise we could continue to modify $\xi'$ into a symplectically fillable contact structure such that $D$ becomes an overtwisted disc. This contradicts the fact that symplectically fillable contact structures are tight. 

The only point at which the above construction can break down is the application of \remref{r:force holonomy} in the case when $\partial D$ bounds a disc $D'$ in the maximal surface which contains $\partial D$ and is tangent to the confoliation. In order to show that $e(\xi)[D\cup D']=0$ one chooses an embedded sphere $S$ close (and homologous) to $D\cup D'$. Then $e(\xi)[S]=0$ follows from the tightness contact structures which are $C^0$-close to the original one.

It remains to treat the case when the holonomy of $\partial D$ in $D$ is trivial. Then one has to show that either $\partial D$ is a vanishing cycle (cf. Chapter 9 in \cite{cc2}) or one can replace $D$ by a smaller disc which has Legendrian boundary along which the holonomy of the characteristic foliation on the disc is not trivial. If $\partial D$ is a vanishing cycle, then one uses results due to S.~Novikov \cite{No} to establish the existence of a solid torus whose boundary $T$ is a leaf of the confoliation. This contradicts $\int_T\omega>0$ because this inequality means that $T$ represents a non-trivial homology class. 





\begin{thebibliography}{12345}

\bibitem{Aeb} B.~Aebischer, M.~Borer, M.~K{\"a}lin, Ch.~Leuenberger, H.~Reimann, {\sl Symplectic geometry}, Birkh{\"a}user 1994.

\bibitem{AlW} S.~Altschuler, L.~Wu, {\em On deforming confoliations}, J. of Diff. Geom. 54 (2000), no. 1, 75--97.

\bibitem{Be} D.~Bennequin, {\em Entrelacements et equations de Pfaff}, Ast{\'e}risque 107-108 (1983), 83--161.

\bibitem{cc} A.~Candel, L.~Conlon, {\sl Foliations I}, Grad. Studies in Math. Vol. 23, Amer. Math. Soc. 2000.

\bibitem{cc2} A.~Candel, L.~Conlon, {\sl Foliations II}, Grad. Studies in Math Vol. 60, Amer. Math. Soc. 2003.

\bibitem{colin} V.~Colin, {\em Chirurgies d'indice un et isotopies de sph{\`e}res dans les vari{\'e}t{\'e}s de contact tendues}, C.R. Acad. Sci. Paris, S{\'e}r. I Math. 324 (1997), no. 6, 659--663.

\bibitem{colin pert} V.~Colin, {\em Recollement de vari{\'e}t{\'e}s de contact tendues}, Bull. Soc. Math. France 127 (1999), 43--69.

\bibitem{El} Y.~Eliashberg, {\em Contact $3$-manifolds twenty years since J.~Martinet's work}, Ann. Inst. Fourier 42, no. 1--2 (1992), 165--192.

\bibitem{confol} Y.~Eliashberg, W.~Thurston, {\sl Confoliations}, University Lecture Series Vol. 13, AMS 1997.

\bibitem{intro-etnyre} J.~Etnyre, {\em Introductory Lectures on Contact Geometry}, Proc. Sympos. Pure Math. 71 (2003), 81--107.

\bibitem{etnyre-notes} J.~Etnyre, {\em Lectures on contact geometry in low-dimensional topology}, http://arxiv.org/abs/math/0610798.

\bibitem{cont from fol} J.~Etnyre, {\em Contact structures on $3$-manifolds are deformations of foliations}, Math.~Res.~Letters Vol. 14, Issue 5 (2007), 775--779. 

\bibitem{etnyreCk} J.~Etnyre, {\em Approximation of foliations by contact structures}, in preparation (the title is preliminary). 

\bibitem{giroux} E.~Giroux, {\em Convexit{\'e} en topologie de contact}, Comm. Math. Helv. 66 (1991), 637--677.

\bibitem{giroux2} E.~Giroux, {\em Structures de contact en dimension trois et bifurcations des feuilletages de surfaces}, Invent.~Math. 141 (2000) no. 3, 615--689. 

\bibitem{guiterrez} C.~Guiterrez, {\em Smoothing continuous flows on two-manifolds and recurrences}, Ergod. Theory \& Dynam. Sys. 6 (1986), 17--44.


\bibitem{hh} G.~Hector, U.~Hirsch, {\sl Introduction to the geometry of foliations - Part B}, Vieweg 1983. 

\bibitem{hind} R.~Hind, {\em Filling by holomorphic disks with weakly pseudoconvex boundary conditions}, Geometry and Functional analysis 7 (1997), 462--495. 

\bibitem{honda} K.~Honda, {\em On the classification of tight contact structures I}, Geom. \& Topology, Vol. 4 (2000), 309--368.


\bibitem{katok} A.~Katok, B.~Hasselblatt, {\sl Introduction to the modern theory of dynamical systems}, Cambridge Univ. Press 2005.

\bibitem{marsden} J.~E.~Marsden, M.~McCracken, {\sl The Hopf bifurcation and its applications}, Springer 1984.

\bibitem{mori} A.~Mori, {\em A note on Thurston-Winkelnkemper's construction of contact forms on $3$-manifolds}, Osaka. J. of Math. 39 (2002), 1--11.

\bibitem{flows} I.~Nikolaev, E.~Zhuzhoma, {\sl Flows on $2$-dimensional manifolds -- An overview}, Lect. Notes in Math. 1705, Springer 1999.

\bibitem{No} S.~Novikov, {\em Topology of foliations}, Trans. of the Moscow Math. Soc. 14 (1965), 268--305.


\bibitem{pet} C.~Petronio, {\sl A theorem of Eliashberg and Thurston on foliations and contact structures}, Scuola Normale Superiore, Pisa, 1997.

\bibitem{rous} R.~Roussarie, {\em Plongements dans les vari{\'e}t{\'e}s feuillet{\'e}es et classification de feuilletages sans holonomie}, Publ. Math. IHES 43 (1973), 101--142.

\bibitem{sacksteder} R.~Sacksteder, {\em Foliations and pseudogroups}, Amer. J. Math. 87 (1965), 79--102.

\bibitem{Th} W.~Thurston, {\em Norm on the homology of $3$-manifolds}, Memoirs of the AMS 339 (1986), 99--130.

\end{thebibliography}
\end{document}